\documentclass[11pt]{article}
\usepackage{amsmath,amssymb,amsthm,fancyhdr,graphicx,pstricks,comment,faktor}
\usepackage[all]{xy}
\title{A complex hyperbolic Riley slice}
\author{John R PARKER \\
Department of Mathematical Sciences\\ 
Durham University\\ Durham DH1 3LE\\ 
England
\and Pierre WILL \\
Universit\'e Grenoble Alpes\\
Institut Fourier, B.P.74\\
38402 Saint-Martin-d'H\`eres Cedex\\
France}
\begin{document}
\maketitle

\def\og{\leavevmode\raise.3ex\hbox{$\scriptscriptstyle\langle\!\langle$~}}
\def\fg{\leavevmode\raise.3ex\hbox{~$\!\scriptscriptstyle\,\rangle\!\rangle$}}
\newcommand{\la}{\langle}
\newcommand{\ra}{\rangle}
\newcommand{\HdC}{\mathbf H^{2}_{\mathbb C}}
\newcommand{\HdR}{\mathbf H^{2}_{\mathbb R}}
\newcommand{\HtR}{\mathbf H^{3}_{\mathbb R}}
\newcommand{\HuC}{\mathbf H^{1}_{\mathbb C}}
\newcommand{\HnC}{\mathbf H^{n}_{\mathbb C}}
\newcommand{\Cdu}{\mathbb C^{2,1}}
\newcommand{\Ct}{\mathbb C^{3}}
\newcommand{\C}{\mathbb C}
\newcommand{\Mn}{\mbox{M}_n\left(\mathbb C\right)}
\newcommand{\R}{\mathbb R}
\newcommand{\A}{\mathbb A}
\newcommand{\cC}{\mathfrak C}
\newcommand{\cCp}{\mathfrak C_{\mbox{p}}}
\newcommand{\cT}{\mathfrak T}
\newcommand{\rR}{\mathfrak R}
\newcommand{\hH}{\mathfrak H}
\newcommand{\n}{\noindent}
\newcommand{\p}{{\bf p}}
\newcommand{\lox}{\mbox{\scriptsize{lox}}}
\newcommand{\Td}{T^{2}_{\left(1,1\right)}}
\newcommand{\T}{T_{\left(1,1\right)}}
\newcommand{\PSL}{\mbox{\rm PSL(2,$\R$)}}
\newcommand{\Pu}{\mbox{\rm PU(2,1)}}
\newcommand{\X}{\mbox{\textbf{X}}}
\newcommand{\tr}{\mbox{\rm tr}}
\newcommand{\B}{\mathcal{B}}
\newcommand{\bx}{{\boxtimes}}
\newcommand{\bp}{{\bf p}}
\newcommand{\bm}{{\bf m}}
\newcommand{\bv}{{\bf v}}
\newcommand{\bA}{{\bf A}}
\newcommand{\bB}{{\bf B}}
\newcommand{\bc}{{\bf c}}
\newcommand{\ba}{{\bf a}}
\newcommand{\bd}{{\bf d}}
\newcommand{\bb}{{\bf b}}
\newcommand{\bn}{{\bf n}}
\newcommand{\bq}{{\bf q}}
\newcommand{\bz}{{\bf z}}
\newcommand{\td}{{\tt d}}

\newcommand{\tD}{{\tt D}}
\newcommand{\I}{\mathcal{I}}
\newcommand{\Z}{{\mathbb Z}}
\newcommand{\Isom}{\mathcal{I}}
\renewcommand{\arg}{\mbox{arg}}
\renewcommand{\Re}{\mbox{\rm Re}\,}
\renewcommand{\Im}{\mbox{\rm Im}\,}
\renewcommand{\leq}{\leqslant}
\renewcommand{\geq}{\geqslant} 
\newtheorem{theo}{Theorem}[section]
\newtheorem*{theo*}{Theorem}
\newtheorem{lem}[theo]{Lemma}
\newtheorem{coro}[theo]{Corollary}
\newtheorem{prop}[theo]{Proposition}
\newtheorem{conj}[theo]{Conjecture}
\theoremstyle{remark}
\newtheorem{rem}{Remark}
\newtheorem{ex}{Example}
\theoremstyle{definition}
\newtheorem{defi}{Definition}

\begin{abstract}
\noindent
We study subgroups of ${\rm PU}(2,1)$ generated by two non-commuting 
unipotent maps $A$ and $B$ whose product $AB$ is also unipotent. 
We call $\mathcal{U}$ the set of conjugacy classes of such groups. We provide a 
set of coordinates on $\mathcal{U}$ that make it homeomorphic to $\R^2$ . 
By considering the action on complex hyperbolic space $\HdC$ of groups in 
$\mathcal{U}$, we describe a two dimensional disc ${\mathcal Z}$ 
in $\mathcal{U}$ that 
parametrises a family of discrete groups. As a corollary, we give a proof of a 
conjecture of Schwartz for $(3,3,\infty)$-triangle groups. We 
also consider a particular group on the boundary of the disc ${\mathcal Z}$ where 
the commutator $[A,B]$ is also unipotent. We show that the boundary of the 
quotient orbifold associated to the latter group gives a spherical CR uniformisation 
of the Whitehead link complement.
\end{abstract}
{\small AMS classification 51M10, 32M15, 22E40}

\section{Introduction}

\subsection{Context and motivation}

The framework of this article is the study of the deformations of a discrete 
subgroup $\Gamma$ of a Lie group $H$ in a Lie 
group $G$ containing $H$. This question has been addressed in many different 
contexts. A classical example is the one where $\Gamma$ is a Fuchsian group,
$H={\rm PSL}(2,\R)$ and $G={\rm PSL}(2,\C)$. When $\Gamma$ is discrete,
such deformations are called quasi-Fuchsian. We will be interested in the case
where $\Gamma$ is a discrete subgroup of $H={\rm SO}(2,1)$ and $G$ is the 
group ${\rm SU}(2,1)$ (or their natural projectivisations over $\R$ and $\C$ 
respectively).
The geometrical motivation is very similar: In the classical case mentioned 
above, ${\rm PSL}(2,\C)$ is the orientation preserving isometry group of
hyperbolic 3-space ${\bf H}^3$ and a Fuchsian group preserves a totally
geodesic hyperbolic plane ${\bf H}^2$ in ${\bf H}^3$. In our case 
$G={\rm SU}(2,1)$ is (a triple cover of) the holomorphic isometry group 
of complex hyperbolic 2-space $\HdC$, and the subgroup $H={\rm SO}(2,1)$ 
preserves a totally geodesic Lagrangian plane isometric to ${\bf H}^2$.
A discrete subgroup $\Gamma$ of ${\rm SO}(2,1)$ is called $\R$-Fuchsian.
A second example of this construction is where $G$ is again ${\rm SU}(2,1)$
but now $H={\rm S}\bigl({\rm U}(1)\times{\rm U}(1,1)\bigr)$. In this case $H$
preserves a totally geodesic complex line in $\HdC$. A discrete subgroup of
$H$ is called $\C$-Fuchsian. Deformations of either $\R$-Fuchsian or 
$\C$-Fuchsian groups in ${\rm SU}(2,1)$ are called complex hyperbolic
quasi-Fuchsian. See \cite{PP2} for a survey of this topic.

The title of this article refers to the so-called \textit{Riley slice of Schottky space} 
(see \cite{KS} or \cite{ASWY}). Riley considered the space of conjugacy classes of 
subgroups of ${\rm PSL}(2,\C)$ generated by two non-commuting parabolic maps.
This space may be identified with $\C-\{0\}$ under the map that associates
the parameter $\rho\in\C-\{0\}$ with the conjugacy class of the group
$\Gamma_\rho$, where 
$$
\Gamma_\rho=\left\langle
\left[\begin{matrix} 1 & 1 \\ 0 & 1 \end{matrix}\right],\ 
\left[\begin{matrix} 1 & 0 \\ \rho & 1 \end{matrix}\right]\right\rangle.
$$
Riley was interested in the set of those parameters $\rho$ for which $\Gamma_\rho$
is discrete. He was particularly interested in the (closed) set where $\Gamma_\rho$
is discrete and free, which is now called the Riley slice of Schottky space \cite{KS}. 
This work has been taken up more recently
by Akiyoshi, Sakuma, Wada and Yamashita. In their book \cite{ASWY}
they illustrate one of Riley's original computer 
pictures\footnote{JRP has one of Riley's printouts of this picture 
dated 26th March 1979}, Figure 0.2a, and their version of this picture, 
Figure 0.2b. Riley's main method was to construct the Ford domain for $\Gamma_\rho$.
The different combinatorial patterns that arise in this Ford domain correspond
to the differently coloured regions in these figures from \cite{ASWY}.
Riley was also interested in groups $\Gamma_\rho$ that are discrete but not free. In 
particular, he showed that when $\rho$ is a complex sixth root of unity then the quotient
of hyperbolic 3-space by $\Gamma_\rho$ is the figure-eight knot complement.

\subsection{Main definitions and discreteness result}

The direct analogue of the Riley slice in complex hyperbolic plane would be the set of 
conjugacy classes of groups generated by two non-commuting, unipotent parabolic 
elements $A$ and $B$ of ${\rm SU}(2,1)$. (Note that in contrast to to ${\rm PSL}(2,\C)$, 
there exist parabolic elements in ${\rm SU}(2,1)$ that are not unipotent. In fact, there is 
a 1-parameter family of parabolic conjugacy classes, see for instance Chapter 6 of 
\cite{Go}.) This choice would give a four dimensional parameter space, and we require
additionally that $AB$ is unipotent; making the dimension drop to $2$.
Specifically, we define
\begin{equation}\label{eq-set-U} 
{ \mathcal{U} = \Bigl\{(A,B)\in {\rm SU}(2,1)^2\ :\  A,\,B,\, AB\,
\hbox{ all unipotent and } AB\neq BA \Bigr\}}/{SU(2,1)}.
\end{equation}

Following Riley, we are interested in the (closed) subset of ${\mathcal U}$
where the group $\langle A,B\rangle$ is discrete and free and our main method for
studying this set is to construct the Ford domain for its action on 
complex hyperbolic space $\HdC$. We shall also indicate
various other interesting discrete groups in ${\mathcal U}$ but these will not
be our main focus. 

In Section \ref{section-coordinates}, we will parametrise $\mathcal{U}$  
so that it becomes the open square $(-\pi/2,\pi/2)^2$. The parameters we
use will be the Cartan angular invariants $\alpha_1$ and $\alpha_2$ of the 
triples of (parabolic) fixed points of $(A,AB,B)$ and $(A,AB,BA)$ respectively (see 
Section  \ref{sec-Cartan} for the definitions). Note that the invariants
$\alpha_1$ and $\alpha_2$ are defined to lie in the closed interval
$[-\pi/2,\pi/2]$. Our assumption that $A$ and $B$ don't commute implies that 
neither $\alpha_1$ nor $\alpha_2$ can equal $\pm\pi/2$ 
(see Section \ref{section-coordinates}).

When $\alpha_1$ and $\alpha_2$ are both zero, that is at the origin of the square, 
the group $\langle A,B\rangle$ is ${\mathbb R}$-Fuchsian. The quotient of
the Lagrangian plane preserved by $\langle A,B\rangle$ is a hyperbolic three times punctured sphere
where the three (homotopy classes of) peripheral elements are represented
by (the conjugacy classes of) $A$, $B$ and $AB$. 
The space $\mathcal{U}$ can thus be thought of as the slice of the 
${\rm SU}(2,1)$-representation variety of the three times punctured 
sphere group defined by the conditions that the peripheral loops are 
mapped to unipotent isometries.

We can now state our main discreteness result.

\begin{theo}\label{thm-main}
Suppose that $\Gamma=\langle A,B\rangle$ is the group associated to parameters
$(\alpha_1,\alpha_2)$ satisfying
$\mathcal{D}\bigl(4\cos^2(\alpha_1),4\cos^2(\alpha_2)\bigr)>0$,
where $\mathcal{D}$ is the polynomial given by 
$$
\mathcal{D}(x,y)=x^3y^3-9x^2y^2-27xy^2+81xy-27x-27.
$$
Then $\Gamma$ is discrete and isomorphic to the free group $F_2$.
This region is $\mathcal{Z}$ in Figure \ref{picture-peach-1}.
\end{theo}

\begin{figure}
\begin{tabular}{cc}
\begin{minipage}[c]{.46\linewidth} \scalebox{0.4}{\includegraphics{knownpoints4-bis.eps}}\end{minipage}
& \begin{minipage}[c]{.46\linewidth}
 \begin{small}\begin{tabular}{ll}
0  & $\R$-Fuchsian representation of the $3$-punctured \\
   & sphere group.\\
1  & The horizontal segment marked $1$ corresponds to \\ 
   & even word subgroups of ideal triangle groups,\\
   & see \cite{GP,S1,S,S3}.\\
2  & Last ideal triangle group, contained with index\\ 
   & three in a group  uniformising the Whitehead \\ 
   & link complement obtained by Schwartz,\\  
   & see \cite{S1,S,S3}.\\
3  & The vertical segment marked $3$ corresponds to \\
   & bending groups that have been proved to be \\
   & discrete in \cite{Wi7}.\\
4  & $(3,3,4)$-group uniformising the figure eight knot \\
   & complement. Obtained by Deraux and Falbel \\
   & in \cite{DF8}.\\
5  & $(3,3,n)$-groups, proved to be discrete by in \cite{PWX}. \\
   & On this picture $4\leqslant n \leqslant 8$.\\
6  & Uniformisation of the Whitehead link complement \\
   & we obtain in this work.\\
7  & Subgroup of the Eisenstein-Picard Lattice, see \cite{FJP2}.
\end{tabular}\end{small}\end{minipage}
\end{tabular}

\caption{\label{picture-peach-1}The parameter space for $\mathcal{U}$. The exterior curve $\mathcal{P}$ corresponds to classes of groups for 
which $[A,B]$ is parabolic. The central dashed curve bounds the region $\mathcal{Z}$ where we prove discreteness. The labels correspond to 
various special values of the parameters. Points with the same labels are obtained from one another by symmetries about the 
coordinate axes. The results of Section \ref{sectionsymmetries} imply that they correspond to groups conjugate in Isom($\HdC$).}
\end{figure}

Note that at the centre of the square, we have $\mathcal{D}(4,4)=1225$ for the
${\mathbb R}$-Fuchsian representation. 
The region ${\mathcal Z}$ where ${\mathcal D}>0$ consists of groups $\Gamma$ 
whose Ford domain has the simplest possible combinatorial structure. It is the
analogue of the outermost region in the two figures from Akiyoshi, Sakuma, Wada 
and Yamashita \cite{ASWY} mentioned above.

\subsection{Decompositions and triangle groups}

We will prove in Proposition \ref{ST} that all pairs $(A,B)$ in $\mathcal{U}$ admit a
(unique) decomposition of the form
\begin{equation}\label{eq-ST}
A=ST\mbox{ and }B=TS,
\end{equation}
where $S$ and $T$ are order three regular elliptic elements (see 
Section \ref{section-isometries}). In turn, the group generated by $A$ and $B$ has 
index three in the one generated by $S$ and $T$. 
When either $\alpha_1=0$ or $\alpha_2=0$ there is a further decomposition
making $\langle A,B\rangle$ a subgroup of a triangle group. 

Deformations of triangle groups in ${\rm PU}(2,1)$ have been considered in many 
places, among which (\cite{GP,Pra,S2,PWX}). 
A complex hyperbolic $(p,q,r)$-triangle is one generated by three complex involutions
about (complex) lines with pairwise angles $\pi/p$, $\pi/q$, and $\pi/r$ where $p$, $q$
and $r$ are integers or $\infty$ (when one of them is $\infty$ the corresponding 
angle is $0$). Groups generated by complex reflections of higher order are also
interesting, see \cite{Most} for example, but we do not consider them here.
For a given triple $(p,q,r)$ with $\min\{p,q,r\}\geqslant 3$ the deformation space of 
the $(p,q,r)$-triangle group is one dimensional, and can be thought of as 
the deformation space of the $\R$-Fuchsian triangle group. 
In \cite{S2}, Schwartz develops a series of conjectures about which points in
this space yield discrete and faithful representations of the triangle group. 
For a given triple $(p,q,r)$, Conjecture 5.1 of \cite{S2} states that a complex hyperbolic 
$(p,q,r)$-triangle group is a discrete and faithful representation of the Fuchsian one 
if and only if the words $I_iI_jI_k$ and $I_iI_jI_kI_j$ (with $i,j,k$ pairwise distinct)
are non-elliptic. Moreover, depending on $p$, $q$ and $r$ he predicts
which of these words one should choose.

We now explain the relationship between triangle groups and groups on the axes
of our parameter space ${\mathcal U}$.
First consider groups with $\alpha_2=0$. Let $I_1$, $I_2$ and $I_3$ be
the involutions fixing the complex lines spanned by the fixed points of
$(A,\,B)$, of $(A,\,AB)$ and of $(B,\,AB)$ respectively. If $\alpha_2=0$ then 
$A$ and $B$ may be decomposed as $A=I_2I_1$ and $B=I_1I_3$, and also 
$\langle A,B\rangle$ has index 2 in $\langle I_1,I_2,I_3\rangle$ 
(Proposition \ref{prop-dec-ST}). Since $I_2I_1=A$, $I_1I_3=B$ 
and $I_2I_3=AB$ are all unipotent, we see that $\langle I_1,I_2,I_3\rangle$ is a 
complex hyperbolic ideal triangle group, as studied by Goldman and Parker \cite{GP} 
and Schwartz \cite{S1,S,S3}. Their results gave a complete characterisation of
when such a group is discrete. (Our Cartan invariant $\alpha_1$ is the same
as the Cartan invariant $\A$ used in these papers.)

\begin{theo}[Goldman, Parker \cite{GP}, Schwartz \cite{S,S3}]
\label{thm-CHIT}
Let $I_1$, $I_2$, $I_3$ be complex involutions fixing distinct, pairwise
asymptotic complex lines. Let $\A$ be the Cartan invariant of the fixed 
points of $I_1I_2$, $I_2I_3$ and $I_3I_1$. 
\begin{enumerate}
\item The group $\langle I_1,I_2,I_3\rangle$ is a discrete and faithful 
representation of an $(\infty,\infty,\infty)$-triangle group if and only 
if $I_1I_2I_3$ is non-elliptic. 
This happens when $|\A|\le\arccos\sqrt{3/128}$.
\item When $I_1I_2I_3$ is elliptic the group is not discrete. In this case
$\arccos\sqrt{3/128}<|\A|<\pi/2$.
\end{enumerate}
\end{theo}

When $\alpha_1=0$ we get an analogous result. In this case, it is the order
three maps $S$ and $T$ from \eqref{eq-ST} which decompose into products of 
complex involutions. Namely,  if $\alpha_1=0$, there exist three involutions 
$I_1$, $I_2$, $I_3$, each fixing a complex line, so that $S=I_2I_1$ and 
$T=I_1I_3$ have order 3 and $ST=A=I_2I_3$ is unipotent 
(Proposition \ref{prop-dec-ST}). Furthermore, writing $B=TS=I_1I_3I_2I_1$ 
we have $[A,B]=(ST^{-1})^3=(I_2I_1I_3I_1)^3$.
A corollary of Theorem \ref{thm-main}
is a statement analogous to Theorem \ref{thm-CHIT} 
for $(3,3,\infty)$-triangle groups, proving a special case of 
Conjecture 5.1 of Schwartz \cite{S2}. Compare with the proof of this conjecture 
for $(3,3,n)$-triangle groups given by Parker, Wang and Xie in \cite{PWX}.

\begin{theo}\label{thm-bend}
Let $I_1$, $I_2$ and $I_3$ be complex involutions fixing distinct complex
lines and so that $S=I_2I_1$ and $T=I_1I_3$ have order three and $A=ST=I_2I_3$
is unipotent. Let $\A$ be the Cartan invariant of the fixed points
of $A$, $SAS^{-1}$ and $S^{-1}AS$.
The group $\langle I_1,I_2,I_3\rangle$ is a discrete and faithful 
representation of the $(3,3,\infty)$-triangle group if and only if 
$I_2I_1I_3I_1=ST^{-1}$ is non-elliptic. This happens when $|\A|\le\arccos\sqrt{3/8}$.
\end{theo}

Theorem \ref{thm-bend} follows directly from Theorem \ref{thm-main} by restricting 
it to the case where $(\alpha_1,\alpha_2)=(0,\A)$.  These groups are a special case 
of those studied by Will in \cite{Wi7} from a different point of view. There, using 
bending he proved that these groups are discrete as long as $|\A|=|\alpha_2|\le \pi/4$. 
The gap between 
the vertical segment in Figure \ref{picture-peach-1} and the curve where $[A,B]$ is 
parabolic illustrates the non-optimality of the result of \cite{Wi7}.

\subsection{Spherical CR uniformisations of the Whitehead link complement}

The quotient of $\HdC$ by an $\R$ or $\C$-Fuchsian 
punctured surface group is a disc bundle over the 
surface. If the surface is non-compact, this bundle is trivial. Its boundary at infinity
is a circle bundle over the surface.   Such three-manifolds appearing on the 
boundary at infinity of quotients of $\HdC$ are naturally equipped with a 
\textit{spherical CR structure}, which is the analogue of the flat conformal 
structure in the real hyperbolic case. These structures are examples of 
$(X,G)$-structure, where $X=S^3=\partial\HdC$ and $G={\rm PU}(2,1)$. 
To any such structure on a three manifold $M$ are associated a holonomy 
representation $\rho : \pi_1(M)\longrightarrow {\rm PU}(2,1)$ and a 
developing map $D =\widetilde M \longrightarrow X$. This motivates the study 
of representations of fundamental groups of hyperbolic three manifolds 
in ${\rm PU}(2,1)$ and ${\rm PGL}(3,\C)$ initiated by Falbel in \cite{FalJDG}, and 
continued in \cite{FKR,FGKRT} (see also \cite{HMP}). Among 
${\rm PU}(2,1)$-representations, \textit{uniformisations} (see 
Definition 1.3 in \cite{Derrep}) are of special interest. 
There, the manifold at infinity is the quotient of the discontinuity region 
by the group action. 

For parameter values in the open region $\mathcal{Z}$, the manifold at infinity 
of $\HdC/\langle S,T\rangle$ is a Seifert fibre space over a $(3,3,\infty)$-orbifold. 
This is obviously true in the case where $\alpha_1=\alpha_2=0$ (the central point on 
Figure \ref{picture-peach-1}). Indeed, for these values the group $\la S,T\ra$ 
preserves $\HdR$ (it is $\R$-Fuchsian) and the fibres correspond to boundaries 
of real planes orthogonal to $\HdR$. As the combinatorics of our fundamental 
domain remains unchanged in $\mathcal{Z}$, 
the topology of the quotient is constant in $\mathcal{Z}$.

Things become interesting if we deform 
the group in such a way that a loop on the surface is represented by a 
parabolic map: the topology of the manifold at infinity can change. 
A hyperbolic manifold arising in this way was first constructed by Schwartz:

\begin{theo}[Schwartz \cite{S1}]\label{thm-wlc-1}
Let $I_1$, $I_2$ and $I_3$ be as in Theorem \ref{thm-CHIT}.
Let $\A$ be the Cartan invariant of the fixed 
points of $I_1I_2$, $I_2I_3$ and $I_3I_1$ and let $S$ be the regular 
elliptic map cyclically permuting these points.
When $I_1I_2I_3$ is parabolic the quotient of $\HdC$ by the group 
$\langle I_1I_2,S\rangle$ is a complex hyperbolic orbifold 
with isolated singularities whose boundary at infinity is a
spherical CR uniformisation of the Whitehead link complement. These groups
have Cartan invariant $\A=\pm\arccos\sqrt{3/128}$, 
\end{theo}

Schwartz's example provides a uniformisation of the 
Whitehead link complement. More recently, Deraux and Falbel 
described a uniformisation of the complement of the figure eight knot in \cite{DF8}.
In \cite{DerF8}, Deraux proved that this uniformisation was flexible: 
he described a one parameter deformation of the uniformisation described 
in \cite{DF8}, each group in the deformation being a uniformisation of the 
figure eight knot complement. 

Our second main result concerns the $(3,3,\infty)$ triangles group from 
Theorem \ref{thm-bend}, and it states that when $I_2I_1I_3I_1$ 
is parabolic the associated groups give a uniformisation 
of the Whitehead link complement which is different from Schwartz's one.
Indeed in our case the cusps of the Whitehead link complement both have 
unipotent holonomy. In Schwartz's case, one of them is unipotent whereas the 
other is screw-parabolic. The representation of the Whitehead link group we consider 
here was identified from a different point of view by Falbel, Koseleff and Rouillier
in their census of ${\rm PGL}(3,{\mathbb C})$ representations of 
knot and link complement groups, see page 254 of \cite{FKR}.

\begin{theo}\label{thm-wlc-2}
Let $I_1$, $I_2$ and $I_3$ be as in Theorem \ref{thm-bend} and define
$S=I_2I_1$ and $A=I_2I_3$. Let $\A$ be the Cartan invariant of the fixed points
of $A$, $SAS^{-1}$ and $S^{-1}AS$. When $I_2I_1I_3I_1$ is parabolic the quotient 
of $\HdC$ by $\langle A,S\rangle$ is a complex hyperbolic orbifold with 
isolated singularities whose boundary is 
a spherical CR uniformisation of the Whitehead link complement. These 
groups have Cartan invariant $\A=\pm\arccos\sqrt{3/8}$.
\end{theo}

Schwartz's uniformisation of the Whitehead link complement corresponds 
to each of the endpoints of the horizontal segment, marked 2 
in Figure \ref{picture-peach-1}, and our uniformisation corresponds
to each of the points on the vertical axis, marked 6 in that figure.

It should be noted that the image of the holonomy representation of our uniformisation 
of the Whitehead link complement is the group generated by $S$ and $T$, which is 
isomorphic to $\Z_3*\Z_3$. We note in Proposition \ref{prop-quotient} that the 
fundamental group of the Whitehead link complement surjects onto $\Z_3*\Z_3$.
Furthermore the group $\Z_3*\Z_3$ is the fundamental 
group of the (double) Dehn filling of the Whitehead link complement with 
slope $-3$ at each cusp in the standard marking (the same as in SnapPy). 
This Dehn filling is non-hyperbolic, as can be easily verified using the 
software SnapPy \cite{SnapPy} (it also follows from Theorem 1.3. in \cite{MP}). 
This fact should be compared with Deraux's remark in \cite{Derrep} that all 
known examples of non-compact finite volume hyperbolic manifold admitting a
spherical CR uniformisation also admit an exceptional Dehn filling which is a 
Seifert fibre space over a $(p,q,r)$-orbifold with $p,q,r,\geqslant 3$.

\subsection{Ideas for proofs.}
\paragraph{Proof of Theorem \ref{thm-main}. }
The rough idea of this proof is to construct fundamental domains for the 
groups corresponding to parameters in the region $\mathcal{Z}$. To this 
end, we construct their \textit{Ford domains}, which can be thought of 
as a fundamental domain for a coset decomposition of the group with respect 
to a parabolic element (here, this element is $A=ST$). The Ford 
domain is invariant by the subgroup generated by $A$ and we obtain a 
fundamental domain for the group by intersecting the Ford domain with a 
fundamental domain for the subgroup generated by $A$. The sides of the Ford 
domain are built out of pieces of \textit{isometric spheres} of various group 
elements (see Sections \ref{def-Isom} and \ref{section-isom-spheres})
This method is classical, and is described in the case of the Poincar\'e disc in 
Section 9.6 of Beardon \cite{Bear}.

We thus have to consider a 2-parameter family of such polyhedra, and the 
polynomial $\mathcal{D}$ controls the combinatorial complexity of the Ford 
domain within our parameter space for $\mathcal{U}$ in the following
sense. The null-locus of 
$\mathcal{D}$ is depicted on Figure \ref{picture-peach-1} as a dashed curve, 
which bounds the region $\mathcal{Z}$. In the interior of this curve, the 
combinatorics of our domain is constant, and stays the same as it is for the 
$\R$-Fuchsian group. On the boundary of $\mathcal{Z}$ the isometric spheres 
of the elements $S$, $S^{-1}$ and $T$ have a common point. More precisely, 
the isometric spheres of $S^{-1}$ and $T$ intersect for all values of $\alpha_1$ 
and $\alpha_2$, but inside $\mathcal{Z}$ their intersection is contained in one 
of the two connected components of the complement of the isometric sphere 
of $S$ in $\HdC$. When one reaches the boundary curve of $\mathcal{Z}$, one 
of their intersection points lies on the isometric sphere of $S$.

We believe that it should be possible to mimic Riley's approach and to construct 
regions in our parameter space where the Ford domain is more complicated. 
However, as with Riley's work, this may only be reasonable via computer experiments.

\paragraph{Proof of Theorem \ref{thm-wlc-2}.}

The groups where $[A,B]=(I_2I_1I_3I_1)^3$ is parabolic 
are the focus of Section \ref{section-limit-group} and 
Theorem \ref{thm-wlc-2} will follow from Theorem \ref{theo-CRWLC}. 
In order to prove this result, we analyse in details our fundamental domain, 
and show that it gives the classical description of the Whitehead link complement 
from an ideal octahedron equipped with face identifications.
The Whitehead link is depicted in Figure \ref{Whitehead}. We refer to 
Section 10.3 of Ratcliffe \cite{Rat} and Section 3.3 of Thurston \cite{Thu} for 
classical information about the topology of the Whitehead link complement and its 
hyperbolic structure.

\subsection{Further remarks}

\paragraph{Other discrete groups appearing in $\mathcal{U}$.}

As well as the ideal triangle groups and bending groups discussed
above, there are some other previously studied discrete groups in this family.
We give them in $(\alpha_1,\alpha_2)$ coordinates and
illustrate them in Figure \ref{picture-peach-1}.
\begin{enumerate}
\item The groups corresponding to $\alpha_1=0$ and 
$\alpha_2=\pm\arccos\sqrt{1/8}$ 
have been studied in great detail by 
Deraux and Falbel who proved that they give a spherical CR uniformisation 
of the figure-eight knot complement \cite{DF8}. This illustrates the fact that
there is no statement for Theorem \ref{thm-bend} analogous to the
second part of Theorem \ref{thm-CHIT}: the group from \cite{DF8} is 
contained in a discrete (non-faithful) $(3,3,\infty)$ triangle groups where 
$I_2I_1I_3I_1$ is elliptic.
\item The groups with parameters $\alpha_1=0$ and for which $ST^{-1}$ 
has order $n$ correspond to the $(3,3,n)$ triangle groups studied by 
Parker, Wang and Xie in \cite{PWX}. The corresponding value of $\alpha_2$ 
is given by $\alpha_2=\pm\arccos\sqrt{(4\cos^2(\pi/n)-1)/8}$.
\item The groups where $\alpha_1=\pm\pi/6$ and $\alpha_2=\pm\pi/3$ 
are discrete, since they are subgroups of the Eisenstein-Picard lattice 
${\rm PU}(2,1;\Z[\omega])$, where $\omega$ is a cube 
root of unity. That lattice has been studied by Falbel and Parker in 
\cite{FJP2}.
\end{enumerate}

\paragraph{Comparison with the classical Riley slice.}
There is, conjecturally, one extremely significant difference between 
the classical Riley slice and our complex hyperbolic version. The boundary of
the classical Riley slice is not a smooth curve and has a dense set of points where
particular group elements are parabolic (see for instance the beautiful picture in 
the introduction of \cite{KS}). On the other hand, we believe that in the complex 
hyperbolic case, discreteness is completely controlled by the commutator $[A,B]$, 
or equivalently $ST^{-1}$, as is true for the two cases where $\alpha_1=0$
or $\alpha_2=0$ described above. If this is true, then the boundary of the set 
of (classes of) discrete and faithful representations in ${\rm SU}(2,1)$ of the 
three punctured sphere group with unipotent peripheral holonomy is piecewise 
smooth, and it is given by the simple closed curve ${\mathcal P}$ 
in Figure \ref{picture-peach-1}. This curve provides a one parameter family of 
(conjecturally discrete) representations that connects Schwartz's 
uniformisation of the Whitehead link complement to ours. We believe that 
all these representations give uniformisations of the Whitehead link complement 
as well, but we are not able to prove this with our techniques. What  seems to happen 
is that if one deforms our uniformisation by following the curve $\mathcal{P}$,
the number of isometric spheres contributing to the boundary at infinity of the Ford 
domain becomes too large to be understood using our techniques. 
Possibly, this is because deformations of fundamental domains with tangencies
between bisectors is complicated. This should be compared to Deraux's
construction \cite{DerF8} of deformations of the figure-eight knot complement
mentioned above. There, he had to use a different domain to the one in \cite{DF8}, 
which also has tangencies between the bisectors. 

\medskip
\subsection{Organisation of the article.}

This article is organised as follows. In Section \ref{section-prelim-mat} we 
present the necessary background facts on complex hyperbolic space 
and its isometries. In Section \ref{section-parameter-space}, we describe 
coordinates on the space of (conjugacy classes) of group generated by two 
unipotent isometries with unipotent product. Section \ref{section-isom-spheres} 
is devoted to the description of the isometric spheres that bound our 
fundamental domains. We state and apply the Poincar\'e polyhedron theorem 
in Section \ref{section-poincare}. In Section \ref{section-limit-group}, we 
focus on the specific case where the commutator becomes parabolic, and 
prove that the corresponding manifold at infinity is homeomorphic to the 
complement of the Whitehead link. In Section \ref{section-technicalities}, 
we give the technical proofs which we have omitted for readability 
in the earlier sections.

\medskip

\subsection{Acknowledgements\label{ackowledgements}} 
The authors would like to thank Miguel Acosta, Martin Deraux, Elisha Falbel and 
Antonin Guilloux for numerous interesting discussions. The second author thanks 
Craig Hodgson, Neil Hoffman and Chris Leininger for kindly answering his na\"ive 
questions. This research was financially supported by ANR SGT and an 
LMS Scheme 2 grant. The research took place  during visits of both authors to  
Les Diablerets, Durham, Grenoble, Hunan University, ICTP and Luminy, and we 
would like to thank all these institutions for hospitality. We are very grateful to the
referee for their careful report and many helpful suggestions for improvement of the 
paper. The second author had the pleasure to share a very useful discussion with 
Lucien Guillou on a sunny June afternoon in Grenoble. Lucien has since passed away, 
and we remember him with affection.

\begin{figure}
\begin{center}
 \scalebox{0.6}{\includegraphics{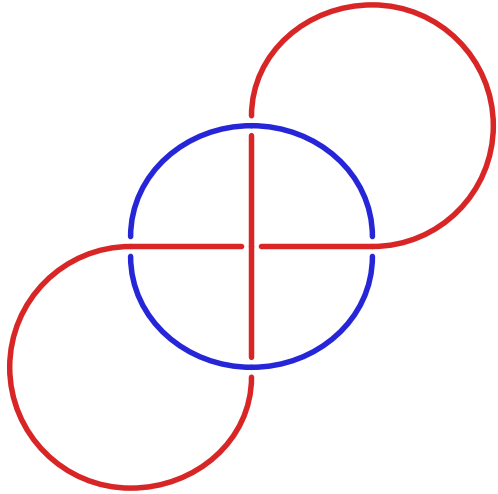}}
\end{center}
\caption{The Whitehead link\label{Whitehead}}
\end{figure}

\section{Preliminary material\label{section-prelim-mat}}

Throughout we will work in the complex hyperbolic plane using a projective 
model and will therefore pass from projective objects to lifts of them. 
Our convention is that the same letter will be used to denote a point in 
$\C P^2$ and a lift of it to $\C^3$ with a bold font for the lift. 
As an example, each time $p$ is a point of $\HdC$, $\bp$ will be a lift of $p$ 
to $\C^3$.

\subsection{The complex hyperbolic plane}
\label{section-proj-model}
The standard reference for complex hyperbolic space is Goldman's 
book \cite{Go}. A lot of information can also be found in Chen and Greenberg's
paper \cite{CG}, see also the survey articles \cite{PP2,Wi8}.

Let  $H$ be the following matrix
$$
H=\begin{bmatrix}0 & 0 & 1\\0 & 1 & 0 \\1 & 0 & 0\end{bmatrix}.
$$
The Hermitian product on $\C^3$ associated to $H$ is given by 
$\la {\bf x},{\bf y}\ra={\bf y}^*H{\bf x}$. The corresponding Hermitian 
form has signature $(2,1)$, and we denote by $V_-$ (respectively 
$V_0$ and $V_+$) the associated negative (respectively null and positive) 
cones in $\C^3$. 

\begin{defi}
The \emph{complex hyperbolic plane} $\HdC$ is the image of $V_-$ in $\C P^2$ by 
projectivisation and its boundary $\partial\HdC$ is the image of $V_0$ in 
$\C P^2$. The complex hyperbolic plane is endowed with the 
\emph{Bergman metric}
$$
ds^2 = \frac{-4}{\la {\bf z},{\bf z}\ra^2}{\rm det}
\left(\begin{matrix} 
\la {\bf z}, {\bf z} \ra & \la d{\bf z},{\bf z} \ra \\ 
\la {\bf z}, d{\bf z} \ra & \la d{\bf z},d{\bf z} \ra  
\end{matrix}\right).
$$
The Bergman metric is equivalent to the \emph{Bergman distance function} 
$\rho$ defined by
$$
\cosh^2{\left(\dfrac{\rho(m,n)}{2}\right)}
=\dfrac{\la \bm  ,\bn\ra\la\bn,\bm\ra}{\la\bm,\bm\ra\la\bn,\bn\ra},
$$ 
where $\bm$ and $\bn$ are lifts of $m$ and $n$ to $\C^3$.
\end{defi}

Let ${\bf z}=[z_1,\,z_2,\,z_3]^T$ be a (column) vector in $\C^3-\{{\bf 0}\}$. 
Then ${\bf z}\in V_-$ (respectively $V_0$) if and only if
$2{\rm Re}(z_1\overline{z}_3)+|z_2|^2<0$ (respectively $=0$). Vectors in $V_0$
with $z_3=0$ must have $z_2=0$ as well. Such a vector is unique up to scalar 
multiplication. We call such its projectivisation the \emph{point at infinity} 
$q_\infty\in\partial\HdC$. If $z_3\neq 0$ then we can use inhomogeneous 
coordinates with $z_3=1$. Writing $\langle{\bf z},{\bf z}\rangle=-2u$ we give
$\HdC\cup\partial\HdC-\{q_\infty\}$ \emph{horospherical coordinates}
$(z,t,u)\in\C\times\R\times\R_{\ge 0}$ defined as follows. A point
$q\in\HdC\cup\partial\HdC$ with horospherical coordinates
$(z,t,u)$ is represented by the following vector, which we call its \textit{standard lift}.
\begin{equation}\label{lift}
{\bf q}=\begin{bmatrix}-|z|^2-u+it\\z\sqrt{2}\\1\end{bmatrix}
\mbox{ if } q\neq q_\infty,\quad 
{\bf q}_\infty=\begin{bmatrix}1\\0\\0\end{bmatrix}
\mbox{ if } q= q_\infty. 
\end{equation}
Points of $\partial\HdC-\{q_\infty\}$ have $u=0$ and we will abbreviate
$(z,t,0)$ to $[z,t]$.

Horospherical coordinates give a model of complex hyperbolic space
analogous to the upper half plane model of the hyperbolic plane.
The \emph{Cygan metric} $d_{\rm Cyg}$ on $\partial\HdC-\{q_\infty\}$ 
plays the role of the Euclidean metric on the upper half plane. It is
defined by the distance function:
\begin{equation}\label{Cygan}
d_{\rm Cyg}(p,q)=\bigl|\langle{\bf p},{\bf q}\rangle\bigr|^{1/2}
=\Bigl||z-w|^2+i\bigl(t-s+{\rm Im}(z\overline{w})\bigr)\Bigr|^{1/2}
\end{equation} 
where $p$ and $q$ have horospherical coordinates $[z,t]$ and 
$[w,s]$. We may extend this metric to points $p$ and $q$ in $\HdC$ 
with horospherical coordinates $(z,t,u)$ and $(w,s,v)$ by writing
$$
d_{\rm Cyg}(p,q)
=\Bigl||z-w|^2+|u-v|+i\bigl(t-s+{\rm Im}(z\overline{w})\bigr)\Bigr|^{1/2}
$$
If (at least) one of $p$ and $q$ lies in $\partial\HdC$ then we still have 
the formula $d_{\rm Cyg}(p,q)=\bigl|\langle{\bf p},{\bf q}\rangle\bigr|^{1/2}$.

\subsection{Isometries\label{section-isometries}}

Since the Bergman metric and distance function are both given solely
in terms of the Hermitian form, any unitary matrix preserving this form
is an isometry. Similarly, complex conjugation of points in
$\C^3$  leaves both the metric and the distance function unchanged. Hence, complex 
conjugation is also an isometry.

Define ${\rm U}(2,1)$ to be the group of unitary matrices preserving
the Hermitian form and ${\rm PU}(2,1)$ to be the projective unitary
group obtained by identifying non-zero scalar multiples of matrices in
${\rm U}(2,1)$. We also consider the subgroup ${\rm SU}(2,1)$  of 
matrices  in ${\rm U}(2,1)$ with determinant $1$.

\begin{prop}
Every Bergman isometry of $\HdC$ is either holomorphic or 
anti-holomorphic. The group of holomorphic isometries is 
${\rm PU}(2,1)$, acting by projective transformations. 
Every antiholomorphic isometry is complex conjugation
followed by an element of ${\rm PU}(2,1)$.
\end{prop}

Elements of ${\rm SU}(2,1)$ fall into three types, according to the number and
type of the fixed points of the corresponding isometry. Namely, an
isometry is \textit{loxodromic} (respectively \textit{parabolic}) if it 
has exactly two fixed points (respectively exactly one fixed point) on 
$\partial\HdC$. It is called \textit{elliptic} when it has (at least) one 
fixed point inside $\HdC$. An elliptic element $A\in{\rm SU}(2,1)$ 
is called \textit{regular  elliptic} whenever it has three distinct 
eigenvalues, and \textit{special elliptic} if it has a repeated eigenvalue. 
The following criterion distinguishes the different isometry types. 

\begin{prop}[Theorem 6.2.4 of Goldman \cite{Go}]\label{tracefunction}
Let $\mathcal{F}$ be the polynomial given by $\mathcal{F}(z)=|z|^4-8\Re(z^3)+18|z|^2-27$,
and $A$ be a non identity matrix in ${\rm SU}(2,1)$. Then
\begin{enumerate}
\item $A$ is loxodromic if and only if $\mathcal{F}(\tr A)>0$,
\item $A$ is regular elliptic if and only if $\mathcal{F}(\tr A)<0$,
\item if $\mathcal{F}(\tr A)=0$, then $A$ is either parabolic or special  elliptic.
\end{enumerate}
\end{prop}

We will be especially interested in elements of ${\rm SU}(2,1)$
with trace $0$ and those with trace $3$.

\begin{lem}[Section 7.1.3 of Goldman \cite{Go}]\label{zerotrace}
\begin{enumerate}
\item A matrix $A$ in ${\rm SU}(2,1)$ is regular elliptic of order three 
if and only if its trace is equal to zero.
\item Let $(p,q,r)$ be three pairwise distinct points in
$\partial\HdC$, not contained in a common complex line. Then there
exists a unique order three regular elliptic isometry $E$ so that
$E(p)=q$ and $E(q)=r$.
\end{enumerate}
\end{lem}

Suppose that $T\in{\rm SU}(2,1)$ has trace equal to $3$. Then all
$T$ eigenvalues of $T$ equal $1$, that is $T$ is \emph{unipotent}.
If $T$ is diagonalisable then it must be the identity; if it
is non-diagonalisable then it must fix a point of $\partial\HdC$.
Conjugating within ${\rm SU}(2,1)$ if necessary, we may assume that
$T$ fixes $q_\infty$. This implies that $T$ is upper triangular with 
each diagonal element equal to $1$.

\begin{lem}[Section 4.2 of Goldman \cite{Go}]\label{lem-unique-unipotent}
Suppose that $[w,s]\in\partial\HdC-\{q_\infty\}$. Then there is a unique
$T_{[w,s]}\in{\rm SU}(2,1)$ taking the point $[0,0]\in\partial\HdC$
to $[w,s]$. As a matrix this map is:
\begin{equation}\label{heisenberg}
T_{[w,s]}=\begin{bmatrix} 
1 & -\overline{w} \sqrt{2} & -|w|^2+is\\
0 & 1 & w\sqrt{2}\\0 & 0 & 1\end{bmatrix}.
\end{equation}
Moreover, composition of such elements gives $\partial\HdC-\{q_\infty\}$
the structure of the Heisenberg group 
$$
[w,s]\cdot[z,t]=\bigl[w+z,s+t-2{\rm Im}(z\overline{w})\bigr]
$$
and $T_{[w,s]}$ acts as left Heisenberg translation on 
$\partial\HdC-\{q_\infty\}$.
\end{lem}

The action of $T_{[w,s]}$ on horospherical coordinates is:
$$
T_{[w,s]}:(z,t,u)\longmapsto \bigl(w+z,s+t-2{\rm Im}(z\overline{w}),u\bigr).
$$
An important observation is that this is an affine map, namely a translation
and shear.

We can restate Lemma \ref{lem-unique-unipotent} in an invariant way.
This result is actually true for any parabolic conjugacy class, as a special 
case of Proposition 3.1 in \cite{ParkWi}.
 
\begin{prop}\label{cor-unique-unipotent}
Let $(p_1,p_2,p_3)$ be a triple of pairwise distinct points in $\partial\HdC$.
Then there is a unique unipotent element of ${\rm PU}(2,1)$ fixing $p_1$ 
and taking $p_2$ to $p_3$.
\end{prop}

\begin{proof}
We can choose $A\in{\rm SU}(2,1)$ taking $p_1$ to $q_\infty$ and
$p_2$ to $[0,0]$. The result then follows from Lemma \ref{lem-unique-unipotent}.
\end{proof}

\subsection{Totally geodesic subspaces.\label{section-subspaces}}

Maximal totally geodesic subspaces of $\HdC$ have real dimension 2, and 
they fall in two types. Complex lines are intersections with $\HdC$ of 
projective lines in $\C P^2$. By Hermitian duality, any complex line $L$ 
is \textit{polar} to a point in $\C P^2$ that is outside the 
closure of $\HdC$. Any lift of this point is called a \textit{polar vector} 
to $L$. Any two distinct points $p$ and $q$ in the closure of $\HdC$ belong 
to a unique complex line, and a vector polar to this line is given by 
$\bp\boxtimes\bq = H \overline{\bp\wedge\bq}$.  This can be verified directly   
using $\la {\bf x},{\bf y}\ra={\bf y}^* H {\bf x}$ and the fact that 
here, $H^2=1$. A more general description of cross products in Hermitian vector spaces 
can be found in Section 2.2.7. of Chapter 2 of Goldman \cite{Go}.

The other type of maximal totally geodesic subspace is a Lagrangian plane.
Lagrangian planes are ${\rm PU}(2,1)$ images of the set of real points 
$\HdR\subset\HdC$. In particular, real planes are fixed points sets of 
antiholomorphic isometric involutions (sometimes 
called \textit{real symmetries}). 
The symmetry fixing $\HdR$ is complex conjugation.  
In turn, the symmetry about any other Lagrangian plane $M\cdot\HdR$,
where $M\in {\rm SU}(2,1)$, is given by 
${\bf z}\longmapsto M\overline{M^{-1}}\,\overline{\bf z}=M\overline{\bigl(M^{-1}{\bf z}\bigr)}$. Note that the matrix 
$N=M\overline{M^{-1}}$ satisfies $N\overline{N}=Id$: this reflects the fact
that real symmetries are involutions. We refer the reader to 
Chapter 3 and 4 of Goldman \cite{Go}.

\subsection{Isometric spheres}
\label{defisom}

\begin{defi}\label{Isomsphere}
For any $B\in{\rm SU}(2,1)$ that does not fix $q_\infty$,
the \emph{isometric sphere} of $B$ (denoted $\Isom(B)$) is
defined to be
\begin{equation}\label{def-Isom}
\Isom(B)=\Bigl\{p\in\HdC\cup\partial\HdC\ :\ 
\bigl|\langle{\bf p},{\bf q}_\infty\rangle\bigr|=
\bigl|\langle{\bf p},B^{-1}({\bf q}_\infty)\rangle\bigr|=
\bigl|\langle B({\bf p}),{\bf q}_\infty\rangle\bigr|\Bigr\}
\end{equation}
where ${\bf p}$ is the standard lift of $p\in\HdC\cup\partial\HdC$ given
in \eqref{lift}.

The \emph{interior} of $\Isom(B)$ is the component of its complement in $\HdC\cup\partial\HdC$ 
that do not contain $q_\infty$, namely,
$$
\Bigl\{p\in\HdC\cup\partial\HdC\ :\ 
\bigl|\langle{\bf p},{\bf q}_\infty\rangle\bigr|>
\bigl|\langle{\bf p},B^{-1}({\bf q}_\infty)\rangle\bigr|\Bigr\}.
$$
The \emph{exterior} of $\Isom(B)$ is the component that contains the point at infinity $q_\infty$
\end{defi}

Suppose $B$ is written as a matrix as
\begin{equation}\label{eq-B-form}
B=\left[\begin{matrix} a & b & c \\ d & e & f \\ g & h & j \end{matrix}\right].
\end{equation}
Then 
$B^{-1}({\bf q}_\infty)=\bigl[\overline{j},\,\overline{h},\,\overline{g}\bigr]^T$.
Thus $B$ fixes $q_\infty$ if and only if $g=0$. If $B$ does not fix $q_\infty$
(that is $g\neq 0$) the horospherical coordinates of $B^{-1}(q_\infty)$ are:
$$
B^{-1}(q_\infty) =\Bigl[ \overline{h}/\bigl(\overline{g}\sqrt{2}\bigr),\,
{\rm Im}\bigl(\overline{j}/\overline{g}\bigr)\Bigr].
$$

\begin{lem}[Section 5.4.5 of Goldman \cite{Go}]\label{lem-isomA}
Let $B\in{\rm PU}(2,1)$  be an isometry of $\HdC$ not fixing $q_\infty$. 
\begin{enumerate}
\item The transformation $B$ maps $\Isom(B)$ to $\Isom(B^{-1})$, and 
the interior of $\Isom(B)$ to the exterior of $\Isom(B^{-1})$. 
\item For any $A\in{\rm PU}(2,1)$ fixing $q_\infty$ and such that the 
corresponding eigenvalue has unit modulus, we have $\Isom(B)=\Isom(AB)$.
\end{enumerate}
\end{lem}

Using the characterisation \eqref{Cygan} of the Cygan metric in terms of 
the Hermitian form, the following lemma is obvious.

\begin{lem}\label{centrisomsphere}
Suppose that $B\in{\rm SU}(2,1)$ written in the form \eqref{eq-B-form}
does not fix $q_\infty$. Then the isometric sphere $\Isom(B)$ is the 
Cygan sphere in $\HdC\cup\partial\HdC$ with centre $B^{-1}(q_\infty)$ and 
radius $r_A=1/|g|^{1/2}$. 
\end{lem}

The importance of isometric spheres is that they form the boundary
of the \emph{Ford polyhedron}. This is the limit of Dirichlet polyhedra
as the centre point approaches $\partial\HdC$; see Section 9.3 of 
Goldman \cite{Go}. The Ford polyhedron $D$ for a discrete group $\Gamma$ 
is the intersection of the (closures of the) exteriors of all isometric 
spheres for elements of $\Gamma$ not fixing $q_\infty$. That is:
$$
D_\Gamma=\Bigl\{p\in\HdC\cup\partial\HdC \ :\ 
\bigl|\langle{\bf p},{\bf q}_\infty\rangle\bigr|\ge 
\bigl|\langle{\bf p},B^{-1}{\bf q}_\infty\rangle\bigr| \hbox{ for all }
B\in\Gamma \hbox{ with } B(q_\infty)\neq q_\infty \Bigr\}.
$$ 
Of course, just as for Dirichlet polyhedra, to construct the Ford polyhedron
one must check infinitely many equalities. Therefore our method will be to
guess the Ford polyhedron and check this using the Poincar\'e polyhedron 
theorem. When $q_\infty$ is either in the domain of discontinuity
or is a parabolic fixed point, the Ford polyhedron
is preserved by $\Gamma_\infty$, the stabiliser of $q_\infty$ in $\Gamma$.
It is a fundamental polyhedron for the partition of $\Gamma$ into 
$\Gamma_\infty$-cosets. In order to obtain a fundamental domain for
$\Gamma$, one must intersect the Ford domain with a fundamental
domain for $\Gamma_\infty$.

\subsection{Cygan spheres and geographical coordinates.
\label{section-geog}}

We now give some geometrical results about Cygan spheres. They are, in
particular, applicable to isometric spheres.
The Cygan sphere $\mathcal{S}_{[0,0]}(r)$ of radius $r>0$ with centre the origin
$[0,0]$ is the (real) hypersurface of $\HdC\cup\partial\HdC$ described in horospherical coordinates by the equation
\begin{equation}\label{unitspinal}
\mathcal{S}_{[0,0]}(r)
=\Bigl\lbrace (z,t,u)\ :\ \bigl(|z|^2+u\bigr)^2+t^2=r^4\Bigr\rbrace.
\end{equation}
From \eqref{unitspinal} we immediately see that when written in
horospherical coordinates the interior of $\mathcal{S}_{[0,0]}(r)$
is convex. The Cygan sphere $\mathcal{S}_{[w,s]}(r)$ of radius $r$
with centre $[w,s]$ is the image of $\mathcal{S}_{[0,0]}(r)$ under
the Heisenberg translation $T_{[w,s]}$. Since Heisenberg 
translations are affine maps in horospherical coordinates, we see that 
the interior of any Cygan sphere is convex. This immediately gives: 

\begin{prop}\label{prop-inter-isom}
The intersection of two Cygan spheres is connected. 
\end{prop}

Cygan spheres are examples of bisectors (otherwise called spinal hypersurfaces) 
and their intersection is an example of what Goldman calls an intersection of 
covertical bisectors. Thus Proposition \ref{prop-inter-isom} 
is a restatement of Theorem 9.2.6 of \cite{Go}. There is a natural system 
of coordinates on bisectors in terms of totally geodesic subspaces, 
see Section 5.1 of \cite{Go}. In particular for Cygan spheres, these are
defined as follows:

\begin{defi}\label{def-geog}
Let ${\mathcal S}_{[0,0]}(r)$ be the Cygan sphere with centre the origin $[0,0]$
and radius $r>0$. The point $g(\alpha,\beta,w)$
of ${\mathcal S}_{[0,0]}(r)$ with \emph{geographical coordinates} 
$(\alpha,\beta,w)$ is the point whose lift to $\C^3$ is:
\begin{equation}\label{liftgeog}
 {\bf g}(\alpha,\beta,w)=\begin{bmatrix}
      -r^2e^{-i\alpha}\\
      rwe^{i(-\alpha/2+\beta)}\\
      1
     \end{bmatrix},
\end{equation}
where 
$\beta\in[0,\pi)$, $\alpha\in[-\pi/2,\pi/2]$ and 
$w\in[-\sqrt{2\cos(\alpha)},\sqrt{2\cos(\alpha)}]$, 

Let ${\mathcal S}_{[z,t]}(r)$ be the Cygan sphere with centre $[z,t]$ and 
radius $r$. Then geographical coordinates on ${\mathcal S}_{[z,t]}(r)$ 
are obtained from the ones on ${\mathcal S}_{[0,0]}(r)$ by applying 
the Heisenberg translation $T_{[z,t]}$ to the vector \eqref{liftgeog}.
\end{defi}

We will only be interested in geographical coordinates on 
${\mathcal S}_{[0,0]}(1)$, the unit Cygan sphere centred at the origin.
Note that for the point $g(\alpha,\beta,w)$ of this sphere,  
$\la g(\alpha,\beta,u),g(\alpha,\beta,u)\ra
=w^2-2\cos(\alpha)$. Therefore the 
horospherical coordinates of $g(\alpha,\beta,w)$ are:
$$
\Bigl(we^{i(-\alpha/2+\beta)}/\sqrt{2},\ \sin(\alpha),\, 
\cos(\alpha)-w^2/2\Bigr)
$$
In particular, the points of ${\mathcal S}_{[0,0]}(1)$ on $\partial\HdC$  
are those with $w=\pm\sqrt{2\cos(\alpha)}$.

The level sets of $\alpha$ and $\beta$ are totally geodesic subspaces
of $\HdC$; see Example 5.1.8 of Goldman \cite{Go}.

\begin{prop}\label{prop-geog}
Let ${\mathcal S}_{[w,s]}(r)$ be a Cygan sphere with geographical coordinates
$(\alpha,\beta,w)$. 
\begin{enumerate}
\item For each $\alpha_0\in(-\pi/2,\pi/2)$ the set of
points $L_{\alpha_0}=\bigl\{g(\alpha,\beta,w)\in{\mathcal S}_{[w,s]}(r)\ :\ 
\alpha=\alpha_0\bigr\}$ is a complex line, called
a \emph{slice} of ${\mathcal S}_{[w,s]}(r)$.
\item For each $\beta_0\in[0,\pi)$ the set of points
$R_{\beta_0}=\bigl\{g(\alpha,\beta,w)\in{\mathcal S}_{[w,s]}(r)\ :\ 
\beta=\beta_0\bigr\}$ is a Lagrangian plane, called a 
\emph{meridian} of ${\mathcal S}_{[w,s]}(r)$.
\item The set of points with $w=0$ is the \emph{spine} of
${\mathcal S}_{[w,s]}(r)$. It is a geodesic contained in every
meridian.
\end{enumerate}
\end{prop}
\begin{rem}\label{rem-proj-cygan-sphere}
 From \eqref{unitspinal}, it is easy to see that projections of boundaries of Cygan spheres onto the $z$-factor are 
closed Euclidean discs in $\C$. This correspond to the vertical projection onto $\C$ in the Heisenberg group. This 
fact is often useful to prove that two Cygan spheres are disjoint.
\end{rem}

\subsection{Cartan's angular invariant.}\label{sec-Cartan}

\'Elie Cartan defined an invariant of triples of pairwise distinct points 
$p_1,\,p_2,\,p_3$ in $\partial\HdC$; see Section 7.1 of Goldman \cite{Go}. 
For any lifts $\bp_j$ of $p_j$ to $\C^3$, this invariant is defined by  
$\arg{\left(-\la\bp_1,\bp_2\ra\la\bp_2\bp_3\ra\la\bp_3,\bp_1\ra\right)}$, 
where the argument is chosen to lie in $(-\pi,\pi]$. 
We state here some important properties of $\A$.

\begin{prop}\label{prop-Cartan}[Sections 7.1.1 and 7.1.2 of \cite{Go}]
\begin{enumerate}
 \item $-\pi/2\le \A(p_1,p_2,p_3)\le \pi/2$ for any triple of pairwise 
distinct points $p_1$, $p_2$, $p_3$.
 \item $\A(p_1,p_2,p_3)=\pm\pi/2$ if and only if $p_1$, $p_2$, $p_3$
lie on the same complex line.
\item $\A(p_1,p_2,p_3)=0$ if and only if $p_1$, $p_2$, $p_3$
lie on the same Lagrangian plane.
\item Two triples $p_1$, $p_2$, $p_3$ and $q_1$, $q_2$, $q_3$ 
have $\A(p_1,p_2,p_3)=\A(q_1,q_2,q_3)$ if and only if there exists
$A\in{\rm SU}(2,1)$ so that $A(p_j)=q_j$ for $j=1,\,2,\,3$.
\item Two triples $p_1$, $p_2$, $p_3$ and $q_1$, $q_2$, $q_3$ 
have $\A(p_1,p_2,p_3)=-\A(q_1,q_2,q_3)$ if and only if there exists
an anti-holomorphic isometry $A$ so that $A(p_j)=q_j$ for $j=1,\,2,\,3$.
\end{enumerate}
 \end{prop}
The following proposition will be useful to us when we parametrise the 
family of classes of groups $\Gamma$.

\begin{prop}\label{Cartan-and-quadruples}
Let $(\alpha_1,\alpha_2)\in(-\pi/2,\pi/2)^2$. Then there exists a unique PU(2,1)-class of 
quadruples $(p_1,p_2,p_3,p_4)$ of pairwise distinct boundary points of $\HdC$  
such that
\begin{enumerate}
\item The complex lines  $L_{12}$ and $L_{34}$ respectively spanned by $(p_1,p_2)$ 
and $(p_3,p_4)$ are orthogonal.
 \item $\A(p_1,p_3,p_2)=\alpha_1$ and $\A(p_1,p_3,p_4)=\alpha_2$.
\end{enumerate}
\end{prop}

\begin{proof}
Since PU(2,1) acts transitively on pairs of distinct points of $\partial\HdC$,
we may assume using the Siegel model, that the points $p_i$ are given in 
Heisenberg coordinates by:
\begin{equation}\label{}
 p_1=q_\infty, \quad p_2=[0,0], \quad p_3=[1,t], \quad p_4=[z,s].
\end{equation}
Using the standard lifts given in Section \ref{section-proj-model} (denoted 
by $\bp_i$), we see by a direct computation using the Hermitian cross-product 
that
$$
\la\bp_1\boxtimes\bp_2,\bp_3\boxtimes\bp_4\ra=|z|^2-1+i(t-s).
$$
Thus the condition $L_{12}\perp L_{34}$ gives $|z|=1$ and $t=s$. We thus write 
$z=e^{i\theta}$ with $\theta\in[0,2\pi)$. Now computing the triple products 
we see that
$$
\A (p_1,p_3,p_2)=\arg(1-it)\quad \mbox{ and }\quad 
\A(p_1,p_3,p_4)=\arg(1-\overline{z})
=\arg\Bigl( 2ie^{i\frac{\theta}{2}}\sin{\bigl(\theta/2\bigr)}\Bigr).
$$
In particular $\alpha_1$ and $\alpha_2$ determine the values of $t$ and 
$\theta$.
\end{proof}

\medskip

\section{The parameter space\label{section-parameter-space}}

\subsection{Coordinates\label{section-coordinates}}
Our space of interest is the following.

\begin{defi}
Let $\mathcal{U}$ be the set of ${\rm PU}(2,1)$-conjugacy classes of 
non-elementary pairs $(A,B)$ such that $A$, $B$ and $AB$ are unipotent. 
\end{defi}

Here, by non-elementary, we mean that the two isometries $A$ and $B$ have 
no common fixed point in $\partial\HdC$. In fact, a slightly stronger 
statement will follow from Theorem \ref{theoclassgroups} below.
Namely $A$ and $B$ do not preserve a common complex line and so
the pair $A$, $B$ have no common fixed point in $\C P^2$ 
(see Section \ref{section-subspaces}). Another way to see this is that
if $A$ in ${\rm PU}(2,1)$ is unipotent and preserves a complex line, then
its action on that complex line is via a unipotent element of 
${\rm SL}(2,{\mathbb R})$ (that is parabolic with trace $+2$). 
It is well known that if $A$ and $B$ are unipotent
elements of ${\rm SL}(2,{\mathbb R})$ whose product is also unipotent
then $A$ and $B$ must share a fixed point (if $A$, $B$ and $AB$ are all
parabolic with distinct fixed points, at least one of them
should have trace $-2$).

Note that $BA=A^{-1}(AB)A=B(AB)B^{-1}$
and so if $AB$ is unipotent then so is $BA$. If $p_{AB}$ and $p_{BA}$
in $\partial\HdC$ are the fixed points of $AB$ and $BA$ then we have
$A(p_{BA})=p_{AB}$ and $B(p_{AB})=p_{BA}$. From 
Proposition \ref{cor-unique-unipotent} this means that $A$ and $B$ are
uniquely determined by the fixed points of $A$, $B$, $AB$ and $BA$.
We describe a set of coordinates on $\mathcal{U}$ expressed in terms 
of the Cartan invariants of triples of these fixed points.

\begin{theo}\label{theoclassgroups}
There is a bijection between $\mathcal{U}$ and the open square 
$(\alpha_1,\alpha_2)\in(-\pi/2,\pi/2)^2$, which is given by the map 
$$
\Lambda : (A,B)\longmapsto \left(\A(p_A,p_{AB},p_B),
\A(p_A,p_{AB},p_{BA})\right),
$$ 
where $p_A$, $p_B$, $p_{AB}$ and $p_{BA}$ are the parabolic fixed points of 
the corresponding isometries.
\end{theo}

This result can be see as a special case of the main result of \cite{ParkWi}. 
For completeness, we include here a direct proof. 

\begin{proof}
First, the two quantities $\alpha_1=\A(p_A,p_{AB},p_B)$ and 
$\alpha_2=\A(p_A,p_{AB},p_{BA})$ are invariant under ${\rm PU}(2,1)$-conjugation 
and thus the map $\Lambda$ is well-defined. Let us first prove that the
image of $\Lambda$ is contained in $(-\pi/2,\pi/2)^2$. In other words, we must
show $\alpha_1\neq\pm\pi/2$ and $\alpha_2\neq\pm\pi/2$.

Fix a choice of lifts $\bp_A$, $\bp_B$, $\bp_{AB}$ and $\bp_{BA}$ for the 
fixed points of $A$, $B$, $AB$ and $BA$. Since the fixed points are
assumed to be distinct, we see that the Hermitian product of each 
pair of these vectors does not vanish. The conditions $A(p_{BA})=p_{AB}$ 
and $B(p_{AB})=p_{BA}$ imply that there exist two non-zero complex numbers 
$\lambda$ and $\mu$ satisfying  
$$
A\bp_{BA}=\lambda \p_{AB}\quad \mbox{ and }\quad B\bp_{AB}=\mu\bp_{BA}.
$$ 
As $AB$ is unipotent, its eigenvalue associated to $\bp_{AB}$ is $1$, 
and therefore $\lambda\mu=1$. 
Moreover, using the fact that $\bp_{A}$ and $\bp_{B}$ are eigenvectors
of $A$ and $B$ with eigenvalue $1$, we have
\begin{equation}\label{eq-eigenvectors}
\la\bp_{BA},\bp_{A}\ra=\la A\bp_{BA},A\bp_{A}\ra=\lambda\la\bp_{AB},\bp_{A}\ra,
\quad
\la\bp_{AB},\bp_{B}\ra=\la B\bp_{AB},B\bp_{B}\ra=\mu\la\bp_{BA},\bp_{B}\ra.
\end{equation}
Using $\lambda\mu=1$ and \eqref{eq-eigenvectors}, it is not hard to show 
that ${\bf n}_1=\lambda\bp_{AB}-\bp_{BA}$ is a
polar vector for the complex line $L_1$ spanned by $\bp_A$ and $\bp_B$ (see Section \ref{section-subspaces}). 
Moreover, $\la\bp_{AB},{\bf n}_1\ra=-\la\bp_{AB},\bp_{BA}\ra\neq 0$.
Thus $p_{AB}$ does not lie on $L_1$. That is, the three of points 
$p_A,\,p_B,\,p_{AB}$ do not lie on the same complex line and so
$\alpha_1\neq\pm\pi/2$.

Likewise, again using $\lambda\mu=1$ and \eqref{eq-eigenvectors}
we find ${\bf n}_2=\la\bp_B,\bp_{AB}\ra\bp_A-\la\bp_A,\bp_{AB}\ra\bp_B$ 
is a polar vector for $L_2$ and 
$\la\bp_A,{\bf n}_2\ra=-\la\bp_A,\bp_{AB}\ra\la\bp_A,\bp_B\ra\neq 0$.
Hence $p_A$ does not lie on $L_2$ and so $\alpha_2\neq\pm\pi/2$.
We remark that, by construction, we have $\langle{\bf n}_1,{\bf n}_2\rangle=0$ 
and so in fact $L_1$ and $L_2$ are orthogonal.

To see that $\Lambda$ is surjective, fix $(\alpha_1,\alpha_2)$ in 
$(-\pi/2,\pi/2)^2$ and define
\begin{equation}\label{eq-notation-x1x2}
 x_1=\sqrt{2\cos(\alpha_1)} \quad\mbox{ and }\quad
x_2=\sqrt{2\cos(\alpha_2)},\quad\mbox{ for } \alpha_i\in(-\pi/2,\pi/2), \mbox{ so } x_1, x_2 \in\R_+^*.
\end{equation}
Now consider the following elements of ${\rm SU}(2,1)$:
\begin{equation}\label{normalAB}
A = \begin{bmatrix}
1 & -x_1x_2^2 & -x_1^2x_2^2 e^{-i\alpha_2}\\
0 & 1 & x_1x_2^2\\          
0 & 0 & 1\end{bmatrix}
\mbox{ and } 
B  = \begin{bmatrix}
1 & 0 & 0\\
x_1x_2^2e^{-i\alpha_1} & 1 & 0\\
-x_1^2x_2^2 e^{i\alpha_2} & -x_1x_2^2e^{i\alpha_1} & 1\end{bmatrix},
\end{equation}
Clearly, $A$ and $B$ are unipotent, and since $\tr(AB)=3$, $AB$ is also 
unipotent. The four fixed points can be lifted to the vectors 
\begin{equation}\label{fixedpoints}
\bp_A=\begin{bmatrix}1 \\ 0 \\ 0\end{bmatrix}, \quad 
\bp_B=\begin{bmatrix}0 \\ 0 \\ 1\end{bmatrix}, \quad
\bp_{AB}=\begin{bmatrix} -e^{i\alpha_1}\\x_1e^{i\alpha_2}\\1\end{bmatrix}, \quad
\bp_{BA}=\begin{bmatrix} -e^{i\alpha_1}\\-x_1e^{-i\alpha_2}\\1\end{bmatrix}.
\end{equation}
They satisfy $\A(p_A,p_{AB},p_B)=\alpha_1$ and $\A(p_A,p_{AB},p_{BA})=\alpha_2$.
Note that when either $\alpha_1$ or $\alpha_2$ tends to $\pm/\pi/2$  
(that is $x_1$ or $x_2$ respectively tends to $0$), $A$ and $B$ both
tend to the identity matrix.

To see that $\Lambda$ is injective, it suffices to prove that 
the quadruple $(p_A,p_B,p_{AB},p_{BA})$ is uniquely determined 
by $(\alpha_1,\alpha_2)$ up to isometry. Indeed, once 
this quadruple is fixed, $A$ and $B$ are uniquely determined by 
Proposition \ref{cor-unique-unipotent}. The above discussion has proved that 
for any pair $(A,B)$ in $\mathcal{U}$ the two complex lines spanned 
respectively by $(p_A,p_B)$ and $(p_{AB},p_{BA})$ are orthogonal. The result 
then follows straightforwardly from Proposition \ref{Cartan-and-quadruples}.
 \end{proof}

\medskip

From now on, we will identify any conjugacy class of pair in 
$\mathcal{U}$ with its representative given by \eqref{normalAB}.
We will repeatedly use the notation $x_i=\sqrt{2\cos(\alpha_i)}$ from 
\eqref{eq-notation-x1x2} and, 
when necessary, we will freely combine $x_i$ with trigonometric notation.
It should be noted that the unipotent isometry $A$ given by 
\eqref{normalAB} is equal to the Heisenberg translation $T_{[\ell_A,t_A]}$ 
(see Lemma \ref{lem-unique-unipotent}), where
\begin{equation}\label{eq-def-lA}
 \ell_A=x_1x_2^2/\sqrt{2}=2\cos(\alpha_1)\cos^2(\alpha_2)\quad
 \mbox{ and } \quad
t_A=x_1^2x_2^2\sin(\alpha_2)=4\cos(\alpha_1)\cos(\alpha_2)\sin(\alpha_2).
\end{equation}

\subsection{Products of order 3 elliptics.}\label{normal}

The following proposition gives a decomposition of pairs in $\mathcal{U}$ 
that we will use in the rest of this work.

\begin{prop}\label{ST}
 For any pair $(A,B)\in\mathcal{U}$, their exists a unique pair of isometries 
$(S,T)$ such that:
\begin{enumerate}
 \item Both $S$ and $T$ have order three, and they cyclically 
permute $(p_A,p_{AB},p_B)$ and $(p_A,p_B,p_{BA})$, respectively.
\item $A=ST$ and $B=TS$.
\end{enumerate}
\end{prop}

\begin{proof}
The first item is a direct consequence of Lemma \ref{zerotrace} (note that 
 neither of the triples $(p_A,p_{AB},p_B)$ and $(p_A,p_B,p_{BA})$ is 
contained in a complex line by Theorem \ref{theoclassgroups}). The action of 
$S$ and $T$ is summed up on Figure \ref{picture-action-ST}. From this, we see 
that $ST$ (resp. $TS$) fixes $p_A$ (resp. $p_B$) and maps $p_{BA}$ to $p_{AB}$ 
(resp. $p_AB$ to $p_{BA}$). Provided $ST$ and $TS$ are unipotent, this suffices 
to prove the second item by Proposition \ref{cor-unique-unipotent}. To see 
that $ST$ and $TS$ are indeed unipotent, we can use the lifts of $p_A$, 
$p_B$, $p_{AB}$ and $p_{BA}$ given by \eqref{fixedpoints}. In this case we have
\begin{equation}\label{S}
 S=e^{-i\alpha_1/3}\begin{bmatrix}
e^{i\alpha_1} & x_1e^{i\alpha_1-i\alpha_2} & -1\\
-x_1e^{i\alpha_2} & -e^{i\alpha_1} & 0\\
-1 & 0 & 0   
 \end{bmatrix}, \quad
 T=e^{i\alpha_1/3}\begin{bmatrix}
0 & 0 & -1\\
0 & -e^{-i\alpha_1} & -x_1e^{-i\alpha_1-i\alpha_2}\\
-1 & x_1e^{i\alpha_2} & e^{-i\alpha_1}
 \end{bmatrix},
\end{equation}
where, as usual, $x_i=\sqrt{2\cos(\alpha_i)}$; see \eqref{eq-notation-x1x2}. 
Computing the products 
$ST$ and $TS$ gives the result.
\end{proof}

We will use the notation $S$ and $T$ for these two order three symmetries 
throughout the paper.





\medskip

A more geometric proof of the existence of order three elliptic isometries 
decomposing pairs of parabolics as above can be found in a slightly more 
general context in \cite{ParkWi}.
\begin{figure}
 \begin{center}
 \scalebox{0.4}{\includegraphics{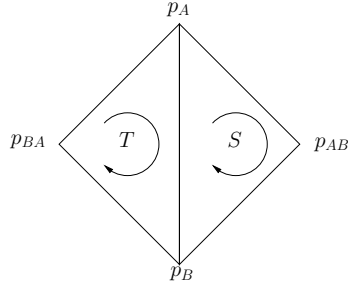}}
\caption{Action of $S$ and $T$ on the tetrahedron $(p_A,p_B,p_{AB},p_{BA})$.
\label{picture-action-ST}}
\end{center}
\end{figure}

One consequence of the existence of this decomposition as a product of order 
three elliptic is that any group generated 
in bu a pair $(A,B)$ in $\mathcal{U}$ is the image of the fundamental group 
of the Whitehead link complement by a morphism  
to PU(2,1). This follows directly from the following.

\begin{prop}\label{prop-quotient}
The free product $\Z_3\ast\Z_3$ is a quotient of the fundamental group of 
the Whitehead link complement.
\end{prop}

\begin{proof}
 The fundamental group of the Whitehead link complement is presented by 
$\pi=\la u,v |{\rm rel}(u,v)\ra$, where
\begin{equation}\label{fund-WLC}
{\rm rel} (u,v)=[u,v]\cdot[u,v^{-1}]\cdot[u^{-1},v^{-1}]\cdot[u^{-1},v]
\end{equation}
Making the substitution $u=st$ and $v=tst$, the relation becomes ${\rm rel }(st,tst)=[st,s^{-1}t^{-3}s^{-2}]$. 
This relation is trivial whenever $s^3=t^3=1$. Therefore, one defines a 
morphism $\mu : \pi\longrightarrow \Z_3\ast\Z_3$ by setting $\mu(u)=st$ and 
$\mu(v)=tst$. The morphism $\mu$ is surjective: $t$ is the image of 
$vu^{-1}$ and $s$ the image of $u^2v^{-1}$.
\end{proof}

\subsection{Symmetries of the moduli space}
\label{sectionsymmetries}

The parameters $(\alpha_1,\alpha_2)$ determine $\Gamma$ up to
${\rm PU}(2,1)$ conjugation. We now show that there is an antiholomorphic
conjugation that changes the sign of both $\alpha_1$ and $\alpha_2$.

\begin{prop}\label{prop-involution}
There is an antiholomorphic involution $\iota$ with the properties:
\begin{enumerate}
\item $\iota$ interchanges $p_A$ and $p_B$ and interchanges 
$p_{AB}$ and $p_{BA}$;
\item $\iota$ conjugates $S$ to $T$ and $A$ to $B$ (and vice versa);
\item $\iota$ conjugates the group $\Gamma$ with parameters 
$(\alpha_1,\alpha_2)$ to the group with parameters $(-\alpha_1,-\alpha_2)$.
\end{enumerate}
\end{prop}

\begin{proof}
The action on $\C^3$ of $\iota$ is:
$$
\iota:\left[\begin{matrix} z_1 \\ z_2 \\ z_3 \end{matrix}\right]
\longmapsto \left[\begin{matrix} 
\overline{z}_3 \\ e^{-i\alpha_1}\overline{z}_2 \\ \overline{z}_1
\end{matrix}\right].
$$
It is easy to see that $\iota^2$ is the identity and that $\iota$ sends
${\bf p}_A$ to ${\bf p}_B$ and sends ${\bf p}_{AB}$ to
$(-e^{-i\alpha_1}){\bf p}_{BA}$. Projectivising gives the first part.

Since $A$ is the unique unipotent map fixing $p_A$ and sending $p_{BA}$ to
$p_{AB}$, we see $\iota A\iota$ is the unique unipotent map fixing
$\iota(p_{A})=p_B$ and sending $\iota(p_{BA})=p_{AB}$ to 
$\iota(p_{AB})=p_{BA}$. Thus $\iota A\iota=B$ and so $\iota B\iota=A$.
Applying Proposition \ref{ST} we see that $\iota S\iota=T$ and
$\iota T\iota=S$, proving the second part.

The parameters associated to the group $\iota\Gamma\iota$ are
$\A(\iota p_A,\iota p_{AB},\iota p_B)=\A(p_B,p_{BA},p_A)=-\alpha_1$ and 
$\A(\iota p_A,\iota p_{AB},\iota p_{BA})=\A(p_B,p_{BA},p_{AB})=-\alpha_2$.
This completes the proof.
\end{proof}

There are other symmetries of the parameter space $\mathcal{U}$ that, in general, 
do not arise from conjugation by isometries.

\begin{prop}\label{prop-sym-moduli}
Let $\phi_h:(\alpha_1,\alpha_2)\longmapsto(\alpha_1,-\alpha_2)$ and
$\phi_v:(\alpha_1,\alpha_2)\longmapsto(-\alpha_1,\alpha_2)$ 
denote the symmetries about the horizontal and vertical axes of the 
$(\alpha_1,\alpha_2)$-square.
Then $\phi_h\circ\phi_v$ induces the conjugation by $\iota$ given in
Proposition \ref{prop-involution}. Moreover:
\begin{enumerate}
 \item $\phi_h$ induces the change of generators 
$(S,T)\longmapsto (T^{-1},S^{-1})$ and $(A,B)\longmapsto (A^{-1},B^{-1})$.
 \item $\phi_v$ induces the change of generators 
$(S,T)\longmapsto (S^{-1},T^{-1})$ and $(A,B)\longmapsto (B^{-1},A^{-1})$, 
\end{enumerate}
\end{prop}

\begin{proof}
Making the change $\phi_h$ to the points in 
\eqref{fixedpoints} and then applying ${\rm diag}[1,\,-1,\,1]\in{\rm PU}(2,1)$
fixes $p_A$ and $p_B$ and swaps $p_{AB}$ and $p_{BA}$. Therefore it
sends $S$ to the map cyclically permuting $(p_A,p_{BA},p_B)$, which is
$T^{-1}$. Similarly it sends $T$ to $S^{-1}$.

It is clear that the change of generators $(S,T)\longmapsto (T^{-1},S^{-1})$ 
sends $A=ST$ to $T^{-1}S^{-1}=A^{-1}$ and $B=TS$ to $S^{-1}T^{-1}=B^{-1}$.

The change of generators $(A,B)\longmapsto (A^{-1},B^{-1})$ fixes
$p_A$ and $p_B$. Since it sends $AB$ to $A^{-1}B^{-1}=(BA)^{-1}$ it sends 
$p_{AB}$ to $p_{BA}$ and similarly sends $p_{BA}$ to $p_{AB}$. 
From this we can calculate the new Cartan invariants and we obtain 
the symmetry $\phi_h$.

Hence all three conditions in the first part are equivalent. 
The second part then follows the first part and 
Proposition \ref{prop-involution} by first applying $\phi_h$ and then
conjugating by $\iota$.
\end{proof}

The fixed point sets of these automorphisms are related to $\R$-decomposability and 
$\C$-decomposability of $\Gamma$.

\begin{defi}[Compare Will \cite{Wi4}]
A pair $(S,T)$ of elements in ${\rm PU}(2,1)$ is $\R$-decomposable if 
there exist three antiholomorphic involutions 
$(\iota_1,\iota_2,\iota_3)$ such that $S=\iota_2\iota_1$ and 
$T=\iota_1\iota_3$.

A pair $(S,T)$ of elements in ${\rm PU}(2,1)$ is $\C$-decomposable if 
there exists three involutions $(I_1,I_2,I_3)$ in ${\rm PU}(2,1)$ 
such that $S=I_2I_1$ and $T=I_1I_3$. 
\end{defi}

The properties of $\R$ and $\C$-decomposability have also been 
studied (in the special case of pairs of loxodromic isometries) from the 
point of view of traces in ${\rm SU}(2,1)$ in \cite{Wi4}, and (in the general 
case) using cross-ratios in \cite{PaW}. We could take either point of view 
here, but instead we choose to argue directly with fixed points.

\begin{prop}\label{prop-dec-ST}
Let $(A,B)$ be in $\mathcal{U}$, and $(S,T)$ be the corresponding elliptic 
isometries.
\begin{enumerate}
\item 
If $\alpha_1=0$, then the pair $(S,T)$ is $\C$-decomposable
and the pair $(A,B)$ is $\R$-decomposable. 
In particular, $\la S,T\ra$ has index $2$ in a $(3,3,\infty)$-triangle group.
\item If $\alpha_2=0$, then the pair $(S,T)$ is $\R$-decomposable
and the pair $(A,B)$ is $\C$-decomposable. In particular 
$\la A,B\ra$ has index two in a complex hyperbolic ideal triangle group.
\end{enumerate}
\end{prop}

\begin{proof}
Consider the antiholomorphic involution
$\iota_1:[z_1,\,z_2,\,z_3]\longmapsto 
\bigl[\overline{z}_1,\,-\overline{z}_2,\,\overline{z}_3 \bigr]$. 
Applying $\iota_1$ to the points in
\eqref{fixedpoints} with $\alpha_1=0$, we see that
$\iota_1$ fixes $p_A$ and $p_B$ and interchanges $p_{AB}$ and $p_{BA}$.
Therefore $\iota_1$ conjugates $A$ to $A^{-1}$ and $B$ to $B^{-1}$.
Hence $A\iota_1A\iota_1$ and $\iota_1B\iota_1B$ are the identity.
That is $\iota_2=A\iota_1$ and $\iota_3=\iota_1B$ are involutions.
Hence $(A,B)$ is $\R$-decomposable.

Again assuming $\alpha_1=0$, consider the holomorphic involution
defined by $I_1=\iota_1\iota$ (where $\iota$ is the involution 
defined in Proposition \ref{prop-involution}). Then $I_1$
fixes $p_{AB}$ and $p_{BA}$ and interchanges $p_A$ and
$p_B$. Therefore, it conjugates $S$ to $S^{-1}$ and $T$ to $T^{-1}$.
This means $I_2=SI_1$ and $I_3=I_1T$ are involutions. Hence
$(S,T)$ is $\C$-decomposable.
 
Now consider the holomorphic involution
$I_1':[z_1,\,z_2,\,z_3]\longmapsto [z_1,\,-z_2,\,z_3]$. This fixes
$p_A$ and $p_B$ and when $\alpha_2=0$ it interchanges $p_{AB}$ and
$p_{BA}$. As above this means $I_2'=AI_1'$ and $I_3'=I_1'B$ are
involutions and $(A,B)$ is $\C$-decomposable.
Finally, define $\iota_1'=I_1'\iota$. Arguing as above, again with
$\alpha_2=0$, we see that $\iota'_2=S\iota'_1$ and $\iota'_3=\iota'_1T$ 
are involutions. Hence $(S,T)$ is $\R$-decomposable.
\end{proof}

As indicated above, when $\alpha_1=0$ the group generated by 
$(I_1,I_2,I_3)$ is a $(3,3,\infty)$ reflection triangle group. 
This group can be thought of as 
a limit as $n$ tends to infinity of the $(3,3,n)$ triangle groups which 
have been studied by Parker, Wang and Xie in \cite{PWX}. The special 
case $(3,3,4)$ has been studied by Falbel and Deraux in \cite{DF8}. 
Both \cite{DF8} and \cite{PWX} constructed Dirichlet domains, and the 
Ford domain we construct can be seen as a limit of these. 
Moreover, $\R$-decomposability of the pair $(A,B)$ when $\alpha_1=0$ can 
be used to show that these groups correspond to the bending representations 
of the fundamental group of a 3-punctured sphere that have been studied in 
\cite{Wi7}. Ideal triangle groups have been studied in great detail 
in \cite{GP,S,S1,S3,Schbook}.

\subsection{Isometry type of the commutator.\label{section-type-commutator}}

The isometry type of the commutator will play an important role in the rest 
of this paper. It is easily described using the order three elliptic maps given 
by Proposition \ref{ST}.

\begin{prop}\label{proptypecomm}
The commutator $[A,B]$ has the same isometry type as $ST^{-1}$. More 
precisely, consider
$\mathcal{G}(x_1^4,x_2^4)
=\mathcal{G}\bigl(4\cos^2(\alpha_1),4\cos^2(\alpha_2)\bigr)$ 
where
$$
\mathcal{G}(x,y)=x^2y^4-4x^2y^3+18xy^2-27.
$$
Then $[A,B]$ is loxodromic (respectively parabolic, elliptic) if and only if 
$\mathcal{G}(x_1^4,x_2^4)$ is positive (respectively zero, negative).
\end{prop}

\begin{proof}
First, from $A=ST$, $B=TS$ and the fact that $S$ and $T$ have order 3, 
we see that 
\begin{equation*}
[A,B]=ABA^{-1}B^{-1} = STTST^{-1}S^{-1}S^{-1}T^{-1}=(ST^{-1})^3.
\end{equation*}
This implies that $[A,B]$ has the same isometry type as $ST^{-1}$ unless 
$ST^{-1}$ is elliptic of order three, in which case $[A,B]$ is the identity. 
This would mean that $A$ and $B$ commute, which can not be because their 
fixed point sets are disjoint. 

Representatives of $S$ and $T$ in ${\rm SU}(2,1)$ are given in \eqref{S}.
A direct calculation using these matrices shows that 
${\rm tr}(ST^{-1})=x_1^2x_2^4e^{i\alpha_1/3}$.
The function $\mathcal{G}(x_1^4,x_2^4)$ above is obtained by plugging this 
value in the function $\mathcal{F}$ given in Proposition \ref{tracefunction}.
\end{proof}

The null locus of $\mathcal{G}\bigl(4\cos^2(\alpha_1),4\cos^2(\alpha_2)\bigr)$
in the square $(-\pi/2,\pi/2)^2$ is a curve, which 
we will refer to as the \textit{parabolicity curve} and denote by 
$\mathcal{P}$. It is depicted on Figure \ref{peachgraph}.  
Similarly, the region where $\mathcal{G}$ is positive (thus $[A,B]$ loxodromic) 
will be denoted by $\mathcal{L}$. It is a topological disc, which is the 
connected component of the complement of the curve $\mathcal{P}$ 
that contains the origin. The region where $[A,B]$ is elliptic will be 
denoted by $\mathcal{E}$.

\section{Isometric spheres and their intersections\label{section-isom-spheres}}

\subsection{Isometric spheres for $S$, $S^{-1}$ and their $A$-translates.}

In this section we give details of the isometric spheres that will contain
the sides of our polyhedron $D$. The polyhedron $D$ is our guess for the
Ford polyhedron of $\Gamma$, subject to the combinatorial restriction
discussed in Section \ref{section-def-Z}.

We start with the isometric spheres $\I(S)$ and $\I(S^{-1})$ for $S$
and its inverse. From the matrix for $S$ given in \eqref{S}, using
Lemma \ref{centrisomsphere} we see that $\I(S)$ and $\I(S^{-1})$
have radius $1/|-e^{-i\alpha_1/3}|^{1/2}=1$ and centres
$S^{-1}(q_\infty)=p_B$ and $S(q_\infty)=p_{AB}$ respectively; see \eqref{fixedpoints}. 
In particular, $\I(S)$ is the Cygan sphere $\mathcal{S}_{[0,0]}(1)$ of radius $1$
centred at the origin; see \eqref{unitspinal}. In our computations
we will use geographical coordinates in $\I(S)$ as in 
Definition \ref{def-geog}. The polyhedron $D$
will be the intersection of the exteriors of $\I(S^{\pm 1})$ and all 
their translates by powers of $A$. We now fix some notation: 

\begin{defi}
For $k\in\Z$ let $\Isom_k^+$ be the isometric sphere $\I(A^kSA^{-k})=A^k\I(S)$
and let $\Isom_k^-$ be the isometric sphere $\I(A^kS^{-1}A^{-k})=A^k\I(S^{-1})$.
\end{defi}

With this notation, we have:

\begin{prop}\label{centre-translates-S-Sm}
For any integer $k\in\Z$, the isometric sphere $\Isom_k^+$ has radius 
$1$ and is centred at the point with Heisenberg coordinates 
$[k\ell_A,kt_A]$, where $\ell_A$ and $t_A$ are as in \eqref{eq-def-lA}.
Similarly, the isometric sphere $\Isom_k^-$ has radius $1$ and centre 
the point with Heisenberg coordinates 
$[k\ell_A+\sqrt{\cos(\alpha_1)}e^{i\alpha_2},-\sin(\alpha_1)]$.
\end{prop}

\begin{proof}
As $A$ is unipotent and fixes $q_\infty$, it is a Cygan isometry, and thus 
preserves the radius of isometric spheres. This gives the part about radius. 
Moreover, it follows directly from Proposition \ref{normalAB} that $A^k$ acts 
on the boundary of $\HdC$ by left Heisenberg multiplication by $[k\ell_A,kt_A]$.
This gives the part about centres by a straightforward verification.
\end{proof}

The following proposition describes a symmetry of the family 
$\{\I_k^{\pm}\ :\ k\in\Z\}$ which will be useful in the study of 
intersections of the isometric spheres $\Isom_k^{\pm}$.

\begin{prop}\label{prop-special-symmetry}
Let $\varphi$ be the antiholomorphic isometry $S\iota=\iota T$, where 
$\iota$ is as in Proposition \ref{prop-involution}. Then $\varphi^2=A$, and $\varphi$
acts on the Heisenberg group as a screw 
motion preserving the affine line parametrised by 
\begin{equation}\label{defi-deltaphi}
 \Delta_\varphi=
\Bigl\lbrace \delta_\varphi(x)
=\Bigl[x+i\dfrac{\sqrt{\cos(\alpha_1)}\sin(\alpha_2)}{2},
x\sqrt{\cos(\alpha_1)}\sin(\alpha_2)-\dfrac{\sin(\alpha_1)}{2}\Bigr]\ : \ 
x\in\R\Bigr\rbrace.
\end{equation}
Moreover, $\varphi$ acts on isometric spheres as
$\varphi(\I_k^+)=\I_k^-$ and $\varphi(\I_k^-)=\I_{k+1}^+$ for all $k\in\Z$.
\end{prop}

\begin{proof}
Using the fact that $T=\iota S\iota$ we see that 
$A=ST=S\iota S\iota=\varphi^2$. Moreover
$\varphi(p_A)=S\iota(p_A)=S(p_B)=p_A$. Hence $\varphi$ is a Cygan
isometry. It follows by direct calculation that $\varphi$ 
sends $\delta_\varphi(x)$ to $\delta_\varphi(x+\ell_A/2)$, and so preserves
$\Delta_\varphi$. Moreover,
$$
\varphi(p_{BA})=S\iota(p_{BA})=S(p_{AB})=p_B,\quad
\varphi(p_B)=S\iota(p_B)=S(p_A)=p_{AB}.
$$
Hence $\varphi$ sends $\I_{-1}^-$ to $\I_0^+$ since it is a Cygan
isometry mapping the centre of $\I_{-1}^-$ to the centre of $\I_0^+$.
Similarly, $\varphi$ sends $\I_0^+$ to $\I_0^-$. The action on other
isometric spheres follows since $\varphi^2=A$.
\end{proof}

\subsection{A combinatorial restriction.}\label{section-def-Z}

The following section is the crucial technical part of our work. As most 
of the proofs are computational, we will omit many of them here; they will 
be provided in Section \ref{section-technicalities}. We are now going to 
restrict our attention to those parameters in the region $\mathcal{L}$ 
such that the three isometric spheres $\I_0^+=\I(S)$, $\I_{0}^-=\I(S^{-1})$ 
and $\I_{-1}^-=\I(T)$ have no triple intersection. We will describe the 
region we are interested in by an inequality on $\alpha_1$ and $\alpha_2$.  
Prior to stating it, let us fix a little notation.

We denote by $\alpha_{2}^{\lim}=\arccos\bigl(\sqrt{3/8}\bigr)$. Then
$\mathcal{G}\bigl(4\cos^2(0),4\cos^2(\pm\alpha_2^{\lim})\bigr)
=\mathcal{G}(4,3/2)=0$ and so the two points 
$(0,\pm\alpha_2^{\lim})$ are the two cusps of the curve $\mathcal{P}$ located 
on the vertical axis (see  figure \ref{peachgraph}). 
Now, let $\mathcal{R}$ be the rectangle (depicted in 
Figure \ref{picture-peachandstone}) defined by
\begin{equation}
 \mathcal{R}=\left\{(\alpha_1,\alpha_2)\ :\ |\alpha_1|\leqslant \pi/6,\, 
|\alpha_2|\leqslant \alpha_2^{\lim}\right\}.
\end{equation}
We remark that in Lemma \ref{lem-R-in-L} we will prove that when 
$(\alpha_1,\alpha_2)\in\mathcal{R}$, the commutator $[A,B]$ is non elliptic. 
This means that $\mathcal{R}$ is contained in  the closure of $\mathcal{L}$.

\begin{defi}\label{def-Z}
Let $\mathcal{Z}$ denote the subset of $\mathcal{R}$ where the triple 
intersection $\Isom_0^+\cap\Isom_{-1}^{-}\cap\Isom_0^-$ is empty.
\end{defi}

\begin{figure}
\begin{center}
 \scalebox{0.3}{\includegraphics{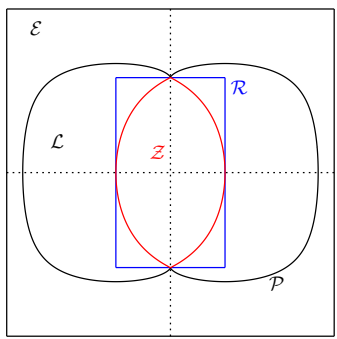}}
\caption{The parameter space, with the parabolicity curve $\mathcal{P}$ and 
the regions $\mathcal{E}$, $\mathcal{L}$. The region $\mathcal{Z}$ is the 
central region, which is contained in the rectangle $\mathcal{R}$.
\label{peachgraph}
\label{picture-peachandstone}}
\end{center}
\end{figure}

The following proposition characterises those points 
$(\alpha_1,\alpha_2)$ that lie in $\mathcal{Z}$. 

\begin{prop}\label{prop-triple-inter}
A parameter $(\alpha_1,\alpha_2)\in\mathcal{R}$ is in $\mathcal{Z}$ if and 
only if it satisfies
$$
\mathcal{D}(x_1^4,x_2^4)=
\mathcal{D}\bigl(4\cos^2(\alpha_1),4\cos^2(\alpha_1)\bigr)>0,
$$
where $\mathcal{D}$ is the polynomial given by 
$$
\mathcal{D}(x,y)=x^3y^3-9x^2y^2-27xy^2+81xy-27x-27.
$$
\end{prop}

The region $\mathcal{Z}$ is depicted in Figure \ref{picture-peachandstone}: 
it is the interior of the central region of the figure. In fact, 
$\mathcal{Z}$ is the region in all of $\mathcal{L}$ where 
$\Isom_0^+\cap\Isom_{-1}^{-}\cap\Isom_0^-$ is empty, but as proving this is
more involved, we restrict ourselves to the rectangle $\mathcal{R}$. 
This provides a priori bounds on the parameters $\alpha_1$ and $\alpha_2$ 
that will make our computations easier. We will prove 
Proposition \ref{prop-triple-inter} in Section \ref{section-no-triple}.  
It relies on Proposition \ref{prop-Z-in-L}, describing the set of points 
where $\mathcal{D}(x_1^4,x_2^4)>0$
and on Proposition \ref{prop-triple-inside}, which gives geometric 
properties of the triple intersection. Proofs of 
Proposition \ref{prop-Z-in-L} and Proposition \ref{prop-triple-inside} 
will be given in Section \ref{section-proof-ZinL} and
Section \ref{section-proof-triple-inside} respectively.

\begin{prop}\label{prop-Z-in-L} 
The region $\mathcal{Z}$ is an open topological disc in $\mathcal{R}$, 
symmetric about the axes and intersecting them
in the intervals $\{\alpha_2=0,\,-\pi/6<\alpha_1<\pi/6\}$ and
$\{\alpha_1=0,\,-\alpha_2^{\lim}<\alpha_2<\alpha_2^{\lim}\}$. 

Moreover, the intersection  of the closure of $\mathcal{Z}$ with the 
parabolicity curve $\mathcal{P}$ consists of the two points 
$(0,\pm\alpha_2^{\lim})$.
\end{prop}

\begin{prop}\label{prop-triple-inside}
\begin{enumerate}
 \item The triple intersection $\I_0^+\cap\I_0^-\cap\I_{-1}^-$ is contained 
in the meridian $\mathfrak{m}$ of $\I_0^+$ defined in geographical coordinates 
by $\beta=(\pi-\alpha_1)/2$.
\item If the triple intersection $\I_0^+\cap\I_0^-\cap\I_{-1}^-$ is non-empty, 
it contains a point in $\partial\HdC$.
\end{enumerate}
\end{prop}

The second part of Proposition \ref{prop-triple-inside} is not true for 
general triples of bisectors. It will allow us to restrict ourselves to 
the boundary of $\HdC$ to prove Proposition \ref{prop-triple-inter}. 
Restricting ourselves to the region $\mathcal{Z}$ will considerably 
simplify the combinatorics of the family of isometric spheres 
$\{\I_k^{\pm}\ :\ k\in\Z\}$. The following fact will be crucial in our study;
compare Figure \ref{fig-proj-isom}.

\begin{prop}\label{prop-pairwise-intersections}
Fix $(\alpha_1,\alpha_2)$ a point in $\mathcal{Z}$. Then the isometric sphere 
$\I_0^+$ is contained in the exterior of the isometric spheres 
$\I_k^\pm$ for all $k$, except for $\I_1^+$, $\I_{-1}^+$, $\I_0^-$ and 
$\I_{-1}^-$.  
\end{prop}
The proof of Proposition \ref{prop-pairwise-intersections} will be detailed 
in Section \ref{section-proof-pairwise}. We can give more information about 
the intersections $\I_0^\pm$ with these four other isometric spheres;
compare Figure \ref{fig-proj-isom}.

\begin{prop}\label{prop-interior}
If $(\alpha_1,\alpha_2)\in\mathcal{Z}$, then the intersection 
$\I_{-1}^-\cap\I_0^-$ is contained in the interior of $\I_0^+$.
\end{prop}

\begin{proof}
Since the point $p_B$ is the centre of $\I_0^+$, it lies in its interior. Moreover, $p_B$ lies on both 
$\I_{-1}^-$ and $\I_0^-$: indeed, $\la \bp_{AB},\bp_B\ra=\la \bp_{BA},\bp_B\ra=1$. By convexity of Cygan spheres
(see Proposition \ref{prop-inter-isom}), the intersection of the latter two 
isometric spheres is connected. This implies that $\I_{-1}^-\cap\I_0^-$ is 
contained in the interior of $\I_0^+$ for otherwise 
$\I_0^+\cap\I_{-1}^-\cap\I_0^-$ would not be empty.
\end{proof}

Using Proposition \ref{prop-special-symmetry}, applying powers of $\varphi$ 
to Propositions \ref{prop-pairwise-intersections} and \ref{prop-interior}
gives the following results describing all pairwise intersections. 

\begin{coro}\label{coro-pairwise-intersection}
Fix $(\alpha_1,\alpha_2)\in \mathcal{Z}$. Then for all $k\in\Z$:
\begin{enumerate}
\item $\Isom_k^+$ is contained in the exterior of all isometric spheres in 
$\{\I_k^{\pm}\ :\ k\in\Z\}$ except 
$\Isom_{k-1}^+$, $\Isom_{k-1}^-$, $\Isom_k^-$ and $\Isom_{k+1}^+$.
Moreover, $\Isom_k^+\cap\Isom_{k-1}^-\cap\Isom_k^-=\emptyset$ and
$\Isom_k^+\cap\Isom_{k-1}^+$ (respectively $\Isom_k^+\cap\Isom_{k+1}^+$)
is contained in the interior of $\Isom_{k-1}^-$ (respectively 
$\Isom_k^-$).
\item $\Isom_k^-$ is contained in the exterior of all isometric spheres in 
$\{\I_k^{\pm}\ :\ k\in\Z\}$ except $\Isom_{k-1}^-$, $\Isom_k^+$, $\Isom_{k+1}^+$,
and $\Isom_{k+1}^-$. Moreover, $\Isom_k^-\cap\Isom_k^-\cap\Isom_{k+1}^-=\emptyset$
and $\Isom_k^-\cap\Isom_{k-1}^-$ (respectively $\Isom_k^-\cap\Isom_{k+1}^-$)
is contained in the interior of $\Isom_k^+$ (respectively $\Isom_{k+1}^+$).
\end{enumerate}
\end{coro}

Proposition \ref{prop-pairwise-intersections} and 
Corollary \ref{coro-pairwise-intersection} are 
illustrated in Figure \ref{proj-spheres-in-Z}.

\begin{figure}
\begin{center}
\scalebox{0.8}{\includegraphics{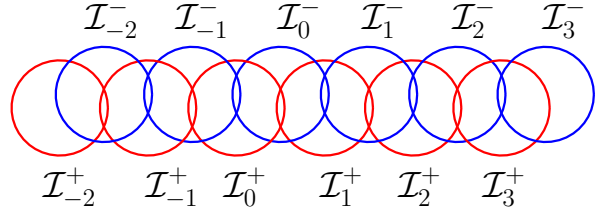}}
\caption{Vertical projections of the isometric spheres $\I_k^\pm$ for small 
values of $k$ at the point $(\alpha_1,\alpha_2)=(0.4,0.3)$\label{proj-spheres-in-Z}
\label{fig-proj-isom}}
\end{center}
\end{figure}

\section{Applying the Poincar\'e polyhedron theorem inside $\mathcal{Z}$.\label{section-poincare}}

\subsection{The Poincar\'e polyhedron theorem\label{section-general-Poincare}}
For the proof of our main result we need to use the Poincar\'e polyhedron
theorem for coset decompositions. The general principle of this result is 
described in Section 9.6 of \cite{Bear} in the context of the Poincar\'e disc. 
A generalisation to the case of $\HdC$ has already appeared in Mostow 
\cite{Most} and Deraux, Parker, Paupert \cite{DPP2}. In these cases it was 
assumed that the stabiliser of the polyhedron is finite.
In our case the stabiliser is the infinite cyclic group generated 
by the unipotent parabolic map $A$. 
There are two main differences from the version given in \cite{DPP2}.
First, we allow the polyhedron $D$ to have infinitely many facets, 
the stabiliser group $\Upsilon$ is also infinite, but we require that there 
are only finitely many $\Upsilon$-orbits of facets.  Secondly, we consider
polyhedra $D$ whose boundary intersects $\partial\HdC$ in an open set, which we
refer to as the ideal boundary of $D$. In fact, the version we need has many 
things in common with the version given by Parker, Wang and Xie \cite{PWX}.
A more general statement will appear in Parker's book \cite{Parbook}. 
In what follows we will adapt our statement of the Poincar\'e theorem
to the case we have in mind. 

\paragraph{The polyhedron and its cell structure}

Let $D$ be an open polyhedron in $\HdC$  and let
$\overline{D}$ denote its closure in $\overline{\HdC}=\HdC\cup\partial\HdC$.
We define the ideal boundary 
$\partial_\infty D$ of $D$ to be the intersection of $\overline{D}$
with $\partial\HdC$. This polyhedron has a natural cell 
structure which we suppose is locally finite inside $\HdC$.
We suppose that the facets of $D$ of all dimensions are piecewise 
smooth submanifolds of $\overline{\HdC}$. Let ${\mathcal F}_k(D)$ be 
the collection of facets of codimension $k$ having non-trivial intersection
with $\HdC$. We suppose that facets are closed subsets of
$\overline{\HdC}$. We write $f^\circ$ to denote the interior of a facet $f$, 
that is the collection of points of $f$ that are not contained in 
$\partial\HdC$ or any facet of a lower dimension (higher codimension). 
Elements of ${\mathcal F}_1(D)$ and ${\mathcal F}_2(D)$ are respectively called \emph{sides}
and \emph{ridges} of $D$. Since $D$ is a polyhedron, ${\mathcal F}_0(D)=\overline{D}$ and
each ridge in ${\mathcal F}_2(D)$ lies in exactly two sides in
${\mathcal F}_1(D)$. Similarly, the intersection of facets of $D$
with $\partial\HdC$ gives rise to a polyhedral structure on a subset of 
$\partial_\infty D$.
We let ${\mathcal {IF}}_k(D)$ denote the ideal facets of 
$\partial_\infty D$ of
codimension $k$ so that each facet in ${\mathcal {IF}}_k(D)$ is contained
in some facet of ${\mathcal F}_\ell(D)$ with $\ell<k$. 
In particular, we will also need to consider \emph{ideal vertices} in 
${\mathcal {IF}}_4(D)$. These are points of either the endpoints of
facets in ${\mathcal F}_3(D)$ or else they are points of $\partial\HdC$ 
contained in (at least) two facets of $D$ that do not intersect 
inside $\HdC$. Note that, since we have defined ideal facets to be
subsets of facets, it may be that $\partial\HdC$ contains points of 
$\partial_\infty D$ not contained in any ideal facet. In the case we consider, 
there will be one such point, namely the point at $\infty$ fixed by $A$.

\paragraph{The side pairing.}

We suppose that there is a \emph{side pairing} 
$\sigma:{\mathcal F}_1(D)\longrightarrow \Pu$ satisfying the following
conditions:
\begin{enumerate}
\item[(1)] For each side $s\in{\mathcal F}_1(D)$ with $\sigma(s)=S$ 
there is another side $s^-\in{\mathcal F}_1(D)$ so that $S$ maps $s$ 
homeomorphically onto $s^-$ preserving the cell structure. 
Moreover, $\sigma(s^-)=S^{-1}$.
Furthermore, if $s=s^-$ then $S=S^{-1}$ and $S$ is an involution.
In this case, we call $S^2=id$ a \emph{reflection relation}. 
\item[(2)] For each $s\in{\mathcal F}_1(D)$ with $\sigma(s)=S$ we have 
$\overline{D}\cap S^{-1}(\overline{D})=s$ and $D\cap S^{-1}(D)=\emptyset$.
\item[(3)] For each $w$ in the interior $s^\circ$ of $s$ 
there is an open neighbourhood $U(w)\subset \HdC$ of $w$ contained in 
$\overline{D} \cup S^{-1}(\overline{D})$.
\end{enumerate} 

In the example we consider, $D$ will be the Ford domain of a group.
In particular, each side $s$ will be contained in the
isometric sphere $\Isom(S)$ of $S=\sigma(s)$. Indeed,  
$s=\Isom(S)\cap\overline{D}$. By construction we have 
$S:\Isom(S)\longmapsto \Isom(S^{-1})$ and in this case
$s^-=\Isom(S^{-1})\cap\overline{D}$. The polyhedron $D$ will be
the (open) infinite sided polyhedron formed by the intersection 
of the exteriors of all the $\Isom(S)$ where $S=\sigma(s)$ and $s$
varies over ${\mathcal F}_1(D)$. By construction, the sides of $D$
are smooth hypersurfaces (with boundary) in $\HdC$.

Suppose that $D$ is invariant under a group $\Upsilon$ that 
is \emph{compatible} with the side pairing map in the sense that
for all $P\in\Upsilon$ and $s\in{\mathcal F}_1(D)$ we have
$P(s)\in{\mathcal F}_1(D)$ and $\sigma(Ps)=P\sigma(s)P^{-1}$.
We call the latter a \emph{compatibility relation}.
We suppose that there are finitely many $\Upsilon$-orbits of facets
in each ${\mathcal F}_k(D)$.  Since $P\in\Upsilon$ cannot fix a
side $s\in{\mathcal F}_1(D)$ pointwise, subdividing sides if necessary, 
we suppose that if $P\in\Upsilon$ maps a side in ${\mathcal F}_1(D)$ to itself
then $P$ is the identity. In particular, given sides $s_1$ and $s_2$ in 
${\mathcal F}_1(D)$, there is at most one $P\in\Upsilon$ sending $s_1$ to 
$s_2$. In the example of a Ford domain $\Upsilon$ will 
be $\Gamma_\infty$, the stabiliser of the point $\infty$ in the group $\Gamma$.

\paragraph{Ridges and cycle relations.}

Consider a ridge $r_1\in{\mathcal F}_2(D)$. 
Then, $r_1$ is contained in precisely two sides of $D$,
say $s_0^-$ and $s_1$. Consider the ordered triple $(r_1,s_0^-,s_1)$. 
The side pairing map $\sigma(s_1)=S_1$ sends $s_1$ to the side $s_1^-$ 
preserving its cell structure. In particular, $S_1(r_1)$ is a ridge of 
$s_1^-$, say $r_2$. Let $s_2$ be the other side containing $r_2$. Then we
obtain a new ordered triple $(r_2,s_1^-,s_2)$. Now apply $\sigma(s_2)=S_2$
to $r_2$ and repeat. Because there are only finitely many $\Upsilon$-orbits
of ridges, we eventually find an $m$ so that the ordered triple
$(r_{m+1},s_m^-,s_{m+1})=(P^{-1}r_1,P^{-1}s_0^-,P^{-1}s_1)$ for some $P\in\Upsilon$
(note that, by hypothesis, $P$ is unique).
We define a map $\rho:{\mathcal F}_2(D)\longrightarrow\Pu$ called
the \emph{cycle transformation} by
$\rho(r_1)=P\circ S_m\circ\cdots\circ S_1$.
(Note that for any ridge $r_1=s_0^-\cap s_1$, the cycle transformation map
$\rho(r_1)=R$ depends on a choice of one of the sides $s_0^-$ and $s_1$. If
we choose the other one then the ridge cycle becomes $R^{-1}$. This follows
from the fact that then $\sigma(s_j^-)=\sigma(s_j)^{-1}$ and from the
compatibility relations.) By construction, the cycle transformation 
$R=\rho(r_1)$ maps the ridge $r_1$ to itself setwise.
However, $R$ may not be the identity on $r_1$, nor on $\HdC$.
Nevertheless, we suppose that $R$ has order $n$. The relation
$R^n=id$ is called the \emph{cycle relation} associated to $r_1$.

Writing the cycle transformation $\rho(r_1)=R$ in terms of $P$ and the $S_j$, 
we let ${\mathcal C}(r_1)$ be the collection of suffix subwords of $R^n$. 
That is
$$
{\mathcal C}(r_1)=\Bigl\{S_j\circ\cdots\circ S_1\circ R^k\ :\ 
0\le j\le m-1,\ 0\le k\le n-1\Bigr\}.
$$
We say that \emph{the cycle condition} is satisfied at $r_1$ provided:
\begin{enumerate}
\item[(1)] 
$$
r_1=\bigcap_{C\in{\mathcal C}(r_1)}C^{-1}(\overline{D}).
$$
\item[(2)] If $C_1,\,C_2\in{\mathcal C}(r_1)$ with $C_1\neq C_2$ then
$C_1^{-1}(D)\cap C_2^{-1}(D)=\emptyset$.
\item[(3)] For each $w\in r_1^\circ$ there is an open neighbourhood
$U(w)$ of $w$ so that
$$
U(w)\subset \bigcup_{C\in{\mathcal C}(r_1)}C^{-1}(\overline{D}).
$$
\end{enumerate}

\paragraph{Ideal vertices and consistent horoballs.} 

 Suppose that the set ${\mathcal {IF}}_4(D)$ of ideal vertices of $D$ 
is non-empty. In our applications, there are no edges (that is 
${\mathcal F}_3(D)$ is empty) and the only ideal vertices arise 
as points of tangency between the ideal boundaries of ridges in
${\mathcal F}_2(D)$. In order to simplify our discussion below, we
will only treat this case. We require that 
there is a system of \emph{consistent horoballs} based at the ideal
vertices and their images under the side pairing maps (see page 152
of \cite{EP} for definition). 
For each ideal vertex $\xi\in{\mathcal {IF}}_4(D)$, 
the consistent horoball $H_\xi$ is a horoball based at $\xi$ with the 
following property. Let $\xi\in{\mathcal {IF}}_4(D)$ and let 
$s\in{\mathcal F}_1(D)$ be a side with $\xi\in s$.  
Then  the side pairing $S=\sigma(s)$ maps $\xi$ to a point $\xi^-$ in $s^-$.  
Note that $\xi^-$ is not necessarily an ideal vertex (since it could be that $\xi$ is 
a point of tangency between two sides whose closures in $\overline{\HdC}$
are otherwise disjoint and $\xi^-$ may be a point of tangency between 
two nested bisectors only one of which contributes a side of $D$). In our case
this does not happen and so we may assume $\xi^-$ also lies
in ${\mathcal {IF}}_4(D)$ and so has a consistent horoball $H_{\xi^-}$.
In order for these horoballs to form a system of consistent horoballs we 
require that for each ideal vertex $\xi$ and each side $s$ with $\xi\in s$
the side pairing map $\sigma(s)$ should map the horoball $H_\xi$ onto 
the horoball $H_{\xi^-}$. In particular, any cycle of side pairing maps
sending $\xi$ to itself must also send $H_\xi$ to itself.

\paragraph{Statement of the Poincar\'e polyhedron theorem.}

We can now state the version of the Poincar\'e polyhedron theorem that we
need (compare \cite{Most} or \cite{DPP2}).

\begin{theo}\label{theo-Poincare-cosets}
Let $D$ be a smoothly embedded polyhedron $D$ in $\HdC$ together with a side
pairing $\sigma:{\mathcal F}_1(D)\longrightarrow \Pu$. Let $\Upsilon<\Pu$
be a group of automorphisms of $D$ compatible with the side pairing
and suppose that each ${\mathcal F}_k(D)$ contains finitely many 
$\Upsilon$-orbits.  Fix a presentation for $\Upsilon$ with generating set
${\mathcal P}_\upsilon$ and relations ${\mathcal R}_\Upsilon$.
Let $\Gamma$ be the group generated by ${\mathcal P}_\Upsilon$ and the side 
pairing maps $\{\sigma(s)\}$. Suppose that the cycle condition is satisfied
for each ridge in ${\mathcal F}_2(D)$ and that there is a system of 
consistent horoballs at all the ideal vertices of $D$ (if any).
Then: 
\begin{enumerate}
\item[(1)] The images of $D$ under the cosets of $\Upsilon$ in $\Gamma$ 
tessellate $\HdC$. That is $\HdC\subset\bigcup_{A\in\Gamma}A(\overline{D})$ and 
$D\cap A(D)=\emptyset$ for all $A\in\Gamma-\Upsilon$.
\item[(2)] The group $\Gamma$ is discrete and a fundamental domain for its 
action on $\HdC$ is obtained from the intersection of $D$ with a fundamental 
domain for $\Upsilon$.
\item[(3)] A presentation for $\Gamma$ (with respect to the generating set 
${\mathcal P}_\Upsilon\cup\{\sigma(s)\}$) has the 
following set of relations: the relations ${\mathcal R}_\Upsilon$ in
$\Upsilon$, the compatibility relations between $\sigma$ and $\Upsilon$, 
the reflection relations and the cycle relations.
\end{enumerate}
\end{theo}

\subsection{Application to our examples.}\label{sec-apply-poincare}
We are now going to apply Theorem \ref{theo-Poincare-cosets} to the group 
generated by $S$ and $A$. Explicit matrices for these transformations are 
provided in equations \eqref{normalAB} and \eqref{S}. Our aim is to prove:

\begin{theo}\label{thm-main-Z}
Suppose that $(\alpha_1,\alpha_2)$ is in $\mathcal{Z}$. That is,
$\mathcal{D}\bigl(4\cos^2(\alpha_1),4\cos^2(\alpha_1)\bigr)>0$,
where $\mathcal{D}(x,y)$ is the polynomial defined in
Proposition \ref{prop-triple-inter}.
Then the group $\Gamma=\langle S,A\rangle$ associated to the parameters
$(\alpha_1,\alpha_2)$ is discrete and has the presentation
\begin{equation}\label{eq-pres-poin}
\langle S,\,A\ :\ S^3=(A^{-1}S)^3=id\rangle.
\end{equation}
\end{theo}

We obtain the presentation $\langle S,\,T\ :\ S^3=T^3=id\rangle$ by changing generators to $S$ and $T=A^{-1}S$.

\paragraph{Definition of the polyhedron and its cell structure. }
The infinite polyhedron we consider is the intersection of the exteriors
of all the isometric spheres in $\{\I_k^\pm\ :\ k\in\Z\}$.

\begin{defi}\label{defi-F}
We call $D$  the intersection of the exteriors of all isometric spheres 
$\I_k^{+}$ and $\I_k^-$ with centres $A^kS^{-1}(q_\infty)$ and $A^kS(q_\infty)$ 
respectively :
\begin{equation}\label{eq-D}
D=\Bigl\{q\in \HdC\ :\ d_{\rm Cyg}\bigl(q,A^kS^{\pm 1}(q_\infty)\bigr)> 1 
\hbox{ for all }k\in\Z\Bigr\}. 
\end{equation}
The set of sides of $D$ is ${\mathcal F}_1(D)=\{s_k^+,\,s_k^-\ :\ k\in\Z\}$ 
where $s_k^+=\Isom_k^+\cap \overline{D}$ and $s_k^-=\Isom_k^-\cap \overline{D}$. 
\end{defi}

Using Corollary \ref{coro-pairwise-intersection} we can completely
describe $s_k^+$ and $s_k^-$.

\begin{prop}\label{prop-describe sides}
The side $s_k^\pm$ is topologically a solid cylinder 
in $\HdC\cup\partial\HdC$. More precisely, $s_k^\pm$ is a product 
$D\times [0,1]$ where for each $t\in[0,1]$, the fibre $D\times \{t\}$ is 
homeomorphic to a closed disc in $\overline{\HdC}$ whose boundary is 
contained in $\partial\HdC$. The intersection of $\partial s_k^+$ 
(respectively $\partial s_k^-$) with $\HdC$ is the disjoint union of 
the topological discs $s_k^+\cap s_{k-1}^-$ and $s_k^+\cap s_k^-$ 
(respectively $s_k^-\cap s_k^+$ and $s_k^-\cap s_{k+1}^-$).
\end{prop}

\begin{proof}
Since $s_k^+$ is contained in $\I_k^+$, its only possible intersections
with other sides are contained in $\I_{k-1}^+$, $\I_{k-1}^-$, $\I_{k+1}^+$ and
$\I_{k+1}^-$ by Corollary \ref{coro-pairwise-intersection}. Since $\I_k^+\cap\I_{k-1}^+$ and $\I_k^+\cap\I_{k+1}^+$ are
contained in the interiors of other isometric spheres, the intersections 
$s_k^+\cap s_{k-1}^+$ and $s_k^+\cap s_{k+1}^+$ are empty. 
Also, $\I_k^+\cap\I_{k-1}^-\cap\I_k^-=\emptyset$ and so 
$s_k^+\cap s_{k-1}^-$ and $s_k^+\cap s_k^-$ are disjoint.
Since isometric spheres are topological balls and their pairwise intersections
are connected, the description of $s_k^+$ follows. A similar argument describes
$s_k^-$.
\end{proof}

The side pairing $\sigma:{\mathcal F}_1(D)\longrightarrow\Pu$ is defined by
\begin{equation}\label{side-pairing}
\sigma(s_k^+)=A^kSA^{-k},\qquad \sigma(s_k^-)=A^kS^{-1}A^{-k}.
\end{equation}
Let $\Upsilon=\la A\ra$ be the infinite cyclic group generated by $A$. 
By construction the side pairing $\sigma$ is compatible with $\Upsilon$. 
Furthermore, using Proposition \ref{prop-describe sides} the set of ridges
is ${\mathcal F}_2(D)=\{r_k^+,\,r_k^-\ :\ k\in\Z\}$ where
$r_k^+=s_k^+\cap s_k^-$ and $r_k^-=s_k^+\cap s_{k-1}^-$. 
We can now verify that $\sigma$ satisfies the first condition of being
a side pairing.

\begin{prop}\label{prop-side-homeo}
The side pairing map $\sigma(s_k^+)=A^kSA^{-k}$ is a homeomorphism from
$s_k^+$ to $s_k^-$. Moreover $\sigma(s_k^-)$ sends $r_k^+=s_k^+\cap s_k^-$
to itself and sends $r_k^-=s_k^+\cap s_{k-1}^-$ to $r_{k+1}^-=s_k^-\cap s_{k+1}^+$.
\end{prop}

\begin{proof}
By applying powers of $A$ we need only need to consider the case where $k=0$.
First, the ridge $r_0^+=s_0^+\cap s_0^-=\I(S)\cap\I(S^{-1})$ is defined by the triple equality
\begin{equation}\label{tripleSSm}
|\la \bz,\bq_\infty\ra|= |\la \bz, S^{-1}\bq_\infty\ra|=|\la \bz, S\bq_\infty\ra|.
\end{equation}
The map $S$ cyclically permutes $p_B=S^{-1}(q_\infty)$, $p_A=q_\infty$, 
$p_{AB}=S(q_\infty)$, and so maps $r_0^+$ to itself. Similarly, consider 
$r_0^-=s_0^+\cap s_{-1}^-$. The side pairing map $S$
sends $A^{-1}S(q_\infty)$, the centre of $\I_{-1}^-$, to
$$
S(A^{-1}S)(q_\infty)=S(T^{-1}S^{-1})S(q_\infty)=ST^2(q_\infty)
=(ST)S^{-1}(ST)(q_\infty)=AS^{-1}(q_\infty),
$$
which is the centre of $\I_1^+$, where we have used 
$A^{-1}=T^{-1}S^{-1}$, $T^{-1}=T^2$ and $ST(q_\infty)=q_\infty$.
Therefore $r_0^-=s_0^+\cap s_{-1}^-$ is sent to $r_1^-=s_0^-\cap s_1^+$
as claimed. The rest of the result follows from our description of $s_k^\pm$ in
Proposition \ref{prop-describe sides}.
\end{proof} 

\paragraph{Local tessellation. }
We now prove local tessellation around the sides and ridges of $D$.
\begin{itemize}
 \item[$s_k^\pm$.] Since $\sigma(s_k^\pm)=A^kS^{\pm 1}A^{-1}$ sends the exterior 
of $\I_k^\pm$ to the
interior of $\I_k^\mp$ we see that $D$ and $A^kS^{\pm 1}A^{-k}(D)$ have
disjoint interiors and cover a neighbourhood of each point in $s_k^{\mp}$.
Together with Proposition \ref{prop-side-homeo} this means $\sigma$
satisfies the three conditions of being a side pairing.
\item[$r_0^+$.]  Consider the case of $r_0^+=s_0^+\cap s_0^-=\I(S)\cap\I(S^{-1})$,
which is given by \eqref{tripleSSm}. Observe that $r_0^+$ is mapped to itself 
by $S$. Using Proposition \ref{prop-side-homeo},
we see that when constructing the cycle transformation 
for $r_0^+$ we have one ordered triple $(r_0^+,s_0^-,s_0^+)$ and the cycle 
transformation $\rho(r_0^+)=S$. The cycle relation is $S^3=id$ and 
${\mathcal C}(r_0^+)=\{id,\,S,\,S^2\}$. Consider an open neighbourhood
$U_0^+$ of $r_0^+$ but not intersecting any other ridge. The intersection
of $D$ with $U_0^+$ is the same as the intersection of $U_0^+$ with 
the Ford domain $D_S$ for the order three group $\langle S\rangle$. 
Since $S$ has order 3 this Ford domain is the intersection of the exteriors 
of $\I(S)$ and $\I(S^{-1})$. For $z$ in  $D_S$, 
$|\langle{\bf z},{\bf q}_\infty\rangle|$ is the smallest of the three
quantities in \eqref{tripleSSm}. Applying $S=\sigma(s_0^+)$ and
$S^{-1}=\sigma(s_0^-)$ gives regions $S(D_S)$ and $S^{-1}(D_S)$ where one of the other two
quantities is the smallest. Therefore $U_0^+\cap S(U_0^+)\cap S(U_0^-)$
is an open neighbourhood of $r_0^+$ contained in $D\cup S(D)\cup S^{-1}(D)$.
This proves the cycle condition at $r_0^+$.
\item[$r_0^-$.] Now consider $r_0^-=s_0^+\cap s_{-1}^-$. When constructing 
the cycle transformation for $r_0^-$ we start with the ordered triple 
$(r_0^-,s_{-1}^-,s_0^+)$. Applying $S=\sigma(s_0^+)$ to $r_0^-$ gives the 
ordered triple $(r_1^-,s_0^-,s_1^+)$, which is simply $(Ar_0^-,As_{-1}^-,As_0^+)$. Thus 
the cycle transformation of $r_0^-$ is $\rho(r_0^-)=A^{-1}S=T^{-1}$, which has order 3.
Therefore the cycle relation is $(A^{-1}S)^3=id$, 
and ${\mathcal C}(r_0^-)=\{id,\,A^{-1}S,\,(A^{-1}S)^2\}$.  
Noting that $\I_0^+$ has centre $S^{-1}(q_\infty)S^{-1}A(q_\infty)=T(q_\infty)$ and 
$\I_{-1}^-$ has centre $A^{-1}S(q_\infty)=T^{-1}(q-\infty)$ we see
$\I_0^+=\I(T^{-1})$ and $\I_0^-=\I(T)$. Therefore a similar argument
involving the Ford domain for $\langle T\rangle$ shows that the cycle 
condition is satisfied at $r_0^-$.
\item[$r_k^\pm$.] Using compatibility of the side pairings with $\Upsilon=\langle A\rangle$, 
we see that $\rho(r_k^+)=A^kSA^{-k}$ with cycle relation 
$(A^kSA^{-k})^3=A^kS^3A^{-k}=id$ and that the cycle condition is satisfied 
at $r_k^+$. Likewise, $\rho(r_k^-)=A^k(A^{-1}S)A^{-k}=A^{k-1}SA^{-k}$ with 
cycle relation $(A^{k-1}SA^{-k})^3=A^k(A^{-1}S)A^{-k}=id$ and the cycle 
condition is satisfied at $r_k^-$.
\end{itemize}

This is sufficient to prove Theorem \ref{thm-main-Z} by applying the Poincar\'e 
polyhedron theorem when $D$ has no ideal vertices, that is to all groups 
$\Gamma$ in the interior of ${\mathcal Z}$.
In particular, $\Gamma$ is generated by the generator $A$ of
$\Upsilon$ and the side pairing maps. Using the compatibility relations,
there is only one side pairing map up to the action of $\Upsilon$, namely $S$.
There are no reflection relations, and (again up to the action of $\Upsilon$)
the only cycle relations are $S^3=id$ and $(A^{-1}S)^3=id$. Thus the
Poincar\'e polyhedron theorem gives the presentation \eqref{eq-pres-poin}.
This completes the proof of Theorem \ref{thm-main-Z}.

For groups on the boundary of $\mathcal{Z}$ the same result is also 
true. This follows from the fact (Chuckrow's theorem): the algebraic limit 
of a sequence of discrete and faithful representations of a non virtually 
nilpotent group in Isom($\HnC$) is discrete and faithful (see for instance 
Theorem 2.7 of \cite{CooLonThi} or \cite{Mart} for a more general result 
in the frame of negatively curved groups).

We do not need to apply the Poincar\'e
polyhedron theorem for these groups. However, to describe the manifold at 
infinity for the limit groups, we will need to know a fundamental domain, 
and we will have to go through a similar analysis in the next section.

\section{The limit group.}\label{section-limit-group}

In this section, we consider the group $\Gamma^{\lim}$, and unless otherwise 
stated, the parameters $\alpha_1$ and $\alpha_2$ will always be assumed to 
be equal to $0$ and $\alpha_2^{\lim}$ respectively. We know already that 
$\Gamma^{\lim}$ is discrete and isomorphic to $\Z_3\ast\Z_3$. Our goal is to 
prove that its manifold at infinity is homeomorphic to the complement of 
the Whitehead link. For these values of the parameters, the maps $S^{-1}T$ and $ST^{-1}$ are unipotent parabolic 
(see the results of Section \ref{section-type-commutator}), and we 
denote by $V_{S^{-1}T}$ and $V_{ST^{-1}}$ respectively the sets of (parabolic) 
fixed points of conjugates of 
$S^{-1}T$ and $ST^{-1}$ by powers of $A$. 
\begin{enumerate}
\item As in the previous section, we apply the Poincar\'e polyhedron theorem, 
this time to the group $\Gamma^{\lim}$. We obtain an infinite $A$-invariant 
polyhedron, still denoted $D$, which is a fundamental domain for $A$-cosets. 
This polyhedron is slightly more complicated than the one in the previous 
section due to the appearance of ideal vertices that are the points in 
$V_{S^{-1}T}$ and $V_{ST^{-1}}$.
\item We analyse the combinatorics of the ideal boundary $\partial_\infty D$ 
of this polyhedron. More precisely, we will see that the quotient of 
$\partial_\infty D\setminus \left(\{p_A\}\cup V_{S^{-1}T}\cup V_{ST^{-1}}\right)$ 
by the action of the group $\la S, T\ra$ is homeomorphic the complement 
of the Whitehead link, as stated in Theorem \ref{theo-CRWLC}. 
\end{enumerate}

\subsection{Matrices and fixed points.}

Before going any further, we provide  specific expressions for the various objects we consider
at the  limit point. When $\alpha_1=0$ and $\alpha_2=\alpha_2^{\lim}$, the map 
$\varphi$ described in Proposition \ref{prop-special-symmetry} is given in 
Heisenberg coordinates by
\begin{equation}\label{limit-phi}
\varphi : [z,t]\longmapsto 
\left[\overline{z}+\sqrt{3/8}+i\sqrt{5/8},-t+x\sqrt{5/2}+y\sqrt{3/2}\right].
\end{equation}
In particular its invariant line $\Delta_\varphi$ is parametrised by 
\begin{equation}\label{param-inv-phi}
\Delta_\varphi
=\Bigl\lbrace\delta_{\varphi}(x)=\left[ \,x+i\sqrt{5/32},\,x\sqrt{5/8}\,\right]
\ : \ x\in\R\Bigr\rbrace.
\end{equation}
The parabolic map $A=\varphi^2$ acts on $\Delta_\varphi$ as  
$A:\delta_\varphi(x)\longmapsto \delta_\varphi(x+\sqrt{3/2})$. As a matrix it
is given by 
\begin{equation}\label{A-lim}
A=\begin{bmatrix}
    1 & -\sqrt{3} & -3/2+i\sqrt{15}/2\\0 & 1 & \sqrt{3}\\ 0&0&1
   \end{bmatrix}.
\end{equation}
We can decompose $A$ into the product of regular elliptic maps $S$ and $T$:
\begin{equation*}
S=\begin{bmatrix}1 & \sqrt{3}/2-i\sqrt{5}/2 & -1 \\
   -\sqrt{3}/2-i\sqrt{5}/2 & -1 & 0\\
   -1 & 0 & 0 
  \end{bmatrix},\quad
 T=\begin{bmatrix} 0 &0 &-1\\ 0& -1& -\sqrt{3}/2+i\sqrt{5}/2\\ 
-1 &\sqrt{3}/2+i\sqrt{5}/2 & 1 
\end{bmatrix}
\end{equation*}
These maps cyclically permute $(p_A,p_{AB},p_B)$ and $(p_A,p_B,p_{BA})$ where
\begin{equation}\label{eq-pslim}
\bp_A=\begin{bmatrix}1 \\ 0 \\ 0\end{bmatrix}, \quad 
\bp_B=\begin{bmatrix}0 \\ 0 \\ 1\end{bmatrix}, \quad
\bp_{AB}=\begin{bmatrix} -1\\ \sqrt{3}/2+i\sqrt{5/2}\\1\end{bmatrix}, \quad
\bp_{BA}=\begin{bmatrix} -1\\-\sqrt{3}/2+i\sqrt{5}/2\\1\end{bmatrix}.
\end{equation}
Using $\alpha_1=0$, we will occasionally
use the facts from Proposition \ref{prop-dec-ST} that $(S,T)$
is $\C$-decomposable and $(A,B)$ is $\R$-decomposable.

As mentioned above, in the group $\Gamma^{\lim}$ the elements 
$ST^{-1}$, $S^{-1}T$, $TST$, $STS$ and the commutator $[A,B]=(ST^{-1})^3$ 
are unipotent parabolic. For future reference, we provide here lifts of their 
fixed points, both as vectors in $\C^3$ and in terms of geographical 
coordinates $g(\alpha,\beta)$ (we omit the $w$ coordinates: since 
we are on the boundary at infinity, it is equal to $\sqrt{2\cos\alpha}$).
\begin{eqnarray}\label{fixed-points-limit}
\bp_{ST^{-1}} & = &\begin{bmatrix}
-1/4+i\sqrt{15}/4\\\sqrt{3}/4+i\sqrt{5}/4\\1\end{bmatrix}
\ =\ {\bf g}\left(\arccos(1/4),\pi/2\right), \nonumber\\
 \bp_{S^{-1}T} & = &\begin{bmatrix}
-1/4-i\sqrt{15}/4\\-\sqrt{3}/4+i\sqrt{5}/4\\1\end{bmatrix}
\ =\ {\bf g}\left(-\arccos(1/4),\pi/2\right),\nonumber\\
 \bp_{TST} & = &\begin{bmatrix}
-1\\-3\sqrt{3}/4+i\sqrt{5}/4\\1\end{bmatrix}
\ =\ {\bf g}\left(0,-\arccos\left(\sqrt{27/32}\right)\right),\nonumber \\
 \bp_{STS} & = &\begin{bmatrix}
-1\\3\sqrt{3}/4+i\sqrt{5}/4\\1\end{bmatrix}
\ =\ {\bf g}\left(0,\arccos\left(\sqrt{27/32}\right)\right). 
\end{eqnarray}

It follows from \eqref{limit-phi} that $\varphi$ acts on these parabolic
fixed points as follows: 
\begin{equation}\label{action-phi}
 \cdots p_{T^{-1}STST}\overset{\varphi}{\longrightarrow} 
p_{TST}\overset{\varphi}{\longrightarrow} 
p_{S^{-1}T}\overset{\varphi}{\longrightarrow} 
p_{ST^{-1}}\overset{\varphi}{\longrightarrow} 
p_{STS}\overset{\varphi}{\longrightarrow} p_{STSTS^{-1}}\cdots 
\end{equation}

\subsection{The Poincar\'e theorem for the limit group.}
\label{section-Poincare-limit}

The limit group has extra parabolic elements. Therefore, in order to
apply the Poincar\'e theorem, we must construct a system of consistent
horoballs at these parabolic fixed points (see Section \ref{section-general-Poincare}).

\begin{lem}\label{lem-tangency}
The isometric spheres $\I_1^+$  and $\I_{-1}^-$ are tangent at $p_{ST^{-1}}$.
The isometric spheres $\I_{-1}^+$ and $\I_0^-$ are tangent at $p_{S^{-1}T}$.
\end{lem}

\begin{proof}
It is straightforward to verify that 
$|\la\bp_{ST^{-1}},\bp_{BA}\ra|=|\la\bp_{ST^{-1}},A(\bp_{B})\ra|=1$, and therefore 
$p_{ST^{-1}}$ belongs to both $\I_{-1}^-$ and $\I_1^+$. Projecting vertically 
(see Remark \ref{rem-proj-cygan-sphere}), we see that the projections of $\I_{-1}^-$ and $\I_1^+$ are tangent 
discs and as they are strictly convex, their intersection contains at most one point. This gives the result. 
The other tangency is along the same lines.
\end{proof}

A consequence of Lemma \ref{lem-tangency} is that the
parabolic fixed points are tangency points of isometric spheres.
The following lemma is proved in 
Section \ref{section-proof-triple-inside}.

\begin{lem}\label{lem-triple-limit}
For the group $\Gamma^{\lim}$ the triple 
intersection $\I_0^+\cap\I_0^{-}\cap\I_{-1}^-$ contains exactly 
two points, namely the parabolic fixed points $p_{ST^{-1}}$ and $p_{S^{-1}T}$.
\end{lem}

\begin{figure}
\begin{center}
\begin{tabular}{cc}
 \includegraphics[scale=0.6]{dessin-bigon-3.eps} & 
\includegraphics[scale=0.6]{dessin-bigon-4.eps}\\
  \end{tabular}
\end{center}
\caption{Two realistic views of the isometric spheres $\I_0^+$, $\I_1^+$ and $\I_{0}^-$ for the limit group $\Gamma^{\lim}$. 
The thin bigon is $\mathcal{B}_0^+$ (defined in Proposition \ref{prop-part-outside}). 
Compare with Figures \ref{fig-faces-in-IS} and \ref{c0plus} \label{realistic1}}
\end{figure}

Applying powers of $\varphi$, we see that these triple intersections are 
actually quadruple intersections of sides and triple intersections of ridges.

\begin{coro}\label{cor-ideal-vertices}
The parabolic fixed point $A^k(p_{ST^{-1}})$ lies on
$\I_{k-1}^-\cap \I_k^+\cap \I_k^-\cap \I_k^+$. In particular, it is
the triple ridge intersection $r_k^-\cap r_k^+\cap r_{k+1}^-$.
Similarly, $A^k(p_{S^{-1}T})$ lies on 
$\I_{-1}^+\cap \I_{-1}^-\cap \I_0^+ \cap \I_0^-$. In particular it is
$r_{k-1}^+\cap r_k^-\cap r_k^+ $.
\end{coro}

To construct a system of consistent horoballs at the 
parabolic fixed points we must investigate the action of the side 
pairing maps on them. 
First, 
$p_{S^{-1}T}\in\Isom_{-1}^+\cap\Isom_{-1}^-\cap\Isom_0^+\cap\Isom_0^-$, we have
\begin{eqnarray*}
\sigma(s_{-1}^+)=A^{-1}SA:p_{S^{-1}T} & \longmapsto & p_{T^{-1}STST}, \\
\sigma(s_{-1}^-)=A^{-1}S^{-1}A:p_{S^{-1}T} &\longmapsto &p_{TST}, \\
\sigma(s_0^+)=S:p_{S^{-1}T} &\longmapsto &p_{ST^{-1}}, \\
\sigma(s_0^-)=S^{-1}:p_{S^{-1}T}&\longmapsto&p_{STS}.
\end{eqnarray*}
Likewise $p_{ST^{-1}}\in\Isom_{-1}^-\cap\Isom_0^+\cap\Isom_0^-\cap\Isom_1^+$. 
We have
\begin{eqnarray*}
\sigma(s_{-1}^-)=A^{-1}S^{-1}A:p_{ST^{-1}} &\longmapsto & A^{-2}(p_{ST^{-1}}), \\
\sigma(s_0^+)=S:p_{ST^{-1}}&\longmapsto & p_{STS}, \\
\sigma(s_0^-)=S^{-1}:p_{ST^{-1}}&\longmapsto &p_{S^{-1}T}, \\
\sigma(s_1^+)=ASA^{-1}:p_{ST^{-1}}&\longmapsto& A^2(p_{ST^{-1}}).
\end{eqnarray*}

We can combine these maps to show how the points $A^k(p_{ST^{-1}})$ and 
$A^k(p_{S^{-1}T})$
are related by the side pairing maps. This leads to an infinite graph, 
a section of which is:
\begin{equation}\label{eq-graph} 
\xymatrix{ 
 \ar[rrr]^-{A^{-1}SA} &&& 
p_{ST^{-1}} \ar[rrrr]^-{ASA^{-1}} \ar[rrd]^-{S} 
&& && 
A^2(p_{ST^{-1}}) \ar[r] &  \\
& p_{T^{-1}STST} \ar[l] \ar[d]_-{A^{-1}SA} &&
p_{S^{-1}T} \ar[ll]_-{A^{-1}SA} \ar[u]^-{S} && 
p_{STS} \ar[ll]^-{S} \ar[d]_-{ASA^{-1}} && 
A^2(p_{S^{-1}T}) \ar[ll]_-{ASA^{-1}} \ar[u]^-{A^2SA^{-2}} &
\ar[l] \\ 
\ar[r] & 
p_{TST} \ar[rrrr]_-{S} \ar[urr]_-{A^{-1}SA} && &&
p_{STSTS^{-1}} \ar[rrr]_-{A^2SA^{-2}} 
\ar[urr]_-{ASA^{-1}} &&& \\
}
\end{equation}
From this it is clear that all the cycles in the graph 
\eqref{eq-graph} are generated 
by triangles and quadrilaterals. Up to powers of $A$, the
triangles lead to the word $S^3$, which is the identity.
Up to powers of $A$ the quadrilaterals lead to words cyclically
equivalent to the one coming from:
\begin{displaymath} 
\xymatrix{ 
p_{S^{-1}T} \ar[rr]^-{S^{-1}} && p_{STS} \ar[rr]^-{ASA^{-1}} &&
p_{STSTS^{-1}} \ar[rr]^-{S^{-1}} && p_{TST} \ar[rr]^{A^{-1}SA} && p_{S^{-1}T}
}
\end{displaymath}
In other words, $p_{S^{-1}T}$ is fixed by
$(A^{-1}SA)(S^{-1})(ASA^{-1})(S^{-1})=(T^{-1}S)^3$. This is parabolic and
so preserves all horoballs based at $p_{S^{-1}T}$.

\begin{figure}
 \begin{center}
\begin{tabular}{cc}
 \scalebox{0.4}{ \includegraphics{projection-verticale.eps}} & 
 \scalebox{0.4}{ \includegraphics{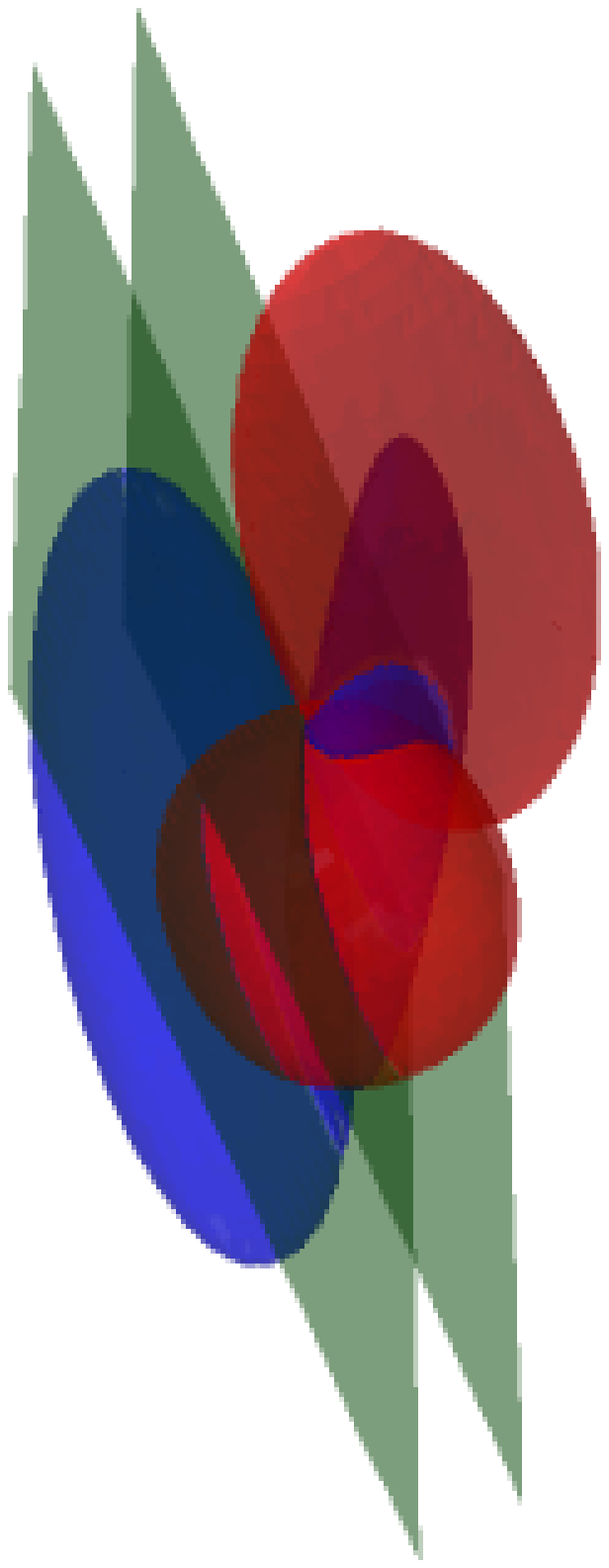}}
\end{tabular} 
\end{center}
\caption{Vertical projection and realistic view of the isometric spheres and the fans $F_0$ and $F_{-1}$ for the parameter 
values $\alpha_1=0$, $\alpha_2=\alpha_2^{\lim}$. Compare with 
Figure \ref{fig-proj-isom}.
\label{fig-vertiproj}}
\end{figure}

Therefore, we can define a system of horoballs as follows. Let $U_0^+$
be a horoball based at $p_{S^{-1}T}$, disjoint from the closure of any 
side not containing $p_{S^{-1}T}$ in its closure. Now define
horoballs $U_k^+$ and $U_k^-$ by applying the side pairing maps to
$U_0^+$. Since every cycle in the graph \eqref{eq-graph} gives rise 
either to the
identity map or to a parabolic map, this process is well defined and
gives rise to a consistent system of horoballs. Therefore we can
apply the Poincar\'e polyhedron theorem for the two limit groups.
Using the same arguments as we did for groups in the interior of
${\mathcal Z}$, we see that $\Gamma$ has the presentation \eqref{eq-pres-poin}.

\subsection{The boundary of the limit orbifold.}

\begin{theo}\label{theo-CRWLC}
The manifold at infinity of the group $\Gamma^{\lim}$ is homeomorphic 
to the Whitehead link complement.
\end{theo}

The ideal boundary of $D$ is made up of those pieces of the isometric spheres 
$\I_{k}^\pm$ that are outside all other isometric spheres in 
$\{\I_k^{\pm}\ :\ k\in\Z\}$. 
Recall that the (ideal boundary of) the side $s_k^\pm$ is the part of 
$\partial \Isom_k^\pm$ which is outside (the ideal boundary of) 
all other isometric spheres. In this section, when we speak of sides and ridges 
we implicitly mean their intersection with $\partial\HdC$.

We will see that each isometric sphere in $\{\I_k^{\pm}\ :\ k\in\Z\}$ 
contributes a side $s_k^\pm$ made up of one quadrilateral, denoted by 
$\mathcal{Q}_k^\pm$ and one bigon $\mathcal{B}_k^\pm$. A very similar 
configuration of isometric spheres has been observed by Deraux and Falbel 
in \cite{DF8}. We begin by analysing the contribution of $\Isom_0^+$.

\begin{prop}\label{prop-part-outside}
The side $(s_0^+)^\circ$ of $D$ has two connected components.
\begin{enumerate}
 \item One of them is a quadrilateral, denoted $\mathcal{Q}_0^+$, 
whose vertices are 
points $p_{ST^{-1}}$, $p_{S^{-1}T}$, $p_{STS}$ and $p_{TST}$ (all of which are
parabolic fixed points) 
 \item The other is a bigon, denoted $\mathcal{B}_0^+$, whose vertices are 
$p_{ST^{-1}}$ and $p_{S^{-1}T}$
\end{enumerate}
\end{prop}

\begin{figure}
\begin{center}
\scalebox{0.65}{\includegraphics{faceonS2ter.eps}}
\end{center}
\caption{Intersections of the isometric spheres $\Isom_0^-$, $\Isom_{-1}^-$, 
$\Isom_1^+$ and $\Isom_{-1}^+$ with $\Isom_0^+$ in the boundary of $\HdC$, 
viewed in geographical coordinates. Recall that 
$r_0^+=\Isom_0^+\cap\Isom_0^-$ and $r_0^-=\Isom_0^+\cap\Isom_{-1}^-$. 
Here $\alpha\in[-\pi/2,\pi/2]$ is the vertical coordinates, and $\beta\in[-\pi,\pi]$ the horizontal one. 
The vertical dash-dotted segments $\beta=\pm\pi/2$ are the two halves of the 
boundary of the meridian $\mathfrak{m}$. The bigon between the two curves 
$r_0^+$ and $r_0^-$ is $\mathcal{B}_0^+$ 
(see Proposition \ref{prop-part-outside}). 
Compare to Figure 2 of \cite{DF8}.\label{fig-faces-in-IS}}
\end{figure}

\begin{proof}
Since isometric spheres are strictly convex, the ideal boundaries of the
ridges $r_0^+=\Isom_0^+\cap\Isom_0^-$ and $r_0^-=\Isom_0^+\cap\Isom_{-1}^-$ 
are Jordan curves on $\Isom_0^+$. We still denote them by $r_0^\pm$.
The interiors of these curves are respectively the connected components 
containing $p_{AB}$ and $p_{BA}$. By Lemma \ref{lem-triple-limit} in 
Section \ref{section-proof-triple-inside}, $r_0^+$ and $r_0^-$
have two intersection points, namely $p_{S^{-1}T}$ and $p_{ST^{-1}}$, and 
 their interiors are disjoint. As a consequence the common exterior 
of the two curves has two connected components, and the points 
$p_{S^{-1}T}$ and $p_{ST^{-1}}$ lie on the boundary of both.

To finish the proof, consider the involution $\iota_1$ defined
in the proof of Proposition \ref{prop-dec-ST}. (Note that since
$\alpha_1=0$ this involution conjugates $\Gamma^{\lim}$ to itself.)
In Heisenberg coordinates it is defined  
by $\iota_1 : [z,t]\longmapsto [-\overline{z},-t]$ and is clearly a
Cygan isometry. As in Proposition \ref{prop-dec-ST}, $\iota_1$
fixes $p_A$ and $p_B$ and it interchanges $p_{AB}$ and $p_{BA}$. 
Thus it conjugates $S$ to $T^{-1}$ and so it interchanges
$p_{ST^{-1}}$ and $p_{S^{-1}T}$ and it interchanges $p_{STS}$ and $p_{TST}$.
Moreover, since it is a Cygan isometry, $\iota_1$
preserves $\I_0^+$ and interchanges $\I_{-1}^-$ and $\I_0^-$ and thus 
it also exchanges the two curves $r_0^+$ and $r_0^-$. Again, 
since it is a Cygan isometry, it maps  interior to interior and exterior 
to exterior for both curves. As a consequence, the two connected components 
of the common exterior are either exchanged or both preserved. 

Now consider the 
point with Heisenberg coordinates $[i,0]$. It is fixed by $\iota_1$, 
and belongs to the common exterior of both $r_0^+$ and $r_0^-$. 
This implies that both connected components are preserved. Finally, since 
$p_{STS}\in\I_0^+\cap\I_{0}^-$ and $p_{TST}\in\I_0^+\cap\I_{-1}^-$ 
are exchanged by  $\iota_1$, these two points belong to the closure of the 
same connected component. As a consequence, one of the two connected 
components has $p_{ST^{-1}}$, $p_{S^{-1}T}$, $p_{STS}$ and $p_{TST}$ on its 
boundary. This is the quadrilateral. The other one has 
$p_{ST^{-1}}$ and $p_{S^{-1}T}$ on its boundary. This is the bigon.
\end{proof}

We now apply powers of $A$ to get a result about all the isometric
sphere intersections in the ideal boundary of $D$. 
Define $\mathcal{Q}_0^-=\varphi(\mathcal{Q}_0^+)$ and 
$\mathcal{B}_0^-=\varphi(\mathcal{B}_0^+)$. Then applying powers of
$A$ we define quadrilaterals $\mathcal{Q}_k^\pm=A^k(\mathcal{Q}_0^\pm)$, 
and bigons $\mathcal{B}_k^\pm=A^k(\mathcal{B}_0^\pm)$. 
The action of the Heisenberg translation $A$ and 
the glide reflection $\varphi$ are:

\begin{figure}[ht]
\begin{center}
\scalebox{0.35}{\includegraphics{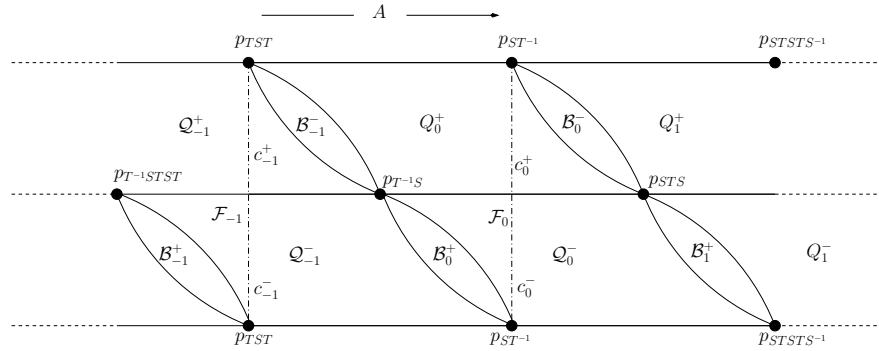}}
\end{center}
\caption{A combinatorial picture of $\partial D$\label{fig-combi}. 
The top and bottom lines are identified.}
\end{figure}

\begin{coro}
For the group $\Gamma^{\lim}$, the (ideal boundary of) the side $s_k^\pm$ is the union of  
the quadrilateral $\mathcal{Q}_k^\pm$ and the bigon $\mathcal{B}_k^\pm$.
The action of $A$ and $\varphi$ are as follows.
\begin{enumerate}
\item[(1)]  $A$ maps $\mathcal{Q}_k^\pm$ to $\mathcal{Q}_{k+1}^\pm$, and 
$\mathcal{B}_k^\pm$ to $\mathcal{B}_{k+1}^\pm$.
\item[(2)] $\varphi$ maps $\mathcal{Q}_k^+$ to $\mathcal{Q}_{k}^-$, 
$\mathcal{Q}_{k}^-$ to $\mathcal{Q}_{k+1}^+$, 
$\mathcal{B}_k^+$ to $\mathcal{B}_{k}^-$ and $\mathcal{B}_{k}^-$ to 
$\mathcal{B}_{k+1}^+$.
\end{enumerate}
\end{coro}

In order to understand the combinatorics of the sides of $D$, 
we describe the edges of the faces lying in $\Isom_0^+$. 
The three points $p_{S^{-1}T},\,p_{ST^{-1}},\,p_{STS}$ lie on the ridge 
$r_0^+=\Isom_0^+\cap\Isom_0^-$. Likewise, the points
$p_{ST^{-1}},\,p_{S^{-1}T},\,p_{TST}$ lie in the ridge 
$r_0^-=\Isom_0^+\cap\Isom_{-1}^-$. Indeed, these points divide (the ideal 
boundaries of) these ridges into three segments. We have listed the ideal
vertices in positive cyclic order (see Figure \ref{fig-faces-in-IS}).
Using the graph \eqref{eq-graph}, the action of the cycle transformations 
$\rho(s_0^+)=S$ and $\rho(r_0^-)=A^{-1}S=T^{-1}$ on these ideal vertices, 
and hence on the segments of the ridges, is:
\begin{displaymath} 
\xymatrix{ 
p_{S^{-1}T} \ar[rr]^-{S} && p_{ST^{-1}} \ar[rr]^-{S} && 
p_{STS} \ar[rr]^-{S} && p_{S^{-1}T}, \\
p_{ST^{-1}} \ar[rr]^-{A^{-1}S} && p_{S^{-1}T} \ar[rr]^-{A^{-1}S} &&
p_{TST} \ar[rr]^-{A^{-1}S} && p_{ST^{-1}}. \\
}
\end{displaymath}
Furthermore, $S$ maps $p_{TST}$ to $p_{STSTS^{-1}}$.

The quadrilateral $\mathcal{Q}_0^+$ has two edges 
$[p_{S^{-1}T},p_{TST}]\cup[p_{TST},p_{ST^{-1}}]$ in the ridge $r_0^-$ 
and two edges $[p_{ST^{-1}},p_{STS}]\cup[p_{STS},p_{S^{-1}T}]$ in the ridge $r_0^+$.
It is sent by $S$ to the quadrilateral $\mathcal{Q}_0^-$ with two
edges $[p_{ST^{-1}},p_{STSTS^{-1}}]\cup[p_{STSTS^{-1}},p_{STS}]$ in $r_1^-$ and two 
edges $[p_{STS},p_{S^{-1}T}]\cup[p_{S^{-1}T},p_{ST^{-1}}]$ in $r_0^+$. 
Similarly, the edges of the bigon $\mathcal{B}_0^+$ are the remaining
segments in $r_0^-$ and $r_0^+$, both with endpoints $p_{S^{-1}T}$ and 
$p_{ST^{-1}}$. It is sent by $S$ to the bigon $\mathcal{B}_0^-$ with vertices
$p_{ST^{-1}}$ and $p_{STS}$.

Applying powers of $A$ gives the other quadrilaterals and bigons.
As usual, the image under $A^k$ can be found by adding $k$ to each subscript
and conjugating each side pairing map and ridge cycle by $A^k$.
The combinatorics of $D$ is summarised on Figure \ref{fig-combi}.

\begin{lem}\label{lem-delta-phi}
The line $\Delta_{\varphi}$ given in \eqref{param-inv-phi} 
is contained in the complement of $D$. 
\end{lem}

\begin{proof}
As noted above, $A$ acts on $\Delta_{\varphi}$ as a translation through 
$\sqrt{3/2}$. We claim that the segment of $\Delta_\varphi$ 
with parameter $x\in[-\sqrt{3/8},\sqrt{3/8}]$ in contained in the
interior of $\I_0^+$. Applying powers of $A$ we see that each point
of $\Delta_\varphi$ is contained in $\I_k^+$ for some $k$. Hence the line
is in the complement of $D$. 

Consider $\delta_\varphi(x)\in\Delta_\varphi$ with 
$x^2\le 3/8$. The Cygan distance between $p_B$ and 
$\delta_\varphi(x)$ satisfies:
$$
d_{\rm Cyg}(p_B,\delta_\varphi(x))^4
=\Bigl|-x^2-5/32+ix\sqrt{5/8}\Bigr|^2
=x^4+15x^2/16+25/1-24\le 529/1024. 
$$
Since $d_{\rm Cyg}(p_B,\delta_\varphi(x))<1$ this means $\delta_\varphi(x)$ is
in the interior of $\I_0^+$ as claimed.
\end{proof}

The following result, which will be proved in 
Section \ref{section-proof-cylinder}, is crucial for proving
Theorem \ref{theo-CRWLC}.

\begin{prop}\label{prop-solid-cylinder}
There exists a homeomorphism $\Psi:\R^3\longrightarrow \partial\HdC-\{q_\infty\}$
mapping the exterior of $S^1\times \R$, that is 
$\bigl\{(x,y,z)\ :\ x^2+y^2\ge 1\bigr\}$, homeomorphically onto $D$ 
and so that $\Psi(x,y,z+1)=A\Psi(x,y,z)$, that is $\Psi$
is equivariant with respect to unit translation along the $z$ axis and $A$.
\end{prop}

As a consequence of Proposition \ref{prop-solid-cylinder}, $D$ admits an 
$A$ invariant 1-dimensional foliation, the leaves being the images of
radial lines $\bigl\{(r\cos(\theta_0),r\sin(\theta_0),z_0)\ :\ r\ge 1\bigr\}$
that foliate the exterior of $S^1\times\R$. Each of these leaves is
a curve connecting a point of $\partial D$ with $q_\infty$.
We can now prove Theorem \ref{theo-CRWLC}.

\begin{proof}[Proof of Theorem \ref{theo-CRWLC}.]
The union $\mathcal{Q}_0^+\cup\mathcal{B}_{0}^+\cup\mathcal{Q}_0^-\cup\mathcal{B}_0^-$ 
is a fundamental domain for the action of $A$ on the boundary cylinder 
$\partial D$. As the foliation obtained above is $A$-invariant, the cone 
to the point $q_\infty$ built over it via the foliation is a fundamental domain 
for the action of $A$ over $D$, and thus, it is a fundamental domain for the 
action of $\Gamma^{\lim}$ on the region of discontinuity 
$\Omega({\Gamma^{\lim}})$. 

This fundamental domain is the union of two pyramids $\mathcal{P}^+$ and 
$\mathcal{P^-}$, with respective bases  $\mathcal{Q}_0^+\cup\mathcal{B}_{0}^-$ 
and $\mathcal{Q}_0^-\cup\mathcal{B}_0^+$, and common vertex $q_\infty=p_{ST}$. 
The two pyramids share a common face, which is a triangle with vertices 
$p_{STS}$, $p_{T^{-1}S}$ and $p_{ST}$. Cutting and pasting, consider the union 
$\mathcal{P}^+\cup S^{-1}\bigl(\mathcal{P}^-\bigr)$. It is again a fundamental 
domain for $\Gamma^{\lim}$. The apex of $S^{-1}(\mathcal{P}^-)$ is
$S^{-1}(q_\infty)=p_B=p_{TS}$. The image under $S^{-1}$ of $\mathcal{Q}_0^-$ is 
$\mathcal{Q}_0^+$, and the bigon $\mathcal{B}_0^+$ is mapped by $S^{-1}$ 
to another bigon connecting $p_{T^{-1}S}$ to $p_{STS}$. Since 
$\mathcal{B}_0^-=S(\mathcal{B}_0^+)$, this new bigon is the
image of $\mathcal{B}_0^-$ under $S^{-2}=S$.

The resulting object is a is a polyhedron (a combinatorial picture is 
provided on Figure \ref{octahedron}), whose faces are triangles and bigons. 
\begin{figure}[ht]
\begin{center}
\scalebox{0.35}{\includegraphics{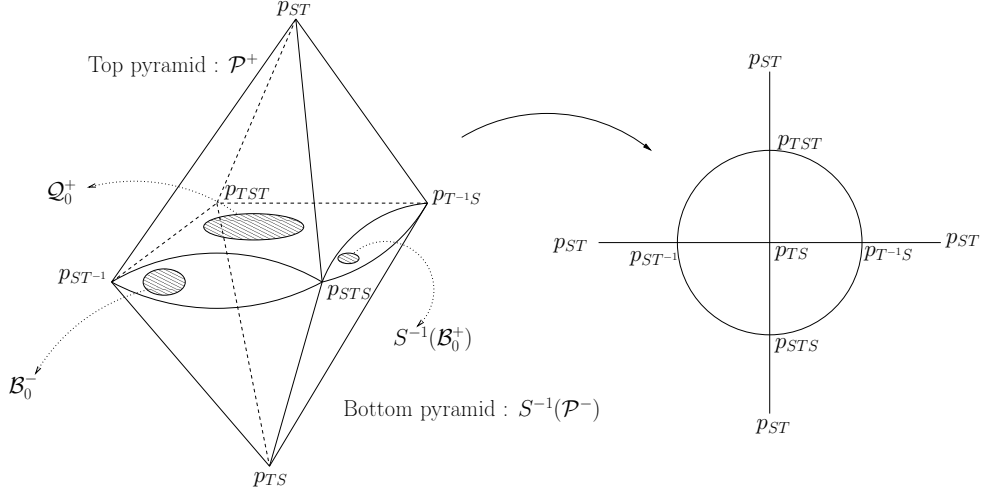}}
\end{center}
\caption{A combinatorial picture of the octahedron.\label{octahedron}}
\end{figure}
The faces of this octahedron are paired as follows.
\begin{eqnarray*}
TS & : & ( p_{TS}, p_{T^{-1}S}, p_{STS}) \longmapsto 
( p_{TS}, p_{TST}, p_{TS^{-1}} ), \\
ST & : & ( p_{ST}, p_{TST}, p_{T^{-1}S} ) \longmapsto 
(p_{ST}, p_{ST^{-1}}, p_{STS} ), \\
T & : & ( p_{ST}, p_{TST}, p_{ST^{-1}} ) \longmapsto 
( p_{TS}, p_{T^{-1}S}, p_{TST} ),\\
S & : &  (p_{TS},p_{ST^{-1}},p_{STS}) \longmapsto 
( p_{ST}, p_{STS}, p_{S^{-1}T} ), \\
S & : & (p_{ST^{-1}},p_{STS})\longmapsto (p_{STS},p_{S^{-1}T}) 
\end{eqnarray*}
The last line is the bigon identification between $\mathcal{B}_0^-$ 
and $S^{-1}(\mathcal{B}_0^+)$. As the triangle 
 $(p_{TS},p_{ST^{-1}},p_{STS})$
and the 
bigon $\mathcal{B}_0^-$ share a common edge and have the same face pairing 
they can be combined into a single triangle, as well as their images. 
Thus the last two lines may be combined into a single side with
side pairing map $S$. We therefore obtain a true combinatorial octahedron. 
The face identifications given above 
make the quotient manifold homeomorphic to the complement of 
the Whitehead link (compare for instance with Section 3.3 of \cite{Thu}).
\end{proof}

\section{Technicalities.}\label{section-technicalities}

\subsection{The triple intersections: proofs of Proposition \ref{prop-triple-inside} and Lemma \ref{lem-triple-limit}\label{section-proof-triple-inside}.}

In this section we first prove Proposition \ref{prop-triple-inside},
which states that the triple intersection must contain a point of 
$\partial\HdC$ and then we analyse the case of the limit group
$\Gamma^{\lim}$, giving a proof of Lemma \ref{lem-triple-limit}.
First recall that the isometric spheres $\Isom_0^-$ and $\Isom_{-1}^-$ are the 
unit Heisenberg spheres with centres given respectively in geographical 
coordinates  by  (see \ref{section-geog})
\begin{eqnarray*}
\bp_{AB} & = &S(\infty)
=g\bigl(-\alpha_1,-\alpha_1/2+\alpha_2,\sqrt{2\cos(\alpha_1)}\bigr)\\
\bp_{BA} & = &A^{-1}S(\infty)
=g\bigl(-\alpha_1,-\alpha_1/2-\alpha_2+\pi,
\sqrt{2\cos(\alpha_1)}\bigr).
\end{eqnarray*}
Consider the two functions of points $q=g(\alpha,\beta,w)\in \I_0^+$
defined by
\begin{eqnarray}
f^{[0]}_{\alpha_1,\alpha_2}(q)
& = & 2\cos^2(\alpha/2-\alpha_1/2)+\cos(\alpha-\alpha_1) 
\label{eq-first-intersection} \\
&& \quad 
-4wx_1\cos(\alpha/2-\alpha_1/2)\cos(\beta+\alpha_1/2-\alpha_2)+w^2x_1^2 ,
\notag \\
f^{[-1]}_{\alpha_1,\alpha_2}(q)
& = & 2\cos^2(\alpha/2-\alpha_1/2)+\cos(\alpha-\alpha_1)
\nonumber \\
&& \quad 
+4wx_1\cos(\alpha/2-\alpha_1/2)\cos(\beta+\alpha_1/2+\alpha_2)+w^2x_1^2.
\label{eq-second-intersection}
\end{eqnarray}
These functions characterise 
those points on $\I_0^+$ that belong to $\I_0^-$ and $\I_{-1}^-$.

\begin{lem}\label{lem-intersections}
A point $q$ on $\I_0^+$ lies on $\Isom_0^-$ (respectively in its interior 
or exterior) if and only if it satisfies $f^{[0]}_{\alpha_1,\alpha_2}(q)=0$  
(respectively is negative or is positive). Similarly, a point $q$ on $\I_0^+$ 
lies on $\Isom_{-1}^{-}$ (respectively in its interior or exterior) if and 
only if it satisfies $f^{[-1]}_{\alpha_1,\alpha_2}(q)=0$  (respectively is 
negative or is positive).
\end{lem}

\begin{proof}
A point $q\in\I_0^+$ lies on $\I_0^-$ (respectively in its interior or 
exterior) if and only if its Cygan distance from the centre of $\I_0^-$, 
which is the point $p_{AB}$, equals $1$ (respectively is less than $1$ or 
greater than $1$). Equivalently (see Section \ref{defisom}), the following 
quantity vanishes, is positive or negative respectively,
\begin{eqnarray*} 
\bigl|\langle\bq,\bp_{AB}\rangle\bigr|^2-1 & = & \Bigl| -e^{-i\alpha}+wx_1
e^{-i\alpha/2+i\beta-i\alpha_2} -e^{-i\alpha_1}\Bigr|^2-1 \\
& = & \Bigl| -2\cos(\alpha/2-\alpha_1/2)
+wx_1
e^{i\beta+i\alpha_1/2-i\alpha_2}\Bigr|^2-1 \\
& = & 4\cos^2(\alpha/2-\alpha_1/2)-1
-4\cos(\alpha/2-\alpha_1/2)wx_1
\cos(\beta+\alpha_1/2-\alpha/2)+ w^2x_1^2\\
& = & f^{[0]}_{\alpha_1,\alpha_2}(q).
\end{eqnarray*} 
On the last line we used $2\cos^2(\alpha/2-\alpha_1/2)=1+\cos(\alpha-\alpha_1)$.
This proves the first part of the Lemma and the second is obtained by a 
similar computation.
\end{proof}

\begin{coro}\label{coro-sum-positive}
For given $(\alpha_1,\alpha_2)$, if the sum 
$f^{[0]}_{\alpha_1,\alpha_2}+f^{[-1]}_{\alpha_1,\alpha_2}$ is positive for all $q$, 
then the triple intersection $\I_0^+\cap\I_0^-\cap\I_{-1}^-$ is empty.
\end{coro}
See Figure \ref{fig-faces-in-IS}. We can now prove 
Proposition \ref{prop-triple-inside}.

\begin{proof}[Proof of Proposition \ref{prop-triple-inside}.]
To prove the first part, note that a necessary condition for a 
point $q\in\I_0^+$ to be in the intersection 
$\I_0^-\cap\I_{-1}^-$ is that 
$f_{\alpha_1,\alpha_2}^{[0]}(q)-f_{\alpha_1,\alpha_2}^{[-1]}(q)=0$. This difference is:
\begin{eqnarray*}
f^{[0]}_{\alpha_1,\alpha_2}(q)
-f^{[-1]}_{\alpha_1,\alpha_2}(q) 
& = & -4wx_1\cos(\alpha/2-\alpha_1/2)\bigl(\cos(\beta+\alpha_1/2-\alpha_2)
+\cos(\beta+\alpha_1/2+\alpha_2)\bigr) \\
& = & -8wx_1\cos(\alpha/2-\alpha_1/2)\cos(\beta+\alpha_1/2)\cos(\alpha_2).
\end{eqnarray*}
Since $\alpha_1$ and $\alpha_2$ lie in $(-\pi/2,\pi/2)$ and
$\alpha\in[-\pi/2,\pi/2]$, the only solutions are
$\cos(\beta+\alpha_1/2)=0$ or $w=0$. Thus either 
$p=g(\alpha,\beta,w)$ lies on the meridian $\mathfrak{m}$,
or on the spine of $\Isom_0^+$, and hence on every meridian, in particular
on $\mathfrak{m}$ (compare with Proposition \ref{prop-geog}).

To prove the second part of Proposition \ref{prop-triple-inside}, 
assume that the triple intersection contains a point 
$q=g\bigl(\alpha,(\pi/2-\alpha_1/2),w\bigr)$ inside $\HdC$, that is such that  
$w^2<2\cos(\alpha)$, and
$$
f^{[0]}_{\alpha_1,\alpha_2}(q)+f^{[-1]}_{\alpha_1,\alpha_2}(q)=0.
$$
In view of Corollary \ref{coro-sum-positive}, we only need to prove that there exists a 
point on $\partial\mathfrak{m}$ where 
the above sum is non-positive, and use the intermediate value theorem. 
To do so, let $\tilde\alpha$ be defined by the 
condition $2\cos(\tilde\alpha)=w^2$ and such that $\tilde\alpha$ and $\alpha_1$ 
have opposite signs. Since $w^2<2\cos(\alpha)$, these conditions imply that 
$|\tilde\alpha|>|\alpha|$. We claim that the point 
$\tilde q=g\bigl(\tilde\alpha,(\pi-\alpha_1)/2,w\bigr)$ is satisfactory.
Indeed, the conditions on $\tilde\alpha$ give
$$
|\alpha-\alpha_1|\le |\alpha|+|\alpha_1|<|\tilde\alpha|+|\alpha_1|
=|\tilde\alpha-\alpha_1|
$$
where the last inequality follows from the fact that $\tilde\alpha$ and 
$\alpha_1$ have opposite signs. Therefore
\begin{equation}\label{eq-tilde}
\cos(\tilde\alpha/2-\alpha_1/2)<\cos(\alpha/2-\alpha_1/2).
\end{equation}
On the other hand, we have
\begin{eqnarray}\label{sum-f0-fm1}
\lefteqn{f^{[0]}_{\alpha_1,\alpha_2}(q)+f^{[-1]}_{\alpha_1,\alpha_2}(q)}\nonumber\\  
& = & 4\cos^2(\alpha/2-\alpha_1/2)+2\cos(\alpha-\alpha_1) 
 -8 w x_1\cos(\alpha/2-\alpha_1/2)\sin(\alpha_2)+2w^2x_1^2 \nonumber\\
& = & 8\cos^2(\alpha/2-\alpha_1/2)-2
-8 w x_1\cos(\alpha/2-\alpha_1/2)\sin(\alpha_2)+2w^2x_1^2.
\end{eqnarray}
We claim this is an increasing function of $\cos(\alpha/2-\alpha_1/2)$.
In order to see this, observe that its derivative with respect to this variable
is
$$
16\cos(\alpha/2-\alpha_1/2)-8wx_1\sin(\alpha_2)
> 16\cos(\alpha/2-\alpha_1/2)-16\sqrt{\cos(\alpha)\cos(\alpha_1)}\ge 0,
$$
where we used $x_1=\sqrt{2\cos(\alpha_1)}$, $w<\sqrt{2\cos(\alpha)}$ 
and $\sin(\alpha_2)\le 1$. Therefore, 
\begin{eqnarray*}
0 & = & f^{[0]}_{\alpha_1,\alpha_2}(q)+f^{[-1]}_{\alpha_1,\alpha_2}(q) \\
& = & 8\cos^2(\alpha/2-\alpha_1/2)-2
-8 w x_1\cos(\alpha/2-\alpha_1/2)\sin(\alpha_2)+2w^2x_1^2 \\
& > & 8\cos^2(\tilde\alpha/2-\alpha_1/2)-2
-8 w x_1\cos(\tilde\alpha/2-\alpha_1/2)\sin(\alpha_2)+2w^2x_1^2 \\
& = & f^{[0]}_{\alpha_1,\alpha_2}(\tilde q)
+f^{[-1]}_{\alpha_1,\alpha_2}(\tilde q).
\end{eqnarray*}
This proves our claim.
\end{proof}

We now prove Lemma \ref{lem-triple-limit} which completely describes 
the triple intersection at the limit point. 

\begin{proof}[Proof of Lemma \ref{lem-triple-limit}]
From the first part of Proposition \ref{prop-triple-inside} we see that
any point $q=g(\alpha,\beta,w)$ in $\I_0^+\cap\I_0^{-}\cap\I_{-1}^-$ must
lie on $\mathfrak{m}$, that is $\beta=(\pi-\alpha_1)/2$. For such points
it is enough to show that 
$f^{[0]}_{0,\alpha_2^{\lim}}(q)+f^{[-1]}_{0,\alpha_2^{\lim}}(q)=0$.
Substituting $\alpha_1=0$ and $\sin(\alpha_2)=\sqrt{5/8}$, this becomes: 
\begin{eqnarray*}
f^{[0]}_{0,\alpha_2^{\lim}}(q)+f^{[-1]}_{0,\alpha_2^{\lim}}(q)
& = & 4\cos^2(\alpha/2)+\cos(\alpha)-4\sqrt{5}w\cos(\alpha/2)+4w^2 \\
& = & \Bigl(2\cos(\alpha/2)-\sqrt{5}w\Bigr)^2+\bigl(2\cos(\alpha)-w^2\bigr).
\end{eqnarray*}
In order to vanish, both terms must be zero. Hence 
$w^2=2\cos(\alpha)$ and $2\cos(\alpha/2)=\sqrt{5}w=\sqrt{10\cos(\alpha)}$
(noting $w$ cannot be negative since $\alpha\in[-\pi/2,\pi/2]$).
This means $\alpha=\pm\arccos(1/4)$ and $w=\sqrt{2\cos(\alpha)}=1/\sqrt{2}$.
Therefore, the only points in $\I_0^+\cap\I_0^{-}\cap\I_{-1}^-$ have
geographical coordinates $g\bigl(\pm\arccos(1/4),\pi/2,1/\sqrt{2}\bigr)$.
Using \eqref{fixed-points-limit}, we see these points are
$p_{ST^{-1}}$ and $p_{S^{-1}T}$.
\end{proof}

\medskip


\subsection{The region $\mathcal{Z}$ is an open disc in the region $\mathcal{L}$: Proof of Proposition \ref{prop-Z-in-L}.\label{section-proof-ZinL}}
Consider the group $\Gamma_{\alpha_1,\alpha_2}$ and, as before, write
$x_1^4=4\cos^2(\alpha_1)$ and $x_2^4=4\cos^2(\alpha_2)$. 
Recall, from Proposition \ref{proptypecomm}, that $(\alpha_1,\alpha_2)$ 
is in ${\mathcal L}$ (respectively ${\mathcal P}$) if 
$\mathcal{G}(x_1^4,x_2^4)>0$ (respectively $=0$) where:
\begin{equation}\label{eqfunccomm}
\mathcal{G}(x,y)=x^2y^4-4x^2y^3+18xy^2-27.
\end{equation}
Recall this means $[A,B]$ is loxodromic (respectively parabolic). Also $(\alpha_1,\alpha_2)$ is in 
the rectangle $\mathcal{R}$ if and only if $(x_1^4,x_2^4)\in[3,4]\times[3/2,4]$. 
From Proposition \ref{prop-triple-inter}, the point 
$(\alpha_1,\alpha_2)\in\mathcal{R}$ is in $\mathcal{Z}$ 
(respectively $\partial\mathcal{Z}$) if $\mathcal{D}(x_1^4,x_2^4)>0$ 
(respectively $=0$) where: 
\begin{equation}\label{eq-def-D}
\mathcal{D}(x,y)=x^3y^3-9x^2y^2-27xy^2+81xy-27x-27.
\end{equation}

\begin{lem}\label{lem-R-in-L}
Suppose $(\alpha_1,\alpha_2)\in\mathcal{R}$. Then 
$(\alpha_1,\alpha_2)\in{\mathcal L}\cup\mathcal{P}$, that is the commutator $[A,B]$ is loxodromic or parabolic 
(see Section \ref{section-type-commutator}). Moreover, 
$(\alpha_1,\alpha_2)\in\mathcal{P}$ if and only if 
$(\alpha_1,\alpha_2)=(0,\pm\alpha_2^{\lim})$.
\end{lem}

\begin{proof}
We first claim that the function $\mathcal{G}(x,y)$ 
has no critical points in $(0,\infty)\times(0,\infty)$.
Indeed, the first partial derivatives of $\mathcal{G}(x,y)$ are
$$
\mathcal{G}_x(x,y)= 2y^2(xy^2-4xy+9), \quad \mathcal{G}_y(x,y)=4xy(xy^2-3xy+9).
$$
These are not simultaneously zero for any positive values of $x$ and $y$.
As a consequence, the minimum of $\mathcal{G}$ on 
$[3,4]\times[3/2,4]$ is attained on the boundary of this rectangle. 
We then have:
$$
\begin{array}{ll}
\mathcal{G}(x,3/2)=\dfrac{27}{16}\left(4-x\right)\left(5x-4\right), \quad &
\mathcal{G}(x,4)=9(32x-3), \\
\mathcal{G}(3,y)=9\left(y-1\right)\left(y^3-3y^2+3y+3\right), \quad &
\mathcal{G}(4,y)=(2y+1)(2y-3)^3.
\end{array}
$$
It is a simple exercise to check that under the assumptions that 
$3\le x\le 4$ and $3/2\le y\le 4$ all four of these terms are positive,
except for when $(x,y)=(4,3/2)$ in which case $\mathcal{G}(4,3/2)=0$. 
Then $(x_1^4,x_2^4)=(4,3/2)$ if and only if 
$(\alpha_1,\alpha_2)=(0,\pm\alpha_2^{\lim})$;
compare to Figure \ref{picture-peachandstone}.
\end{proof}

\begin{figure}
\begin{center}
 \scalebox{0.15}{\includegraphics{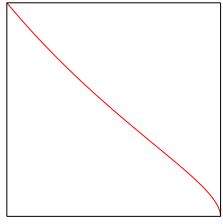}}
\caption{The null locus of $\mathcal{D}(x,y)$ in the rectangle 
$[3,4]\times[3/2,4]$.\label{zeroD}}
\end{center}
\end{figure}

\begin{lem}\label{lem-roots-partialD}
The region $\mathcal{Z}$ is an open topological disc in $\mathcal{R}$ 
symmetric about the axes and intersecting them
in the intervals $\{\alpha_2=0,\,-\pi/6<\alpha_1<\pi/6\}$ and
$\{\alpha_1=0,\,-\alpha_2^{\lim}<\alpha_2<\alpha_2^{\lim}\}$. 
Moreover, the only points of $\partial\mathcal{Z}$ that lie in
the boundary of $\mathcal{R}$ are $(\alpha_1,\alpha_2)=(0,\pm\alpha_2^{\lim})$
and $(\alpha_1,\alpha_2)=(\pm\pi/6,0)$.
\end{lem}

\begin{proof}
First we examine the values of $\mathcal{D}(x,y)$ on the boundary of 
$[3,4]\times[3/2,4]$: 
\begin{equation}\label{special-values-D}
\begin{array}{ll}
\mathcal{D}\left(x,3/2\right) =\frac{27}{8}(x-4)(x^2-2x+2), \quad & 
\mathcal{D}(x,4) = (x-3)(3+8x)^2, \\
\mathcal{D}(3,y) =  27(y-4)(y-1)^2,  \quad &
\mathcal{D}(4,y)  =  (16y-15)(2y-3)^2.
\end{array}
\end{equation}

We claim that, for any $y_0\in[3/2,4]$ the polynomial $\mathcal{D}(x,y_0)$
has exactly one root in $[3,4]$. Indeed, we have $\mathcal{D}(3,y_0)\le 0 \le \mathcal{D}(4,y_0)$ 
and thus $\mathcal{D}(x,y_0)$ has at least one such root. The $x$-derivative of $\mathcal{D}$ is 
$$
\partial_x\mathcal{D}(x,y)=3(x-3)y^2(xy+3y-6)+27(y-1)^3,
$$
which is positive when $x\in[3,4]$ and  $y\in[3/2,4]$. Thus $\mathcal{D}(x,y_0)$ is increasing, 
and the root is unique. 


Similarly, we claim that, for any $x_0\in[3,4]$, the polynomial 
$\mathcal{D}(x_0,y)$ has a unique root in $[3/2,4]$. It is clear from \eqref{special-values-D} 
when $x_0=4$ (there the root is $y=3/2$).
Now suppose $3\le x_0<4$. Arguing as before, 
we have $\mathcal{D}(x_0,3/2)< 0\le \mathcal{D}(x_0,4)$. However, it is 
not true that $\mathcal{D}(x_0,y)$ is a monotone function of $y$. 
The partial derivative of $\mathcal{D}(x,y)$ with respect to $y$ is
$$
\partial_y\mathcal{D}(x,y)=3x(x^2y^2-6xy-18y+27).
$$
Therefore, for a fixed $x_0\in[3,4)$ we have
$\partial_y\mathcal{D}(x_0,3/2)=27x_0^2(x_0-4)/4< 0$.
Since $\mathcal{D}(x_0,y)$ is a cubic with leading coefficient $x_0^3>0$,
such that both $\mathcal{D}(x_0,3/2)$ and $\partial_y\mathcal{D}(x_0,3/2)$
are negative we see that $\mathcal{D}(x_0,y)$ has exactly one zero 
in $(3/2,\infty)$. Since $\mathcal{D}(x_0,4)\ge 0$ this zero must lies
in $(3/2,4]$ as claimed.

Thus the zero-locus of $\mathcal{D}(x,y)$ in $[3,4]\times [3/2,4]$ is the graph 
of a continuous bijection connecting the two points $(3,4)$ and $(4,3/2)$. 
The polynomial $\mathcal{D}(x,y)$ is positive in the part of 
$[3,4]\times[3/2,4]$ above the zero-locus, that is containing the point 
$(x,y)=(4,4)$ (see Figure \ref{zeroD}). Likewise, it is negative in
the part below the zero locus, that is containing the point $(x,y)=(3,3/2)$.
Changing coordinates to $(\alpha_1,\alpha_2)$, we 
see that the zero locus of 
$\mathcal{D}\bigl(4\cos^2(\alpha_1),4\cos^2(\alpha_2)\bigr)$ 
in the rectangle $[0,\pi/6]\times[0,\alpha_2^{\lim}]$ is the graph
of a continuous bijection connecting the points 
$(\alpha_1,\alpha_2)=(\pi/6,0)$ and $(0,\alpha_2^{\lim})$. 
Moreover, $\mathcal{D}$ is positive on the part below this curve, 
in particular on the interval $\alpha_1=0$ and $0\le \alpha_2<\alpha_2^{\lim}$
and the interval $\alpha_2=0$ and $0\le\alpha_1<\pi/6$. 
The region $\mathcal{Z}$ is the union of the four copies 
of this region by the symmetries about the horizontal 
and vertical coordinate axes. It is clearly a disc and contains the
relevant parts of the axes. This completes the proof.
\end{proof}

Combining Lemmas \ref{lem-R-in-L} and \ref{lem-roots-partialD} proves
Proposition \ref{prop-Z-in-L}.

\subsection{Condition for no triple intersections: Proof of Proposition \ref{prop-triple-inter}.\label{section-no-triple}}

In this section we find a condition on $(\alpha_1,\alpha_2)$
that characterises the set $\mathcal{Z}$ where the triple intersection 
of isometric spheres $\I_0^+\cap\I_0^-\cap\I_{-1}^-$ is empty.

\begin{lem}\label{triple-int-poly}
The triple intersection $\I_0^+\cap\I_0^-\cap\I_{-1}^-$ is empty
if and only if $f_{\alpha_1,\alpha_2}(\alpha)>0$ for all 
$\alpha\in[-\pi/2,\pi/2]$ where
\begin{eqnarray}
f_{\alpha_1,\alpha_2}(\alpha)
& = & 4\cos^2(\alpha/2-\alpha_1/2)+2\cos(\alpha-\alpha_1) 
+8\cos(\alpha)\cos(\alpha_1) \label{eq-exteriors} \\
&& \quad -16\sqrt{\cos(\alpha)\cos(\alpha_1)}
\cos(\alpha/2-\alpha_1/2)\bigl|\sin(\alpha_2)\bigr|. \notag
\end{eqnarray}
\end{lem}

\begin{proof}
 By Corollary \ref{coro-sum-positive}, it is enough to show that  
$f^{[0]}_{\alpha_1,\alpha_2}+f^{[-1]}_{\alpha_1,\alpha_2}> 0$.  This sum is made explicit in 
\eqref{sum-f0-fm1}. In view of the second part of 
Proposition \eqref{prop-triple-inside}, we 
can restrict our attention to showing that
the triple intersection $\I_0^+\cap\I_0^-\cap\I_{-1}^-$ 
contains no points of $\partial\HdC$.
That is, we may assume $w=\pm\sqrt{2\cos(\alpha)}$. Using the first part of 
Proposition \ref{prop-triple-inside} we restrict our attention to points
$m$ in the meridian $\mathfrak{m}$ where $\beta=(\pi-\alpha_1)/2$.
The triple intersection is empty if and only if the sum 
$f^{[0]}_{\alpha_1,\alpha_2}(q)+f^{[-1]}_{\alpha_1,\alpha_2}(q)$ is positive for any 
value of $\alpha$, where 
$q=g\bigl(\alpha,(\pi-\alpha_1/2),\pm\sqrt{2\cos(\alpha)}\bigr)$.
When $w\sin(\alpha_2)$ is negative all terms in \eqref{sum-f0-fm1} are
positive. Therefore we may suppose 
$w\sin(\alpha_2)=\sqrt{2\cos(\alpha_1)}\bigl|\sin(\alpha_2)\bigr|\ge 0$.
Substituting these values in the expression for $f^{[0]}_{\alpha_1,\alpha_2}(q)+f^{[-1]}_{\alpha_1,\alpha_2}(q)$ 
given in  \eqref{sum-f0-fm1} gives the function
$f_{\alpha_1,\alpha_2}(\alpha)$ in \eqref{eq-exteriors}.
\end{proof}

We want to convert \eqref{eq-exteriors} into a polynomial expression 
in a function of $\alpha$. The numerical condition given in the statement of 
Proposition \ref{prop-triple-inter} will follow from the next lemma.

\begin{lem}\label{lem-tech-proof-theo-triple}
If $\alpha\in[-\pi/2,\pi/2]$ is a zero of $f_{\alpha_1,\alpha_2}$ 
then $T_\alpha=\tan(\alpha/2)\in[-1,1]$ is a root of the quartic polynomial 
$L_{\alpha_1,\alpha_2}(T)$, where
\begin{eqnarray}\label{poly-condition}
 L_{\alpha_1,\alpha_2}(T) 
& = & T^4\left(2x_1^4x_2^4-4x_1^2x_2^4+x_1^4+10x_1^2+1\right)
-8T^3\sin(\alpha_1)\left(x_1^2x_2^4-x_1^2-1\right)\nonumber\\
&  &-2T^2\left(2x_1^4x_2^4+3x_1^4-9\right)
+8T\sin(\alpha_1)\left(x_1^2x_2^4-x_1^2+1\right)\nonumber\\
&  &  + \left(2x_1^4x_2^4+4x_1^2x_2^4+x_1^4-10x_1^2+1\right)
\end{eqnarray}
\end{lem}

\begin{proof}
Squaring the two lines of \eqref{eq-exteriors} and using $\sqrt{2\cos(\alpha_1)}\bigl|\sin(\alpha_2)\bigr|\ge 0$, 
we see that the condition $f_{\alpha_1,\alpha_2}(\alpha)=0$ is equivalent to 
\begin{equation}\label{condi-f-0}
\Bigl(2\cos^2(\alpha/2-\alpha_1/2)+\cos(\alpha-\alpha_1) 
+4\cos(\alpha)\cos(\alpha_1)\Bigr)^2 
=64\cos(\alpha)\cos(\alpha_1)
\cos^2(\alpha/2-\alpha_1/2)\sin^2(\alpha_2).\nonumber\\
\end{equation}
After rearranging and expanding, we obtain the following polynomial equation 
in $\cos(\alpha)$ and $\sin(\alpha)$. 
\begin{eqnarray*}
 0 & = &4\left(8\cos^2(\alpha_1)\cos^2(\alpha_2)+2\cos^2(\alpha_1)-1\right)
\cos^2(\alpha)
+4\cos(\alpha_1)\bigl(8\cos^2(\alpha_2)-5\bigr)\cos(\alpha)
\nonumber\\
& & \quad +8\cos(\alpha_1)\sin(\alpha_1)
\left(4\cos^2(\alpha_2)-1\right)\cos(\alpha)\sin(\alpha)
+4\sin(\alpha_1)\sin(\alpha)-4\cos^2(\alpha_1)+5. 
\end{eqnarray*}
Substituting $\tan(\alpha/2)=T$, $2\cos(\alpha_1)=x_1^2$ and 
$2\cos(\alpha_2)=x_2^2$ into this equation gives $L_{\alpha_1,\alpha_2}(T)$.
\end{proof}

Before proving Proposition \ref{prop-triple-inter}, we analyse the situation 
on the axes $\alpha_1=0$ and $\alpha_2=0$.

\begin{lem}\label{lem-L-axes}
Let $L_{\alpha_1,\alpha_2}(T)$ be given by \eqref{poly-condition}.
\begin{enumerate}
\item When $\alpha_2=0$ and $-\pi/6<\alpha_1<\pi/6$ 
then $L_{\alpha_1,0}(T)$ has two real double roots $T_-$ and $T_+$
where $T_-< -1$ and $T_+> 1$, and no other roots.
\item When $\alpha_1=0$ and $0<\alpha_2<\alpha_2^{\lim}$ 
or $-\alpha_2^{\lim}<\alpha_2<0$ the polynomial
$L_{0,\alpha_2}(T)$ has no real roots.
\end{enumerate}
\end{lem}

\begin{proof}
First, substituting $\alpha_2=0$ in \eqref{poly-condition} we find $L_{(\alpha_1,0)}=M_{\alpha_1}(T)^2$, where
$$
M_{\alpha_1}(T)=T^2(3x_1^2-1)-4T\sin(\alpha_1)-(3x_1^2+1).
$$
The condition on $\alpha_1$ guarantees that
$3x_1^2-1>0$ and so as $T$ tends to $\pm\infty$ so
$M_{\alpha_1}(T)$ tends to $+\infty$. On the other hand,
$$
M_{\alpha_1}(-1)=4\sin(\alpha_1)-2< 0,\quad 
M_{\alpha_1}(1)=-4\sin(\alpha_1)-2< 0.
$$
Therefore $M_{\alpha_1}(T)$ has two real roots $T_-< -1$ and $T_+> 1$
as claimed. Since $M_{\alpha_1}(T)$ is quadratic, it cannot have any
more roots. In particular, it is negative for $-1\le T\le 1$.

Secondly, we substitute $\alpha_1=0$ in \eqref{poly-condition}, giving:
$$
L_{0,\alpha_2}(T)=\left(5T^2-\frac{8x_2^4+3}{5}\right)^2+
\frac{32}{25}(2x_2^4-3)(4-x_2^4).
$$
When $\alpha_2\in(-\alpha_2^{\lim},\alpha_2^{\lim})$ and $\alpha_2\neq 0$,
we have $x_2^4=4\cos^2(\alpha_2)\in(3/2,4)$. In particular, this means that
$(2x_2^4-3)(4-x_2^4)>0$ and so $L_{0,\alpha_2}(T)$ has no real roots,
proving the second part.
\end{proof}

We note that when $\alpha_1=\alpha_2=0$ then $L_{0,0}(T)$ has double
roots at $T=\pm\sqrt{7/5}$ and when $\alpha_1=0$ and 
$\alpha_2=\pm\alpha_2^{\lim}$ then $L_{0,\pm\alpha_2^{\lim}}(T)$ has double 
roots at $T=\pm\sqrt{3/5}$.

\begin{lem}\label{lem-L-in-Z}
If $(\alpha_1,\alpha_2)\in\mathcal{Z}$ then the polynomial
$L_{\alpha_1,\alpha_2}(T)$ has no roots $T$ in $[-1,1]$.
\end{lem}

\begin{proof}
We analyse the number, type (real or non-real) and location of roots 
of the polynomial $L_{\alpha_1,\alpha_2}(T)$ when 
$(\alpha_1,\alpha_2)\in\mathcal{R}$. 
As $L_{\alpha_1\alpha_2}(T)$ has real coefficients, whenever it has only simple 
roots, its root set is of one of the following types:
\begin{enumerate}
\item[(a)] two pairs of complex conjugate non real simple roots,
\item[(b)] a pair of non-real complex conjugate simple roots and two simple 
real roots,
\item[(c)] four simple real roots.
\end{enumerate}
But the set of roots of a polynomial is a continuous map (in bounded degree) 
for the Hausdorff distance on compact subsets of $\C$. In particular, the 
root set type of $L_{\alpha_1,\alpha_2}(T)$ is a continuous function of $\alpha_1$ 
and $\alpha_2$. This implies that it is not possible to pass from one of 
the above types to another without passing through a polynomial having a 
double root. 

We compute the discriminant $\Delta_{\alpha_1,\alpha_2}$
of $L_{\alpha_1,\alpha_2}(T)$ (a computer may be useful to do so):
\begin{eqnarray}\label{full-resultant}
\Delta_{\alpha_1,\alpha_2} & = & 
2^{16}x_1^4\bigl(x_1^4+1\bigr)^2
\bigl(2x_1^2(2-x_1^2)(4-x_2^4)+(3x_1^2-1)^2\bigr)\bigl(4-x_2^4\bigr)^2\cdot
\mathcal{D}\bigl(x_1^4,x_2^4\bigr)
\end{eqnarray}
where $\mathcal{D}(x,y)$ is as in Proposition \ref{prop-triple-inter}, and
$x_i=\sqrt{2\cos(\alpha_i)}$. The polynomial $L_{\alpha_1,\alpha_2}(T)$ 
has a multiple root in $\C$ if and only if $\Delta_{\alpha_1\alpha_2}=0$. 
Let us examine the different factors.
\begin{itemize}
 \item The first two factors $x_1^4$ and $(x_1^4+1)^2$ are positive when
$(\alpha_1,\alpha_2)\in(-\pi/2,\pi/2)^2$. 
 \item Note that $(2-x_1^2)(4-x_2^4)\ge 0$ and 
$(3x_1^2-1)^2>0$ when $\sqrt{3}\le x_1^2\le 2$ and $x_2^4\le 4$, and so the
third factor is positive.
\end{itemize}
Thus, the only factors of $\Delta_{\alpha_1,\alpha_2}$ that can vanish on 
$\mathcal{R}$ are $(4-x_2^4)^2=16\sin^4\alpha_2$ and $\mathcal{D}(x_1^4,x_2^4)$. 
In particular $L_{\alpha_1,\alpha_2}(T)$ has a multiple root in $\C$ if and only 
if one of these two factors vanishes. 
We saw in Proposition \ref{prop-Z-in-L} that the subset of $\mathcal{R}$
where $\mathcal{D}(x_1^4,x_2^4)>0$ is a topological disc $\mathcal{Z}$, 
symmetric about the $\alpha_1$ and $\alpha_2$ axes and intersecting them
in the intervals $\{\alpha_2=0,\,-\pi/6<\alpha_1<\pi/6\}$ and
$\{\alpha_1=0,\,-\alpha_2^{\lim}<\alpha_2<\alpha_2^{\lim}\}$. 
Therefore, the rectangle $\mathcal{R}$ contains two open discs on which
$\Delta_{\alpha_1,\alpha_2}>0$, namely 
$$
\mathcal{Z}^+=\{(\alpha_1,\alpha_2)\in\mathcal{Z}\ :\ \alpha_2>0\}, \quad 
\mathcal{Z}^-=\{(\alpha_1,\alpha_2)\in\mathcal{Z}\ :\ \alpha_2<0\}.
$$
These two sets each contain an open interval of the $\alpha_2$ axis. 
We saw in the second part of Lemma \ref{lem-L-axes} that on both these 
intervals $L_{\alpha_1,\alpha_2}(T)$ has no real roots, that is its roots are
of type (a). Therefore it has no real roots on all of $\mathcal{Z}^+$ 
and $\mathcal{Z}^-$.

The only points of $\mathcal{Z}$ yet to be considered are those in
the interval $\{\alpha_2=0,\,-\pi/6<\alpha_1<\pi/6\}$.
We saw in the first part of Lemma \ref{lem-L-axes} that 
for such points $L_{\alpha_1,\alpha_2}(T)$ has no roots with $-1\le T\le1$. 
This completes the proof of Proposition \ref{prop-triple-inter}.
\end{proof}

\subsection{Pairwise intersection: Proof of Proposition \ref{prop-pairwise-intersections}.\label{section-proof-pairwise}}
Proposition \ref{prop-pairwise-intersections} will follow from the next lemma.

\begin{lem}\label{lem-est-1}
If $0< x\le 4$ and $\mathcal{D}(x,y)\ge 0$ then $xy\ge 6$  
with equality if and only if $(x,y)=(4,3/2)$.
\end{lem}

\begin{proof}
Substituting $y=6/x$ in \eqref{eq-def-D} and simplifying, we find
$\mathcal{D}(x,6/x)=-27(x-4)(x-9)/x$.
When $0< x\le 4$ we see immediately that this is non-positive
and equals zero if and only if $x=4$. This means that 
$xy-6$ has a constant sign on the region where $\mathcal{D}(x,y)>0$.
Checking at $(x,y)=(4,4)$ we see that it is positive.
\end{proof}

\begin{proof}[Proof of Proposition \ref{prop-pairwise-intersections}]
To prove the disjointness of the given isometric spheres we calculate the
Cygan distance between their centres. Since all the isometric spheres
have radius $1$, if we can show their centres are a Cygan distance at least 
$2$ apart, then the spheres are disjoint. (Note that the Cygan distance
is not a path metric, so it may be the distance is less than $2$ but
the spheres are still disjoint. This will not be the case in our examples.)

The centre of $I_k^+$ is 
$A^k(p_B)=\bigl[kx_1x_2^2/\sqrt{2},kx_1^2x_2^2\sin(\alpha_2)\bigr]$; see
Proposition \ref{centre-translates-S-Sm}.  
We will show that $d_{\rm Cyg}\bigl(A^k(p_B),p_B\bigr)^4>16$ when $k^2\ge 4$
and $(\alpha_1,\alpha_2)\in\mathcal{R}$, that is 
$(x_1^4,x_2^4)\in[3,4]\times[3/2,4]$:
\begin{eqnarray*}
d_{\rm Cyg}\bigl(A^k(p_B),p_B\bigr)^4
& = & \frac{k^4x_1^4x_2^8+k^2x_1^4x_2^4(4-x_2^4)}{4} \\
& \ge & \frac{27 k^4}{16}.
\end{eqnarray*} 
This number is greater than $16$ when $k\ge 2$ or $k\le-2$ as claimed.

Again using Proposition \ref{centre-translates-S-Sm}, the centre of $\I_k^-$ is 
$A^k(p_{AB})=\bigl[(kx_1x_2^2+x_1e^{i\alpha_2})/\sqrt{2},-\sin(\alpha_1)\bigr]$.
We suppose $(x_1^4,x_2^4)\in[3,4]\times[3/2,4]$ satisfies $x_1^4x_2^4\ge 6$, 
which is valid for $(\alpha_1,\alpha_2)\in{\mathcal Z}$ by 
Lemma \ref{lem-est-1}. 
\begin{eqnarray*}
d_{\rm Cyg}\bigl(A^k(p_{AB}),p_B\bigr)^4
& = & \frac{\bigl(k(k+1)x_1^2x_2^4+x_1^2\bigr)^2+4-x_1^4}{4} \\
& = & 1 + \frac{k^2(k+1)^2x_1^4x_2^8+2k(k+1)x_1^4x_2^4}{4} \\
& \ge & \left(\frac{3k(k+1)}{2}+1\right)^2.
\end{eqnarray*}
This number is at least $16$ when $k\ge 1$ or $k\le-2$ as claimed.
Moreover, we have equality exactly when $k=1$ or $k=-2$ and
when $x_1^4x_2^4=6$ and $x_2^4=3/2$; that is when  
$(x_1^4,x_2^4)=(4,3/2)$.
\end{proof}

\subsection{$\partial_\infty D$ is a cylinder: Proof of Proposition \ref{prop-solid-cylinder}\label{section-proof-cylinder}.}
To prove Proposition \ref{prop-solid-cylinder}, we adopt the following 
strategy.
\begin{itemize}
 \item {\bf Step 1.} First, we intersect $D$ with a fundamental domain $D_A$ 
for the action of $A$ on the Heisenberg group. The domain $D_A$ is 
bounded by two parallel vertical planes $F_{-1}$ and $F_0$ that are boundaries 
of \textit{fans} in the sense of \cite{GP2}. These two fans are such that 
$A(F_{-1})=F_0$ (see Figure \ref{fig-vertiproj} for a view  of the situation 
in vertical projection). We analyse the intersections of $F_0$ and $F_{-1}$ 
with $D$, and show that they are topological circles, denoted by $c_{-1}$ and 
$c_0$ with $A(c_{-1})=c_0$.
\item {\bf Step 2.} Secondly, we consider the subset of the complement of 
$D$ which is contained in $D_A$, and prove that it is a 3-dimensional 
ball that intersects $F_{-1}$ and $F_0$ along topological discs (bounded by 
$c_{-1}$ and $c_{0}$). This proves that $D\cap D_A$ is the 
complement a solid tube in $D_A$, which is unknotted using
Lemma \ref{lem-delta-phi}. Finally, we prove that, gluing together copies 
by powers of $A$ of $D\cap D_A$, we indeed obtain the complement of a 
solid cylinder. 
\end{itemize}

We construct a fundamental domain $D_A$ for the cyclic group of Heisenberg
translations $\langle A\rangle$. The domain $D_A$ will be bounded by two fans, 
chosen to intersect as few bisectors as possible. The fan $F_0$
will pass through $p_{ST^{-1}}$ and will be tangent to both
$\I_1^+$ and $\I_{-1}^-$; compare Figure \ref{fig-vertiproj}. Similarly,
$F_{-1}=A^{-1}(F_0)$ will pass through $A^{-1}(p_{ST^{-1}})=p_{TST}$
and be tangent to both $\I_0^+$ and $\I_{-2}^-$.
We first give $F_0$ and $F_{-1}$ in terms of horospherical coordinates and 
then we give them in terms of their own geographical coordinates
(see \cite{GP2}). In horospherical coordinates they are:
\begin{eqnarray}
F_0 & = & \Bigl\{[x+iy,t]\ :\ 3x\sqrt{3}-y\sqrt{5}=\sqrt{2}/2\Bigr\}, 
\label{eq-F0} \\
F_{-1} & = & \Bigl\{[x+iy,t]\ :\ 3x\sqrt{3}-y\sqrt{5}=-4\sqrt{2}\Bigr\}.
\label{eq-F-1}
\end{eqnarray}
This leads to the definition of $D_A$:
\begin{equation}\label{eq-DA}
D_A = \Bigl\{[x+iy,t]\ :\ 
-4\sqrt{2}\le 3x\sqrt{3}-y\sqrt{5}\le \sqrt{2}/2\Bigr\}.
\end{equation}

We choose geographical coordinates $(\xi,\eta)$ on $F_0$:  the lines where 
$\xi$ is constant (respectively $\eta$ is constant) are boundaries of 
complex lines (respectively Lagrangian planes). These coordinates correspond 
to the double foliation of fans by real planes and complex lines, which 
is described in Section 5.2 of \cite{GP2}. The particular choice is made
so that the origin is the midpoint of the centres of $\I_0^+$ and $\I_0^-$.
Doing so gives the fan $F_0$ as the set of points $f(\xi,\eta)$:
$$
f(\xi,\eta)=\left\{
\left[ \frac{\sqrt{5}\xi+\sqrt{3}+3i\sqrt{3}\xi+i\sqrt{5}}{4\sqrt{2}},\
\eta-\xi/4\right]\ :\ \xi,\,\eta\in{\mathbb R}\right\}.
$$
The standard lift of $f(\xi,\eta)$ is given by 
$$
{\bf f}(\xi,\eta)=\left[\begin{matrix} 
-\xi^2-\sqrt{15}\xi/4-1/4+i\eta-i\xi/4 \\
\sqrt{5}\xi/4+\sqrt{3}/4+3i\sqrt{3}\xi/4+i\sqrt{5}/4 \\
1 \end{matrix}\right].
$$

Using the convexity of Cygan spheres, we see that their intersection with 
$F_0$ (or $F_{-1}$) is one of: empty, a point or a topological 
circle. For the particular fans and isometric spheres of interest to
us, the possible intersections are summarised in the following result:

\begin{prop}\label{prop-inter-Fi-F}
The intersections of the fans $F_{-1}$ and $F_0$ with the isometric spheres 
$\I_k^\pm$ are empty, except for those indicated in the following table.
$$
\begin{array}{|c|c|c|c|c|c|c|c|c|}
\hline
\bigcap & \I_{-2}^- & \I_{-2}^+ & \I_{-1}^- 
& \I_{-1}^+ & \I_0^- & \I_0^+ & \I_{1}^- & \I_1^+ \\
\hline
F_{0} & \emptyset & \emptyset  & \lbrace p_{ST^{-1}} \rbrace 
& \emptyset &  \hbox{a circle} & \hbox{a circle} & \emptyset 
& \lbrace p_{ST^{-1}}\rbrace\\
F_{-1} & \lbrace{ p_{TST}\rbrace} & \emptyset & \hbox{a circle}  
& \hbox{a circle} & \emptyset & \lbrace{ p_{TST}\rbrace} 
& \emptyset & \emptyset\\
\hline
\end{array}
$$
Moreover, the point $p_{S^{-1}T}$ belongs to the interior of $D_A$. 
The parabolic fixed points $A^k(p_{ST^{-1}})$ lie outside $D_A$ for all 
$k\ge 1$ and $k\le -1$; parabolic fixed points
$A^k(p_{S^{-1}T})$ lie outside $D_A$ for all $k\neq 0$.
\end{prop}

A direct consequence of this proposition is that the only point in the 
closure of the quadrilateral $\mathcal{Q}_{-1}^-$ and the bigon $\mathcal{B}_{-1}^-$ that lie on 
$F_0$ is their vertex $p_{TST}$.

\begin{proof}
The part about intersections of fans and isometric spheres is proved easily 
by projecting vertically onto $\C$, as in the proof of 
Proposition \ref{prop-pairwise-intersections} (see Figure \ref{fig-vertiproj}). 
Note that as isometric spheres are strictly convex, their intersections with 
a plane is either empty or a point or a topological circle.
The part about the  parabolic fixed points is a direct verification using 
\eqref{eq-F0} as well as \eqref{fixed-points-limit}. 
\end{proof}

We need to be slightly more precise about the intersection of $F_0$ 
with $\I_0^+$ and $\I_{0}^-$.

\begin{prop}\label{prop-inter-F0-I0}
The intersection of $F_0$ with $\I_0^+\cup\I_{0}^-$ (and thus with $\partial D$) 
is a topological circle $c_0$, which is the union of two topological segments 
$c_0^+$ and $c_0^-$, where the segment $c_0^\pm$ is the part of 
$F_0\cap\I_0^\pm$ that is outside $\I_0^\mp$.  
The two segments $c_0^+$ and $c_0^-$ have the same endpoints: one of 
them is $p_{ST^{-1}}$, and we will denote the other by $q_0$. Moreover, 
the point $q_0$ lies on the segment $[p_{STS},p_{S^{-1}T}]$ of $\I_0^+\cap\I_0^-$.
\end{prop}

The point $q_0$ appears in Figures \ref{c0plus}, \ref{interF0} and  \ref{FinDA}.

\begin{proof}
The point $f(\xi,\eta)$ of the fan $F_0$ lies of $\Isom_0^+$ whenever
$1=\bigl|\langle{\bf f}(\xi,\eta),{\bf p}_{B}\rangle\bigr|$ and on $\Isom_0^-$ 
whenever $1=\bigl|\langle{\bf f}(\xi,\eta),{\bf p}_{AB}\rangle\bigr|$.
We first find all points on $F_0\cap\Isom_0^+\cap\Isom_0^-$. These correspond
to simultaneous solutions to:
\begin{equation}\label{eq-fan-ridge}
1=\bigl|\langle{\bf f}(\xi,\eta),{\bf p}_{B}\rangle\bigr|
=\bigl|\langle{\bf f}(\xi,\eta),{\bf p}_{AB}\rangle\bigr|
\end{equation}

Computing these products and rearranging, we obtain
\begin{eqnarray*}
\bigl|\langle{\bf f}(\xi,\eta),{\bf p}_{B}\rangle\bigr|^2 
& = & (\xi^2+1/4)^2+\xi^2+\eta^2
+\xi\Bigl(\sqrt{15}\xi^2+\sqrt{15}/4-\eta\Bigr)/2, \\
\bigl|\langle{\bf f}(\xi,\eta),{\bf p}_{AB}\rangle\bigr|^2
& = & (\xi^2+1/4)^2+\xi^2+\eta^2
-\xi\Bigl(\sqrt{15}\xi^2+\sqrt{15}/4-\eta\Bigr)/2.
\end{eqnarray*}
Subtracting, we see that solutions to \eqref{eq-fan-ridge} 
must either have $\xi=0$ or $\eta=\sqrt{15}(\xi^2+1/4)$.
Substituting these solutions into
$1=\bigl|\langle{\bf f}(\xi,\eta),{\bf p}_{B}\rangle\bigr|^2$,
we see first that $\xi=0$ implies $1 = \eta^2+1/16$; and secondly that 
$\eta=\sqrt{15}(\xi^2+1/4)$ implies
$$
1=(\xi^2+1/4)^2+\xi^2+15(\xi^2+1/4)^2=(4\xi^2+1)^2+\xi^2.
$$
Clearly the only solution to this equation is $\xi=0$. 
So both cases lead to the solutions $(\xi,\eta)=(0,\pm\sqrt{15}/4)$.
Thus the only points satisfying \eqref{eq-fan-ridge}, that is the points in 
$F_0\cap\Isom_0^+\cap\Isom_0^-$, are
$$
f(0,\sqrt{15}/4)=\left[\frac{\sqrt{3}+i\sqrt{5}}{4\sqrt{2}},\ 
\frac{\sqrt{15}}{4}\right] \mbox{ and }
f(0,-\sqrt{15}/4)=\left[\frac{\sqrt{3}+i\sqrt{5}}{4\sqrt{2}},\ 
\frac{-\sqrt{15}}{4}\right].
$$
Note that the first of these points is $p_{ST^{-1}}$. We call the
other point $q_0$.

These two points divide $F_0\cap\Isom_0^+$ and $F_0\cap\Isom_0^-$ into
two arcs. It remains to decide which of these arcs is outside the
other isometric sphere.
Clearly $\bigl|\langle{\bf f}(\xi,\eta),{\bf p}_{B}\rangle\bigr|>
\bigl|\langle{\bf f}(\xi,\eta),{\bf p}_{AB}\rangle\bigr|$ if and only
if $\xi\bigl(\sqrt{15}\xi^2+\sqrt{15}/4-\eta\bigr)>0$. Close to 
$\eta=-\sqrt{15}/4$ we see this quantity changes sign only when $\xi$ does. 
This means that if $f(\xi,\eta)\in\Isom_0^-$ with $\xi>0$ then 
$f(\xi,\eta)$ is in the exterior of $\Isom_0^+$. Similarly, if 
$f(\xi,\eta)\in\Isom_0^+$ with $\xi<0$ then $f(\xi,\eta)$ 
is in the exterior of $\Isom_0^-$.
In other words, $c_0^+$ is the segment of $F_0\cap\Isom_0^+$ where $\xi<0$
and $c_0^-$ is the segment of $F_0\cap\Isom_0^-$ where $\xi>0$.

Finally, consider the involution $I_2=SI_1$ in ${\rm PU}(2,1)$ from the 
proof of Proposition \ref{prop-dec-ST}. (Note that since $\alpha_1=0$, this
involution conjugates $\Gamma^{\lim}$ to itself.) The involution $I_2$
preserves $F_0$, acting on it by sending $f(\xi,\eta)$ to $f(-\xi,\eta)$, 
and hence interchanging the components of its complement. 
In Heisenberg coordinates $I_2$ is given by 
\begin{equation}\label{invol-I1}
I_2:\Bigl[x+iy,\,t\Bigr] \longleftrightarrow
\Bigl[-x-iy+\sqrt{3/8}+i\sqrt{5/8},\,t-\sqrt{5/2}\,x+\sqrt{3/2}\,y\Bigr].
\end{equation}
As $I_2$ is elliptic and fixes the point $q_\infty$, it is a Cygan 
isometry (see Section \ref{defisom}). 
Since it interchanges $p_B$ and $p_{AB}$, it also interchanges 
$\I_0^+$ and $\I_0^-$. Hence their intersection is preserved setwise.
The involution $I_2$ also interchanges $p_{S^{-1}T}$ and $p_{STS}$ contained
in $\I_0^+\cap\I_0^-$ (but not on $F_0$). Therefore, these two points 
lie in different components of the complement of $F_0$. Hence there
must be a point of $F_0$ on the segment $[p_{S^{-1}T},p_{STS}]$. This point 
cannot be $p_{ST^{-1}}$, and so must be $q_0$ (see Figure \ref{c0plus}).
\end{proof}

\begin{figure}
\begin{center}
\scalebox{0.65}{\includegraphics{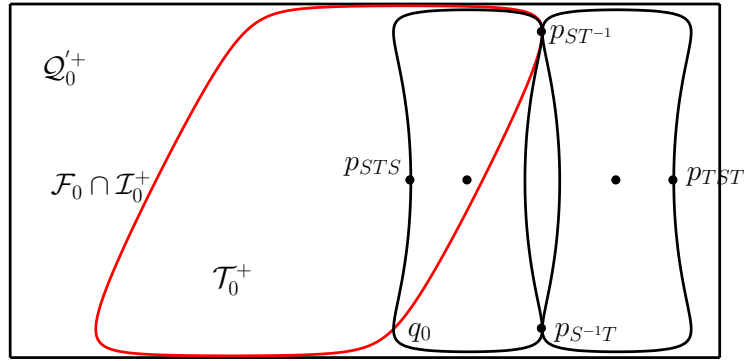}}
\end{center}
\caption{The intersection of $F_0$ with $\I_0^+$ drawn on $\I_0^+$, in 
geographical coordinates.\label{c0plus}}
\end{figure}


Let $D^c$ denote closure of the complement of $D$ in $\partial\HdC-\{q_\infty\}$.

\begin{prop}\label{prop-ball}
The closure of the intersection $D^c\cap D_A$ is a solid tube homeomorphic 
to a 3-ball. 
\end{prop}

\begin{proof}
We describe the combinatorial cell structure of $D^c\cap D_A$; see
Figure \ref{FinDA}. 
Using Proposition \ref{prop-inter-F0-I0}, it is clear $D^c$ intersects $F_0$ 
in a topological disc whose boundary circle is made up of two edges,
$c_0^\pm$ and two vertices $p_{ST^{-1}}$ and $q_0$. Combinatorially, this is
a bigon. Applying $A^{-1}$ we see $D^c$ intersects $F_{-1}$ in a bigon with 
boundary made up of edges $c_{-1}^\pm$ and two vertices $p_{TST}$ and $q_{-1}$. 

Moreover, Proposition \ref{prop-inter-F0-I0} immediately implies 
that $c_0$ cuts $\mathcal{Q}_0^\pm$ into a
quadrilateral and a triangle, which we denote by $\mathcal{{Q'}}_0^\pm$ and
$\mathcal{T}_0^\pm$. Since $D_A$ contains $p_{S^{-1}T}$ and
$p_{TST}$, we see that $D_A$ contains $\mathcal{{Q'}}_0^+$ and
$\mathcal{T}_0^-$. These have vertex sets 
$\{p_{ST^{-1}},\,p_{TST},\,p_{S^{-1}T},\,q_0\}$ and
$\{p_{S^{-1}T},\,p_{ST^{-1}},\,q_0\}$ respectively.
Applying $A^{-1}$ we see that $c_{-1}$ cuts $\mathcal{Q}_{-1}^\pm$ into
a quadrilateral, denoted $\mathcal{{Q'}}_{-1}^\pm$, and a 
triangle, denoted $\mathcal{T}_{-1}^\pm$. Of these the quadrilateral 
$\mathcal{{Q'}}_{-1}^-$ and the triangle $\mathcal{T}_{-1}^+$ lie in $D_A$. 
Finally, the bigons $\mathcal{B}_0^+$ and $\mathcal{B}_{-1}^-$ also lie in
$D_A$.

In summary, the boundary of $D^c\cap D_A$ has a combinatorial cell structure
with five vertices $\{p_{ST^{-1}},\,p_{S^{-1}T},\,p_{TST},\,q_0,\,q_{-1}\}$
and eight faces. 
$$
\{\mathcal{Q'}_0^+,\,\mathcal{Q'}_{-1}^-,\,\mathcal{T}_0^-,\,
\mathcal{T}_{-1}^+,\,\mathcal{B}_0^+,\,\mathcal{B}_{-1}^-,\,
F_0\cap D^c,\,F_{-1}\cap D^c\}.
$$
These are respectively two quadrilaterals, two triangles and four bigons. 
Therefore, in total the cell structure has 
$(2\times 4+2\times 3+4\times 2)/2=11$ edges.
Therefore the Euler characteristic of $\partial(D^c\cap D_A)$ is
$$
\chi\bigl(\partial(D^c\cap D_A)\bigr)=5-11+8=2.
$$
Hence $\partial(D^c\cap D_A)$ is indeed a sphere. This means
$D^c\cap D_A$ is a ball as claimed.
\end{proof}

\begin{figure}
     \begin{minipage}[c]{.46\linewidth}
      \scalebox{0.5}{\includegraphics{interF0bis.eps}}
\caption{ The intersection of $F_0$ with $\I_0^+\cap \I_0^-$. 
The disc $\mathcal{D}_0$ is the interior of  
$c_0=c_0^+\cap c_0^-$. The two segments $c_0^+$ and $c_0^-$ are 
the thicker parts of $F_0\cap\I_0^+$ and $F_0\cap\I_0^-$. \label{interF0}}
   \end{minipage} \hfill
\begin{minipage}[c]{.46\linewidth}
      \scalebox{0.5}{\includegraphics{FinDA.eps}}
\caption{A combinatorial picture of the intersection of $\partial D$ with 
$D_A$ \label{FinDA}. The top and bottom lines are identified. 
The curve $c_0$ corresponds to the right hand side of the figure.}
   \end{minipage}
\end{figure}

\begin{rem}\label{rem-simple-combi}
The combinatorial structure described on Figure \ref{FinDA} is quite simple. 
However, the geometric realisation of this structure is much more intricate. 
As an example, there are fans $F$ parallel to $F_0$ and $F_{-1}$ whose
intersection with $D^c$ is disconnected. This means that the 
foliation described right after Proposition \ref{prop-solid-cylinder} that is 
used in the proof of Theorem \ref{theo-CRWLC} is actually quite ``distorted''.
\end{rem}

We are now ready to prove Proposition \ref{prop-solid-cylinder}.

\begin{prop}\label{prop-finite-solid cylinder}
There is a homeomorphism $\Psi_A:\R^2\times[0,1]\longrightarrow D_A$
that satisfies $\Psi_A(x,y,1)=A\Psi_A(x,y,0)$ and 
so that $\Psi_A$  restricts to a homeomorphism from the exterior of 
$S^1\times [0,1]$, that is 
$\bigl\{(x,y,z)\ :\ x^2+y^2\ge 1,\, 0\le z\le 1\bigr\}$, to $D\cap D_A$. 
\end{prop}

\begin{proof}
We have shown Proposition \ref{prop-ball} that $D^c\cap D_A$ is a
solid tube homeomorphic to a $3$-ball and (using  
Proposition \ref{prop-inter-F0-I0}) that $D^c$ intersects $\partial D_A$
in two discs, one in $F_0$ bounded by $c_0$ and the other in $F_{-1}$
bounded by $c_{-1}$. This means we can
construct a homeomorphism $\Psi_A^c$ from the solid cylinder 
$\bigl\{(x,y,z)\ :\ x^2+y^2\le 1,\,0\le z\le 1\bigr\}$ to $D^c\cap D_A$
so that the restriction of $\Psi_A^c$ to $S^1\times[0,1]$ is a
homeomorphism to $\partial D\cap D_A$, with 
$\Psi_A^c: S^1\times \{0\}\longmapsto c_{-1}$ and
$\Psi_A^c: S^1\times \{1\}\longmapsto c_0$. 
Adjusting $\Psi_A^c$ if necessary, we can assume that 
$\Psi_A^c(x,y,1)=A\Psi_A^c(x,y,0)$.

Furthermore, in Lemma \ref{lem-delta-phi}, we showed that $D^c$ contains
the invariant line $\Delta_\varphi$ of $\varphi$. This means that the
cylinder $D^c\cap D_A$ is a thickening of $\Delta_\varphi\cap D_A$ and so, in
particular, it cannot be knotted. Hence we can extend $\Psi_A^c$ to
homeomorphism $\Psi_A:\R^2\times[0,1]\longrightarrow D_A$ satisfying
$\Psi_A(x,y,1)=A\Psi_A(x,y,0)$. In particular, $\Psi$ maps 
$\bigl\{(x,y,z)\ :\ x^2+y^2\ge 1,\, 0\le z\le 1\bigr\}$
homeomorphically to $D\cap D_A$ as claimed. 
\end{proof}

Finally, we prove Proposition \ref{prop-solid-cylinder} by extending
$\Psi_A:\R^2\times[0,1]\longrightarrow D_A$ equivariantly to a 
homeomorphism $\Psi:\R^3\longmapsto \partial\HdC-\{q_\infty\}$. 
That is, if $(x,y,z+k)\in\R^3$ where
$k\in\Z$ and $z\in[0,1]$, we define $\Psi(x,y,z+k)=A^k(x,y,z)$.
Since $\Psi(x,y,1)=A\Psi(x,y,0)$ there is no ambiguity at the boundary.

\bibliographystyle{plain}

\end{document}